\documentclass{amsart}
\usepackage{amsmath,amssymb,graphicx,setspace,verbatim,url}
\usepackage{color}
\usepackage{hyperref}
\usepackage{multicol}
\usepackage{enumitem}
\usepackage{float}
\usepackage{tikz}
\usepackage{tikz-cd}
\usepackage{soul}
\usepackage{multirow}

\usepackage[margin=2.5cm]{geometry}

\input xy
\xyoption{all}
\newtheorem{prop}{Proposition}[section]
\newtheorem{coro}[prop]{Corollary}
\newtheorem{thm}[prop]{Theorem}
\newtheorem{lemma}[prop]{Lemma}
\newtheorem{conjecture}[prop]{Conjecture}

\newtheorem{definition}[prop]{Definition}
\newtheorem{notation}[prop]{Notation}
\newtheorem{construction}[prop]{Construction}

\newcommand{\PSL}{\mathrm{PSL}} 
\newcommand{\PGL}{\mathrm{PGL}}

\newcommand{\AGL}{\mathrm{AGL}}

\newcommand{\Dy}{\mathrm{D}} 
\newcommand{\Cyc}{\mathrm{C}} 
\newcommand{\Alt}{\mathrm{A}} 
\newcommand{\Sym}{\mathrm{S}}

\newcommand{\lcs}{\mathrm{lcs}}

\newenvironment{purple}{\relax\color{purple}}{\hspace*{.5ex}\relax}
\newcommand{\bcp}{\begin{purple}}
\newcommand{\ecp}{\end{purple}}

\begin{document}
\title{Regular polytopes of rank $n/2$ for transitive groups of degree $n$}

\author[M. E. Fernandes]{Maria Elisa Fernandes}
\address{Maria Elisa Fernandes,
Center for Research and Development in Mathematics and Applications, Department of Mathematics, University of Aveiro, Portugal
}
\email{maria.elisa@ua.pt}

\author[Claudio Piedade]{Claudio Alexandre Piedade}
\address{Claudio Alexandre Piedade, Centro de Matemática da Universidade do Porto, Universidade do Porto, Portugal
}
\email{claudio.piedade@fc.up.pt}

\date{}
\maketitle

\begin{abstract}
Previous research established that the maximal rank of the abstract regular polytopes whose automorphism group is a transitive proper subgroup of $\Sym_n$ is $n/2 + 1$. Up to isomorphism and duality, when $n\geq 12$, there are only two polytopes attaining this rank and they occur when $n/2$ is odd, and hence have even rank.

 In this paper, we investigate the case where the rank is equal to $n/2$ ($n\geq 14$). Our analysis suggests that reducing the rank by one results in a substantial increase in the number of regular polytopes. 
 
\end{abstract}


\noindent \textbf{2000 Math Subj. Class:} 52B11, 20B35, 20B30, 05C25.

\section{Introduction}

Abstract polytopes are combinatorial objects that describe standard regular polytopes using their a face-lattice \cite{arp}. The rank of an abstract polytope is the length of a maximal chain minus two (there exists a maximal face and a minimal face which are not counted for the rank).
An abstract polytope is regular when its group of automorphisms acts regularly on the maximal chains (usually called flags).
A notable feature of these structures lies in their one to one correspondence with their automorphism groups,  which are string C-groups.
These algebraic structures are defined not only by the group itself but also by a specified set of involutory generators, the size of which determines the rank.

The maximal rank of an abstract regular polytope whose automorphism group has degree $n$ is $n-1$.
 For $n\geq 5$, the simplex is the only polytope achieving the maximal rank \cite{fl,corr}. One permutation representation of the group of automorphism of the simplex on $n$ points is 
$\langle (1,2),\, (2,3),\, \ldots, (n-1,n)\rangle$, the polytope with  Schl\"{a}fli  symbol $\{3,\ldots, 3\}$ corresponding to the standard Coxeter group of type $\Alt_{n-1}$.
Indeed $n-1$ is the maximal size of an independent set in $\Sym_n$,  and $\Sym_{n}$ is the unique group of degree $n$ that admits an independent set of generators of size $n-1 $\cite{2000W}. The classification of abstract regular polytopes of ranks $n-k$ for groups of degree $n\geq 2k+3$, was also accomplished for $k\in\{1,2,3,4\}$ in \cite{extension}. 
The automorphism group of all these high-rank abstract polytopes is $\Sym_n$. The rank of a regular polytope with automorphism group being an alternating group is considerably lower than $n-1$. In fact, the highest rank for the alternating group $\Alt_n$ is $\lfloor (n-1)/2\rfloor$ for $n\geq 12$  \cite{2017CFLM}.

The analyses of the \emph{C-rank} (maximal size of a set of generators of a string C-group) of other transitive permutation groups of degree $n$, different from $\Alt_n$ and $\Sym_n$, started in \cite{2016CFLM}. The C-rank of a primitive proper subgroup of $\Sym_n$ is at most $n/2+1$. Indeed, only few of these groups have C-rank greater than or equal $n/2$ (see Table 2  of \cite{2016CFLM}).

In this paper we give the classification of the transitive proper subgroups of $\Sym_n$ with C-rank greater than or equal to $n/2$ ($n\geq 14$),  refining the classification given in \cite{2016CFLM}. 
The group generated by the set $S=\{\rho_0,\ldots, \rho_{r-1}\}$ of permutations of degree $n$ will be represented by a graph, which is a slight modification of the Schreier coset digraph, 
where each pair of opposite arcs with label $\rho_i$ are replaced by a single edge with label $i$. For the maximal possible rank $r=n/2+1$, we noticed that in Theorem 1.2 (b) of \cite{2016CFLM} there is a possibility for the set of generators that is missing, namely the one having the permutation representation graph (2) of Table~\ref{RC2}.

The main result of the paper is the following.

\begin{thm}\label{main}
Let $n/2\geq 7$ and $G$ be a transitive proper subgroup of $\Sym_n$.  If $G$ is the automorphism group of an abstract regular polytope of rank $r\geq n/2$, then $G$ is one of the following groups.
$$\textrm{C}_2\times \Sym_{n/2},\,(\textrm{C}_2)^{n/2-1}: \Sym_{n/2},\, (\textrm{C}_2)^{n/2}: \Sym_{n/2}\mbox{ or  }(\Sym_{n/2})^2:\textrm{C}_2.$$ 
More precisely, all the possibilities for the set generators $S=\{\rho_0,\ldots, \rho_{r-1}\}$ of $G$ are, up to duality, included among the graphs of the tables displayed in Section~\ref{tables+proof}. \end{thm}

\begin{conjecture}
The graphs given in Tables~(\ref{RC2})--(\ref{m=2table}) are permutation representation graphs of string C-groups.
\end{conjecture}

The paper is organized as follows. 
The results in Section~\ref{prem} allow us to reduce our study to the case where \( G \) is a transitive imprimitive group. 
In Section~\ref{Imp}, we further reduce our analysis to two cases: either \( G \) is a transitive imprimitive group with a block system consisting of \( n/2 \) blocks of size two, or with two blocks of size \( n/2 \).
These two cases are treated separately in Sections~\ref{k=2} and~\ref{m=2}, respectively. Finally, in Section~\ref{tables+proof}, we present a proof of Theorem~\ref{main}.


\section{Preliminaries}\label{prem}

The automorphism group of an abstract regular polytope is a string C-group. One important fact about string C-groups is the one-to-one correspondence between abstract regular polytopes and string C-groups~\cite{arp}. Hence, a classification of string C-groups results in a classification of abstract regular polytopes. Before introducing the formal definition of a string C-group, we first present some preliminary results on independent sets and string groups, which are closely related concepts.

\subsection{Independent generating sets}\label{subIGS}

\begin{definition}
Let $G$ be a group. A set $S=\{\rho_0,\ldots, \rho_{r-1}\}$ is an {\bf independent generating set} of $G$  if $\rho_i\notin\langle \rho_j\,|\, j\neq i\rangle$ and $G=\langle S\rangle$.  
\end{definition}

One important result of Whiston~\cite{2000W} regards the maximal size of an independent generating set of the symmetric group of degree $n$.

\begin{thm}\cite[Theorem 1]{2000W}\label{w1}
The maximal size of an independent generating set for a group of degree $n$ is $n-1$. Moreover $\Sym_n$ is the only group having an independent generating set of size $n-1$.
\end{thm}

In what follows let $T$ be a tree with $n$ vertices and let $S(T)$ be the set of transpositions corresponding to the edges of $T$.

\begin{thm}\cite[Theorem 2.1]{2002CC} \label{CC} Let $S$ be an independent generating set for $\Sym_n$ of size $n-1$, where
$n\geq 7$. Then there is a tree $T$ on $\{1,\ldots,n\}$ such that one of the following holds:
\begin{enumerate}
\item $S = S(T )$;
\item for some element $s\in S(T )$, we have
$$S = \{s\} \cup \{(st)^{\epsilon(t)}\,:\, t\in S(T) \setminus \{s\}\} \mbox{ where } \epsilon(t)=\pm 1.$$
\end{enumerate}
Conversely, each of these sets is an independent generating set for $\Sym_n$.
\end{thm}

\subsection{Sggi's and  permutation representation graphs}\label{subSGGI}

\begin{definition}\label{sggi}
A {\bf string group generated by involutions} or, for short, a {\bf sggi} is a pair $\Gamma=(G, S)$ where $G=\langle S\rangle$ with  $S=\{\rho_0,\ldots, \rho_{r-1}\}$ being an ordered set of involutions 
that satisfy the following property, called the commuting property.
$$\forall i,j\in\{0,\ldots, r-1\}, \;|i-j|>1\Rightarrow (\rho_i\rho_j)^2=1.$$
The {\bf dual of a sggi} is obtained by reversing the sequence of generators (take the ordering induced by the indices of the elements in the generating set).
\end{definition}

When $G$ is a group of degree $n$, a sggi $(G,S)$ can be represented by its permutation representation graph, defined as follows.

\begin{definition}
Suppose that $G$ is a permutation group of degree $n$ and let $\Gamma=(G, \{\rho_0,\ldots,\rho_{r-1}\})$ be a sggi. The {\bf permutation representation graph} $\mathcal{G}$ of $\Gamma$ is an $r$-edge-labelled multigraph with $n$ vertices and with an $i$-edge
 $\{a,\,b\}$ whenever $a\rho_i=b$ with $a\neq b$ and $i\in\{0,\ldots,r-1\}$.
\end{definition}

To avoid some cumbersome notation, we introduce the following notation that will be heavely used in the following sections.

\begin{notation}Let us consider the following notation.
\[\begin{array}{ll}
I_{i_1,\ldots,i_k} :=\{0,\ldots, r-1\}\setminus \{i_1,\ldots,i_k\} & I^{\leq i}:=\{0,\ldots, i\}\qquad I^{\geq i}:=\{i,\ldots, r-1\}\\[5pt]
I^{< i}:=\{0,\ldots, i-1\}&I^{> i}:=\{i+1,\ldots, r-1\}\\[5pt]
I^{\leq i}_{i_1,\ldots,i_k} :=\{0,\ldots, i\}\setminus\{i_1,\ldots,i_k\} & I^{\geq i}_{i_1,\ldots,i_k}:=\{i,\ldots, r-1\}\setminus\{i_1,\ldots,i_k\} \\[5pt]
I^{<i}_{i_1,\ldots,i_k} :=\{0,\ldots, i-1\}\setminus\{i_1,\ldots,i_k\} & I^{>i}_{i_1,\ldots,i_k}:=\{i+1,\ldots, r-1\}\setminus\{i_1,\ldots,i_k\}\\ [5pt]
G_{i_1,\ldots,i_k}:=\langle \rho_i\,|\, i\in I_{i_1,\ldots,i_k}\rangle &  G_{\{i_1,\ldots,i_k\}}:=\langle \rho_i\,|\, i\in \{i_1,\ldots,i_k\}\rangle\\[5pt]
G_{<i}:=\langle \rho_j\,|\, j\in I^{<i}\rangle & G_{>i}:=\langle \rho_j\,|\, j\in I^{>i}\rangle\\[5pt]

\Gamma_{i_1,\ldots,i_k} :=(G_{i_1,\ldots,i_k}, \{\rho_j \,|\, j \in I_{i_1,\ldots,i_k}\})&\Gamma_{\{i_1,\ldots,i_k\}} :=(G_{\{i_1,\ldots,i_k\}}, \{\rho_j \,|\, j \in \{i_1,\ldots,i_k\}\});\\[5pt]
\Gamma_{<i} := (G_{<i},\{ \rho_0,\ldots, \rho_{i-1}\})\qquad (i\neq 0)& \Gamma_{>i} := (G_{>i},\{ \rho_{i+1},\ldots, \rho_{r-1}\})\qquad (i\neq r-1);\\[5pt]

\end{array}\]
Let $\mathcal{G}_{i_1,\ldots,i_k}$ (resp. $\mathcal{G}_{\{i_1,\ldots,i_k\}}$) denote the permutation representation graph of  $\Gamma_{i_1,\ldots,i_k}$ (resp. $\Gamma_{\{i_1,\ldots,i_k\}}$). 
\end{notation}

Notice that when $\rho_i$ is a $k$-transposition (a  product of $k$ disjoint transpositions), $\mathcal{G}_{\{i\}}$ is a matching with $k$ edges. 
A consequence of the commuting property (see Definition~\ref{sggi}) is that, if $i$ and $j$ are nonconsecutive the connected components of $\mathcal{G}_{\{i,j\}}$ with more than two vertices are \emph{$\{i,j\}$-squares} (squares with alternating labels $i,\, j,\, i,\, j$).  A \emph{$J$-edge} is a set of $|J|$ parallel edges  with label-set $J$.  Sometimes we represent these set of edges by a single edge with the label $J$.

The following lemma lists all sggi's of rank at least $n-3$ associated with primitive groups of degree $n$, other than $\Sym_n$ and $\Alt_n$.

\begin{lemma}\label{maroti2}
Let $G$ be a primitive subgroup of $\Sym_n$ not equal to $\Alt_n$ or $\Sym_n$. Let $S$ be an independent generating set of size $r$ for $G$ such that $(G,S)$ is an sggi. 
 If  $r\geq  n-3$ then $G$ is one of the groups in the following table.
\begin{table}[h!]
\begin{tabular}{|c|c|c|cc|}
\hline
$G$                        & $n$                  & $r$                  & \multicolumn{2}{c|}{Permutation representation graphs} \\ \hline
$\Dy_{10}$                  & 5                  & 2                  & \multicolumn{2}{c|}{\xymatrix@-1.3pc{*+[o][F]{} \ar@{-}[rr]^0 && *+[o][F]{} \ar@{-}[rr]^1 && *+[o][F]{} \ar@{-}[rr]^0 && *+[o][F]{} \ar@{-}[rr]^1 && *+[o][F]{}}}                                  \\ \hline
\multirow{2}{*}{$\PSL_2(5)$} & \multirow{2}{*}{6} & \multirow{2}{*}{3} & \multicolumn{2}{c|}{\xymatrix@-1.3pc{*+[o][F]{} \ar@{-}[rr]^0 && *+[o][F]{} \ar@{-}[rr]^1 && *+[o][F]{} \ar@{-}[rr]^2 && *+[o][F]{} \ar@{-}[rr]^1 && *+[o][F]{} \ar@{=}[rr]^{\{0,2\}} && *+[o][F]{} } }                                  \\
                         &                    &                    & \multicolumn{2}{c|}{\xymatrix@-1.3pc{*+[o][F]{} \ar@{-}[rr]^0 && *+[o][F]{} \ar@{-}[rr]^1 && *+[o][F]{} \ar@{=}[rr]^{\{0,2\}} && *+[o][F]{} \ar@{-}[rr]^1 && *+[o][F]{} \ar@{-}[rr]^2 && *+[o][F]{} }}                                  \\ \hline
\multirow{2}{*}{$\PGL_2(5)$} & \multirow{2}{*}{6} & \multirow{2}{*}{3} & \xymatrix@-1.3pc{ && && *+[o][F]{}\ar@{=}[rr]^{\{0,1\}} && *+[o][F]{} \\ *+[o][F]{} \ar@{-}[rr]^0 && *+[o][F]{} \ar@{-}[rr]^1 && *+[o][F]{} \ar@{-}[rr]^0 \ar@{-}[u]^2 && *+[o][F]{}\ar@{-}[u]_2}                          &   \xymatrix@-1.3pc{&& && *+[o][F]{}\ar@{=}[rr]^{\{0,1\}} && *+[o][F]{} \\ *+[o][F]{} \ar@{=}[rr]^{\{0,2\}} && *+[o][F]{} \ar@{-}[rr]^1 && *+[o][F]{} \ar@{-}[rr]^0 \ar@{-}[u]^2 && *+[o][F]{}\ar@{-}[u]_2 }                        \\
                         &                    &                    & \xymatrix@-1.3pc{&& && *+[o][F]{}\ar@{-}[rr]^{0}\ar@{-}[rrd]^{1} && *+[o][F]{} \\ *+[o][F]{} \ar@{-}[rr]^0 && *+[o][F]{} \ar@{-}[rr]^1 && *+[o][F]{} \ar@{-}[rr]^0 \ar@{-}[u]^2 && *+[o][F]{}\ar@{-}[u]_2}                           &     \xymatrix@-1.3pc{*+[o][F]{}  \ar@{-}[rr]^1 \ar@{=}[d]_{\{0,2\}} && *+[o][F]{}  \ar@{-}[rr]^0 && *+[o][F]{}  \ar@{-}[d]^1\\ *+[o][F]{}  \ar@{-}[rr]^1 && *+[o][F]{} \ar@{-}[rr]^2 && *+[o][F]{}}      \\ \hline
\end{tabular}\\[5pt]
\caption{Primitive sggi's of degree $n$ and rank $r\geq n-3$ other than $\Sym_n$ and $\Alt_n$.}\label{primsggi}
\end{table}
\end{lemma}
\begin{proof}
From \cite[Proposition 3.3]{2017CFLM}, for $n\geq 8$, $r\leq n-4$.
Hence, here we will deal with the cases where $n\leq 7$.
In what follows, $\lcs(G)$ denotes the size of a longest chain of subgroups of $G$ in its subgroup lattice.
The following table lists all transitive primitive groups $G$ of degree $n\leq 7$, that are neither $\Sym_n$ nor $\Alt_n$, having a longest chain of subgroups with size $\lcs(G)\geq n-3$.
\begin{table}
\begin{tabular}{ccccc}
\hline
$n$ & $G$&$\lcs(G)$ &
 Generated by involutions \\ \hline
5 & $\Dy_{10}$   & 2   & yes   \\
  &  $\AGL_1(5)$ & 3   & no   \\
6 & $\PSL_2(5)$  & 4   & yes   \\
  & $\PGL_2(5)$  & 5   & yes             \\
7 & $\PSL_3(2)$  & 5   & yes              \\ \hline
\\
\end{tabular}
\end{table}
Computationally, we can exclude $\AGL_1(5)$ since it is not generated by involutions, and 
 it can be checked that $\PSL_3(2)$ is neither a sggi of rank $4$ nor $5$.
The remaining ones are the ones in the statement of  this lemma - the given permutation representation graphs were obtained computationally.

\end{proof}

We now provide a rank bound for some particular sggi's related to imprimitive groups.
\begin{lemma}\label{X}
Let $\Gamma=(G, S)$ is a sggi where $G$ is a transitive subgroup of $\Sym_n$.
Suppose $G$ satisfies the following conditions:
\begin{enumerate}
\item $S=\{\rho_0,\ldots, \rho_{r-1}\}$ is independent;
\item $G_{r-1}$ is intransitive; and
\item $G_j$ is transitive for some $j\notin \{0,r-1\}$.
\end{enumerate}
If $j$ is the maximal label satisfying (c), $k$ is the size of each $G_{<j}$-orbit and $m=n/k$, then $G_j\leq \Sym_k\wr\Sym_m$ and
$$r-1\leq (k-1)+(m-1).$$
\end{lemma}
\begin{proof} 
As $G_{r-1}$ is intransitive, $G_{<j}$ is intransitive. 
Hence, as $G_j=G_{<j}\times G_{>j}$, there is a block system for $G_j$ where the blocks are the $G_{<j}$-orbits. Let $k$ be the size of a $G_{<j}$-orbit and $m$ be the number of $G_{<j}$-orbits. It follows that  $G_j\leq \Sym_k\wr\Sym_m$.
Since $G_{<j}\leq \Sym_k$ is generated by $j$ elements, it follows that $j\leq k-1$.
As $j$ is the maximal label satisfying (c), $G_i$ is intransitive for $i>j$.
Thus, for  each $i>j$ there exists a pair of   $G_{<j}$-orbits that belong to different $G_i$-orbits.
Consider a graph whose vertices are the $G_{<j}$-orbits having exactly one edge $i$ between $G_{<j}$-orbits that belong to different $G_i$-orbits.
The graph is a forest with $m$ vertices and $(r-1)-j$ edges. Hence $(r-1)-j\leq m-1$.
Consequently $r-1\leq (k-1)+(m-1)$, as wanted.
\end{proof}


\subsection{String C-groups}\label{subSCG}

\begin{definition} A {\bf string C-group of rank $r$}, denoted as $\Gamma=(G, S)$, is a sggi which satisfies the following property called the {\bf intersection property}:
$$\forall J, K \subseteq \{0,\ldots,r-1\}, \langle \rho_j \mid j \in J\rangle \cap \langle \rho_k \mid k \in K\rangle = \langle \rho_j \mid j \in J\cap K\rangle.$$
\end{definition}

If $p_i$ is the order of $\rho_{i-1}\rho_i$,  $i=1,\ldots, r-1$,  the string C-group $G$ is the group of automorphisms of an abstract regular polytope with {\bf Schl{\"a}fli symbol} $\{p_1,\ldots, p_{r-1}\}$.

The set $S$ in the previous definition is an independent generating set for $G$. An immediate consequence of Theorem~\ref{CC} is the following:
\begin{coro}\label{CCcoro}
Let $\Gamma=(G,\{\rho_0,\ldots,\rho_{r-1}\})$ be a sggi of degree $n$ generated by independent involutions.
If $r=n-1$ and $n\geq 7$ then $\Gamma$ is a string C-group, namely it is the group of automorphisms of the $(n-1)$-simplex.
\end{coro}

The following two theorems of \cite{fl} give the classification of the abstract regular polytopes, whose the automorphism group is a subgroup of $\Sym_n$, having rank $r\in\{n-1,n-2\}$.

\begin{thm}\cite[Theorem 1]{fl}\label{1}
 For $n\geq 5$, the $(n-1)$-simplex is, up to isomorphism, the unique polytope of rank
$n-1$ having a group $\Sym_n$ as automorphism group. For $n = 4$, there are, up to isomorphism and
duality, two abstract regular polyhedra whose automorphism group is $\Sym_4$, namely the hemicube
and the tetrahedron. Finally, for $n = 3$, there is, up to isomorphism, a unique abstract regular
polygon whose automorphism group is $\Sym_3$, namely the triangle.
\end{thm}

\begin{thm}\cite[Theorem 2]{fl}\label{2}
For $n\geq 7$, there exists, up to isomorphism and duality, a unique $(n-2)$-polytope 
having a group $\Sym_n$ as automorphism group and Schl{\"a}fli symbol $\{4,6,3,3,\ldots, 3\}$.
\end{thm}

In \cite{fl} and \cite{corr} the authors give the possible permutation representation graphs of the string C-groups of Theorems 1 and 2. We list them in the following proposition.

\begin{prop}\label{highC}
The permutation representation graph of degree $n$ of the group of automorphisms  of the abstract regular polytopes of rank $r= n-1$ ($n\geq 5$) is  as follows:
$$ \xymatrix@-1.3pc{*+[o][F]{}  \ar@{-}[rr]^0 &&*+[o][F]{}  \ar@{-}[rr]^1 && *+[o][F]{} \ar@{.}[rr]&& *+[o][F]{}\ar@{-}[rr]^{r-1}&& *+[o][F]{}}$$
The permutation representation graph of degree $n$ of the group of automorphisms of the abstract regular polytope of rank $r=n-2$ ($n\geq 7$) is, up to duality,  as follows:
$$ \xymatrix@-1.3pc{*+[o][F]{}  \ar@{-}[rr]^1&&*+[o][F]{}  \ar@{-}[rr]^0 && *+[o][F]{}  \ar@{-}[rr]^1&& *+[o][F]{} \ar@{-}[rr]^2&& *+[o][F]{} \ar@{.}[rr]&&*+[o][F]{}\ar@{-}[rr]^{r-1}&& *+[o][F]{}}$$
\end{prop}

In \cite{extension}, the authors also provide a classification of abstract regular polytopes whose automorphism group is a subgroup of $\Sym_n$, with rank $r \in {n - 3, n - 4}$ for sufficiently large $n$. One of the key consequences of their results is the following.

\begin{prop}\label{sggiint}
Let $\Gamma=(G,S)$ be a sggi and $x\in \{ 1, 2, 3, 4 \}$.  If $G_j$ is  intransitive for all $j\in\{ 0, \ldots, r -1 \}$, $r = n- x$  and  $n\geq 3 + 2x$, then $G\cong \Sym_n$.
Moreover if $x\in\{1,2\}$ then $\Gamma$ has, up to duality, one of the permutation representation graph given in Proposition~\ref{highC}.
\end{prop}
\begin{proof}
We just need to observe that each sggi $G$ having one of the permutation representation graphs of  Table 2 of \cite{extension} is isomorphic to $\Sym_n$.
Indeed in each case we find a transposition $(a,b)\in G$ such that the stabilizer of $a$ is transitive of $\{1,\ldots,n\}\setminus\{a\}$. Hence $G\cong \Sym_n$.

In addition, when $x\in\{1,2\}$ Table 2 of \cite{extension}  gives only two possibilities for the permutation representation graph, precisely the ones of Proposition~\ref{highC}.
 \end{proof}

 The main goal of this paper is to study string C-groups $\Gamma=(G,S)$ of rank $r\geq n/2$, for $n\geq 14$, when $G$ is a proper subgroup of $\Sym_n$ that is distinct from $\Sym_n$ and $\Alt_n$. The following proposition shows that primitive groups cannot arise in this classification.

\begin{prop}\label{prim}\cite[Proposition 3.2]{2016CFLM}
If $\Gamma=(G,S)$ is a string C-group, where $G\leq \Sym_n$ is a primitive group other than $\Sym_n$, $\Alt_n$ or any of the groups in Table~\ref{primPolys}, then $r<n/2$.
\end{prop}

\begin{table}[htbp]
\begin{center}
\begin{tabular}{|ccc|}
\hline
$n$&$G$&Schl\"afli  symbols\\
\hline
&&\\[-5pt]
10&$\Sym_6$&$\{3,3,3,3\}$\\[5pt]
6&$\Alt_5$&$\{3,5\}$, $\{5,5\}$\\[5pt]
6&$\Sym_5$&$\{3,3,3\}$, $\{4,5\}$, $\{4,6\}$, $\{5,6\}$, $\{6,6\}$\\[5pt]
\hline
\end{tabular}\\[5pt]
\caption{Primitive string C-groups of degree $n$ and rank $r\ge n/2$ other than $\Sym_n$ and $\Alt_n$.}\label{primPolys}
\end{center}
\end{table}

The following theorem summarises the results regarding the maximal rank of a string C-group $\Gamma=(G,S)$ where $G=\Alt_n$.

\begin{thm}\cite[Theorem 1.1]{2017CFLM}\label{maxan}
The maximal rank of a string C-group for $\Alt_n$ is $\lfloor\frac{n-1}{2}\rfloor$ if $n\geq 12$. If $n=3, 4, 6, 7$ or $8$, the group $\Alt_n$ does not admit a string C-group.  For the remaining degrees, the maximal ranks are listed in the following table:\\
\begin{center}
\begin{tabular}{c|c}
Group & Maximal Rank\\
\hline
$\Alt_5$& 3\\
$\Alt_9$& 4\\
$\Alt_{10}$& 5\\
$\Alt_{11}$& 6\\
\end{tabular}
\end{center}
\end{thm}

Finally, we state the theorem that establishes the bound $n/2 + 1$ for the rank of any transitive subgroup of $\Sym_n$ distinct from $\Alt_n$ and $\Sym_n$.

\begin{thm}\cite[Proposition 2.1]{2016CFLM}\label{maximprim}
Let $G$ be a transitive imprimitive subgroup of $\Sym_n$. If $\Gamma=(G,S)$ is a string C-group of rank $r$, then $r \leq n/2+1$. Moreover if $r=n/2+1$ and $n\geq 10$ then $G\cong \Cyc_2\wr \Sym_{n/2}$ and  $n\equiv 2 \mod 4$.  If $r=n/2+1$ and $n\leq 9$ then $\Gamma$ is one of the string C-groups of the following table:
\begin{center}
\begin{table}[htbp]
\begin{tabular}{|c|c|c|}
\hline
$n$&G& Schl\"afli symbols\\
\hline
&&\\[-5pt]
8&$2^4:\Sym_{3}:\Sym_{3}$&$\{3,4,4,3\}$\\[5pt]
6&$\Sym_3 \times \Sym_3$&$\{2,3,3\}$\\[5pt]
6&$2^3:\Sym_{3}$&$\{2,3,3\}$\\[5pt]
6&$2^3:\Sym_{3}$&$\{2,3,4\}$\\[5pt]
\hline
\end{tabular}\\[5pt]
\caption{Imprimitive string C-groups of degree $n\leq 9$ with rank $r\ge n/2+1$.}\label{imprimPolys}
\end{table}
\end{center}

\end{thm}

The goal of this paper is to extend this result to the case where the rank is exactly $r = n/2$.

\subsection{ Permutation representation of string C-groups for $\Sym_n$ of rank $r\in\{n-2,n-1\}$ on $2n$ points}\label{subS2n}

There is a one-to-one correspondence between faithful transitive permutation representations of a group 
$G$ and its core-free subgroup\footnote{$H$ is a core-free subgroup of $G$ when $\cap_{g\in G} g^{-1}Hg$ is the trivial.} \cite[Section 1.3]{DixMor}.
With this in mind, we show below that the automorphism group of the $(n-1)$-simplex, as well as that of the polytope of rank $n-2$ having Schl\"{a}fli type $\{4,6,3,\ldots,3\}$,  admits a unique transitive permutation representation of degree 
$2n$ when $n\geq 7$.

\begin{prop}\label{2npointsub}
  Let $n \geq 7$. There exists exactly one faithful transitive permutation representation of $\Sym_n$ on $2n$ points.
  \end{prop}
  \begin{proof}
  Let us prove that, up to conjugacy, there exists only one core-free subgroup of $\Sym_n$ of index $2n$ when $n\geq 7$.
  Suppose that $H$ is a subgroup of  $\Sym_n$ of index at most  $2n$. 
  By the O'Nan-Scott Theorem \cite{aschbacher1985maximal}, we have one of the following possibilities for $H$:
  (a) $H\leq \Sym_a \times \Sym_{n-a}$, for $1 \leq a \leq n/2$ ; (b)  $H\leq \Sym_{n/a}^a \rtimes \Sym_{a}$, for $a|n$ and $1 < a < n$;
  or (c) $H$ is a primitive subgroup of $\Sym_n$ (different from $\Alt_n$ and $\Sym_n$). 
  
  In case (a) $|H|\leq a!(n-a)!$, hence $\frac{n!}{a!(n-a)!} \leq 2n$,  which is only possible if $a = 1$. 
  Then $H\leq \Sym_{n-1}$ and $|\Sym_{n-1}:H|=2$, which gives $H\cong \Alt_{n-1}$. Case (b) is never possible as  $\frac{n!}{a!(n/a)!^a} > 2n$, for $n\geq 7$.
  In case (c) using the bound for a primitive group given in
    \cite{praeger}, $|H| \leq 4^n$, meaning that $\frac{n!}{4^n} \leq 2n$,
    which is only possible for $n\leq 11$. But then, using \cite{dimitrisubglattice} we find no possibility for small degrees either.  
  \end{proof}
  
 Notice that, when $n=6$, we find two faithful transitive permutation representation of $\Sym_6$ on $12$ points.
  
  \begin{prop}\label{2npointsrep}
    If $\Gamma=(G,\{\rho_0,\ldots, \rho_{r-1}\})$ is a string C-group of rank $r\in\{n-1,n-2\}$ and $G$ is a transitive group of degree $2n$ isomorphic to $\Sym_n$,  then
   $\Gamma$ has, up to duality, one the following permutation representations:
   
   \begin{tabular}{ll}
    $r= n-1:$&  $ 
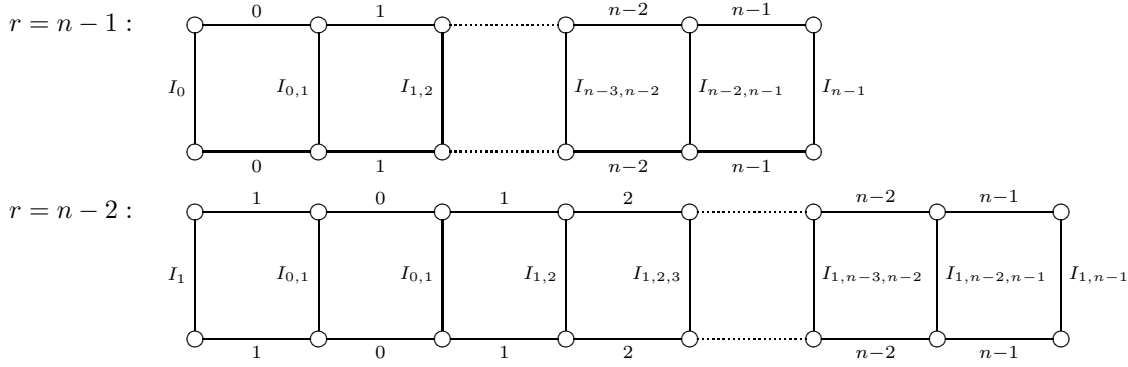
\xymatrix@-.55pc{
      *+[o][F]{}\ar@{-}[dd]_{I_0}\ar@{-}[rr]^0&&*+[o][F]{}\ar@{-}[dd]_{I_{0,1}}\ar@{-}[rr]^1&&*+[o][F]{}\ar@{-}[dd]_{I_{1,2}}\ar@{.}[rr]  &&*+[o][F]{} \ar@{-}[dd]^{I_{n-3,n-2}}\ar@{-}[rr]^{n-2}&&*+[o][F]{} \ar@{-}[dd]^{I_{n-2,n-1}}\ar@{-}[rr]^{n-1}&& *+[o][F]{}   \ar@{-}[dd]^{I_{n-1}} \\
      &&&&&&&&\\
      *+[o][F]{}\ar@{-}[rr]_0&&*+[o][F]{}\ar@{-}[rr]_1&&*+[o][F]{}\ar@{.}[rr]&&*+[o][F]{} \ar@{-}[rr]_{n-2}&&*+[o][F]{} \ar@{-}[rr]_{n-1}&& *+[o][F]{} }$\\
   
      $r= n-2:$&$ \xymatrix@-.55pc{
      *+[o][F]{}\ar@{-}[dd]_{I_1}\ar@{-}[rr]^1&& *+[o][F]{}\ar@{-}[dd]_{I_{0,1}}\ar@{-}[rr]^0&& *+[o][F]{}\ar@{-}[dd]_{I_{0,1}}\ar@{-}[rr]^1&&*+[o][F]{}\ar@{-}[dd]_{I_{1,2}}\ar@{-}[rr]^2 &&*+[o][F]{}\ar@{-}[dd]_{I_{1,2,3}}\ar@{.}[rr]  &&*+[o][F]{} \ar@{-}[dd]^{I_{1,n-3,n-2}}\ar@{-}[rr]^{n-2}&&*+[o][F]{} \ar@{-}[dd]^{I_{1,n-2,n-1}}\ar@{-}[rr]^{n-1}&& *+[o][F]{}   \ar@{-}[dd]^{I_{1,n-1}} \\
      &&&&&&&&\\
      *+[o][F]{}\ar@{-}[rr]_1&& *+[o][F]{}\ar@{-}[rr]_0&&*+[o][F]{}\ar@{-}[rr]_1 &&*+[o][F]{}\ar@{-}[rr]_2 &&*+[o][F]{}\ar@{.}[rr]&&*+[o][F]{} \ar@{-}[rr]_{n-2}&&*+[o][F]{} \ar@{-}[rr]_{n-1}&& *+[o][F]{} }$\\
  
 \end{tabular}

  \end{prop}
  \begin{proof}
    Consider first that $\Gamma$ is the string C-group of rank $n-1$, which is known as the $(n-1)$-simplex.
    Consider the subgroup $H= \langle \rho_1\rho_2, \rho_2\rho_3, \ldots, \rho_{r-3}\rho_{r-2}, \rho_{r-2}\rho_{r-1} \rangle$ of $G_0$. This group is an index $2$ subgroup of $G_0$, known as the rotational group of $\Gamma_0$.
     As $G_0\cong \Sym_{n-1}$, $H\cong \Alt_{n-1}$.
     By Proposition~\ref{2npointsub} we only need to compute the Schreier coset graph with respect to $H$. Using the  Todd–Coxeter algorithm\footnote{An algorithm that enumerates the cosets of a subgroup in a finitely presented group and, in doing so, produces a permutation representation of the group's generators acting on those cosets \cite{Jo}.} we get the graph given in the statement of this proposition.
      
   Now consider the case $r=n-2$. In this case $\Gamma$ is the abstract regular polytope of Theorem~\ref{2}.
   Up to duality $\Gamma$ is the abstract regular polytope of with Schl\"{a}fli  symbol $\{4,6,3,\ldots,3\}$. Then $G_{r-1}\cong \Sym_{n-1}$ and the rotational subgroup $H$ of $G_{r-1}$ is isomorphic to $\Alt_{n-1}$. Applying the  Todd–Coxeter algorithm to $H$ we obtain the second permutation representation given in  the statement of this proposition.      
  \end{proof}


\subsection{Imprimitive groups with blocks of size $2$ with block action isomorphic to $\Sym_{n/2}$ or  $\Alt_{n/2}$}\label{Gk=2}

Mortimer, in \cite[Lemma 2]{mortimer}, provides sufficient conditions under which all 
$G$-submodules over a field of characteristic 2 are among the following: the trivial submodule consisting of the all-zeros vector, the constant submodule generated by the all-ones vector, the even-weight submodule, and the full module. In particular, when 
$G$ contains the alternating group of degree equal to the dimension of the module, the required conditions are satisfied. As a consequence, we obtain the following result.


\begin{thm}\label{C2Sn/2}
Let $n/2\geq 3$, Suppose that $G\leq \Cyc_2\wr \Sym_{n/2}$ and let $f:G\to \Sym_{n/2}$ be the homomorphism induced by the action of 
$G$ on the $n/2$ blocks.
If ${\rm Im}(f)$ is either $\Sym_{n/2}$ or $\Alt_{n/2}$  then ${\rm Ker}(f)$ is either trivial or isomorphic to one of the groups:
$$\Cyc_2, \, (\Cyc_2)^{n/2-1}\mbox{ or  }(\Cyc_2)^{n/2}.$$ 
\end{thm}
\begin{proof}
Let $n/2\geq 3$ and $G$ be either $\Sym_{n/2}$ or $\Alt_{n/2}$.  
By Lemma 2 of  \cite{mortimer} there are only four $G$-modules over a field of characteristic $2$.
Each submodule is in one-to-one correspondence with the possible kernels of $f$, as follows:

\begin{center}
\begin{tabular}{c|l}
Submodulo& $\textrm{Ker}(f)$\\
\hline
trivial submodule & trivial group\\
constant submodule  & $\Cyc_2$ \\
even-weight submodule &  $(\Cyc_2)^{n/2-1}$\\
full module & $(\Cyc_2)^{n/2}$
\end{tabular}
\end{center}
\end{proof}


The following lemma addresses the case in which the block action is the alternating group.

\begin{lemma}\label{C2An/2lemma}
Let $n\geq 6$ and $2<l<n/2$. Suppose that $G$ is a transitive subgroup of $\Cyc_2\wr \Alt_{n/2}$ and let $f:G\to \Alt_{n/2}$ be the homomorphism induced by the action of $G$ on the $n/2$ blocks.
If ${\rm Im}(f)\cong \Alt_{n/2}$ then the following holds:
\begin{enumerate}
\item If $G$ contains a transposition fixing the blocks, then ${\rm Ker}(f)\cong (\Cyc_2)^n$;
\item If $G$ contains a $2$-transposition fixing the blocks, then $(\Cyc_2)^{n-1}\leq {\rm Ker}(f)$; and
\item If $G$ contains a $l$-transposition fixing the blocks, then  $ (\Cyc_2)^{n-1}\leq {\rm Ker}(f)$.
\end{enumerate}
\end{lemma}
\begin{proof}
Let $\tau\in G$ be a permutation fixing the blocks. We  will consider separately the cases: (a)  $\tau$ is a transposition, (b) $\tau$ is a  $2$-transposition and  (c) $\tau$ is a $(n/2-1)$-transposition.

(a) Let $B_1$ be the block where $\tau$ acts nontrivially.  Let $B_i$ be other block. There exist a permutation $g\in G$ such that $gB_1=B_i$. 
Moreover $\tau^g$ is the transposition swapping the pairs of points of $B_i$. Hence ${\rm Ker}(f)\cong (\Cyc_2)^n$.

(b) As $\Alt_{n/2}$ is $2$-transitive, any $2$-transposition fixing the blocks can be obtained by a conjugation of $\tau$.
Let $B_1$ and $B_2$ be the blocks where $\tau$ acts nontrivially.
For any pair of $\{B_i,B_j\}$ there exist $g\in G$ such that $\{B_1,B_2\}g=\{B_i,B_j\}$.
Hence  $(\Cyc_2)^{n-1}\leq {\rm Ker}(f)$.

(c)  Suppose first that $n/2$ is odd. Let  $\alpha\in G$ be a product of two cycles of size $n/2$ (permuting all blocks in a single cycle). 
Then $\tau\tau^{\alpha}$ is a $2$-transposition fixing the blocks.
Now suppose that $n/2$ is even. Let $\mathrm{Fix}(\tau)=\{a,b\}$.
Let  $\beta\in G$ be a permutation that acts on the blocks as a cycle of size $n/2-1$ that does not fix the block $\{a,b\}$, that is $\{a,b\}\beta\neq \{a,b\}$.
Then, as before, $\tau\tau^{\beta}$ is a $2$-transposition fixing the blocks. 
In both cases, by (b), we get  $(\Cyc_2)^{n-1}\leq {\rm Ker}(f)$.
\end{proof}

Combining the results of Theorem~\ref{C2Sn/2} and Lemma~\ref{C2An/2lemma}, we get the following theorem.

\begin{thm}\label{indexes}
Let $H$ be either $\Alt_{n/2}$ or $\Sym_{n/2}$ and $n\geq 6$. If $G$ is a transitive subgroup of  degree $n$  embedded into $\Cyc_2\wr H$, then the following statements hold:
\begin{enumerate}
\item The index of $G$ in $\Cyc_2\wr H$ is either  $1,\,2,\,2^{n/2-1}\mbox{ or }2^{n/2}$;
\item If the index of $G$ in $\Cyc_2\wr H$ is equal to $2^{n/2-1}$ then $G$ contains the permutation  that swaps the two points in every block;
\item If the index of $G$ in $\Cyc_2\wr H$ is equal to $2$ then $G$ contains all even permutations fixing the blocks; and
\item Let $n/2$ be even. If $|G|=2|H|$ then $G$ is even.
\end{enumerate}
\end{thm}
\begin{proof}
(a), (b) and (c) are immediate consequences of Theorem~\ref{C2Sn/2}, Lemma~\ref{C2An/2lemma} and their proofs.

(d)  In this case $G$ contains the permutation $\alpha$ swapping all pairs of points within the blocks.

Suppose first that $H=\Alt_{n/2}$.
As $n/2$ is even $G$ contains a permutation $\delta$  that fixes exactly one block $B$ and that permutes all the other blocks cyclically.
This permutation is odd if it acts nontrivially in $B$. In any case $\delta^2$ is even, indeed it is written as a product of two $(n/2-1)$-cycles.
Now consider any block $X$ and the set of three blocks $\{X, X\delta^2, B\}$. 
There exists a permutation $\beta$ that permutes these blocks cyclically, fixes another block $Y$ and swaps the remaining blocks pair wisely.
By construction $\beta^4$ is a product of two $3$-cycles permuting the blocks $\{X, X\delta^2, B\}$.
Hence  $\langle \beta^4, \delta^2\rangle$ acts has $\Alt_{n/2}$ on the blocks.
As $|G|=2|\Alt_{n/2}|$, then $G=\langle\beta^4, \delta^2, \alpha\rangle$, hence $G$ is even, as wanted.

Now consider the case $H=\Sym_{n/2}$.
In this case $G$ contains the permutation $\delta$ that permutes all the other blocks cyclically and $\delta^2$ is written as a product of two $n/2$-cycles.
As in the previous case given a triple of blocks $G$ contains a permutation, that is a product of two $3$-cycles, permuting these blocks and fixing all the other points.
With this we get a set of even generators for the group $G$, which shows that $G$ is even.
\end{proof}


\section{Imprimitive string C-groups of rank $r\geq n/2$}\label{Imp}

Consider a string C-group $\Gamma=(G,S)$ of rank $r$, where $G$ is a transitive proper subgroup of $\Sym_n$, which is neither $\Sym_n$ nor $\Alt_n$.
We will consider from now on that $r\geq n/2$. The case where $G$
is primitive is addressed in Proposition~\ref{prim}; therefore, from from now on $G$ is an imprimitive group with $m$ blocks of imprimitivity, each of size $k$.
In what follows let $\mathcal{G}$ be the permutation representation graph of $G$.

As $G$ is imprimitive, we can separate the set of generators $S$ into the following three disjoint sets:
\begin{itemize}
  \item $L$ - the set of generators of $G$ that define an independent generating set for the action of $G$ on the set of $m$ blocks;
  \item $C$ - the set of generators of $G$ which commute with all the elements of $L$; and
  \item $R$ - the remaining set of generators of $G$ that is in neither $L$ nor $C$.
\end{itemize}
In what follows, we provide upper bounds for the sizes of $L$, $C$, and $R$. These results will allow us to conclude that either $k = 2$ or $m = 2$. We treat these cases separately in Sections~\ref{k=2} and~\ref{m=2}, respectively.

The next lemma establishes a bound for $|C|$. We recall its proof here. 
\begin{lemma}\label{boundC}\cite[Section 2]{2016CFLM}.
If $\langle L\rangle$ acts primitively on the blocks, then $|C|\leq k-1$.
\end{lemma}
\begin{proof}
Suppose that all the elements of $C$ fix the blocks.
Since each element of $C$ commutes with all elements of $L$, the permutation representation graph of 
$\langle C\rangle$  consists of $m$ disjoint copies (one per block) of a permutation representation graph of degree 
$k$. Hence $\langle C\rangle$ acts faithfully on a block.
Hence $\langle C\rangle\leq \Sym_k$ and therefore $|C|\leq k-1$, as wanted.
Consider the general case where the elements of $C$ do not necessarily fix the blocks.

Suppose first that $m>2$ and let $f: G \to \Sym_m$ be the homomorphism induced by the action of $G$ on the $m$ blocks. If $\alpha\in C$ is an element that permutes the blocks,  $f(\alpha)$ is a central involution in $\textrm{Im}(f)$. Hence, the orbits of the cyclic group $\langle f(\alpha)\rangle$ form a nontrivial block system for 
$\textrm{Im}(f)$, contradicting the assumption that 
$\textrm{Im}(f)$ is primitive.
Thus for $m>2$, $|C|\leq k-1$.

Now suppose that $m=2$ and let $L=\{\alpha\}$. 
Consider the endomorphism of $G$, defined by the correspondence $\rho\mapsto \bar{\rho}$, where $\bar{\rho}$ is the permutation fixing the blocks defined as follows:
\[
\bar{\rho} = 
\begin{cases}
\rho \alpha & \text{if } \rho \text{ swaps the blocks}, \\
\rho & \text{if } \rho \text{ fixes the blocks}.
\end{cases}
\]
This gives a one-to-one correspondence between $C$ and $\bar{C}=\{\bar{\rho}_i\,|\, \rho_i\in C\}$. 
Now $\langle \bar{C}\rangle \leq \Sym_{k}$. 
Moreover if $\bar{\rho_i}\in \langle \rho_j\,|\, j\neq i\rangle$ then $\rho_i\in \langle C\setminus\{\rho_i\}\rangle \langle L\rangle$, a contradiction.
Thus  $\bar{C}$ is an independent set of permutations in $\Sym_k$, hence $|C|=|\bar{C}|\leq k-1$.
\end{proof}


\begin{lemma}\label{C4}
If $m=4$ and the action of $G$ on blocks is $(C_2)^2$, then $|C|\leq n/4-1$. Moreover we have the following:
\begin{enumerate}
\item If $|C|=n/4-1$ then $\langle L\cup C\rangle\cong \Sym_{n/4}\times (C_2)^2$; and 
\item If $\gamma\in S$ is a central involution fixing the blocks then $|C|\leq n/8$.
\end{enumerate}
\end{lemma}
\begin{proof}
Let  $L=\{\alpha,\beta\}$. Now either $(\alpha\beta)^2$ is trivial or not. Let us deal with these two cases separately.

(1) Suppose that there exists a set of nonconsecutive generators $L=\{\alpha,\beta\}$ generating the block action.
Now consider the mapping $\rho\mapsto \bar{\rho}$ where $\bar{\rho}$ is a permutation fixing the blocks that is obtained by undoing the block action using elements of $\langle L\rangle$. That is, $\bar{\rho}=\rho\gamma$ with $\gamma\in \langle \alpha,\beta\rangle$.
As $\bar{\rho}$ centralizes $\langle L\rangle$, we conclude that $\langle \bar{\rho}\,|\,\rho\in C\rangle\leq \Sym_{n/4}$ and $\{ \bar{\rho}\,|\,\rho\in C\}$ is independent, thus $|C|\leq n/4-1$.

When $|C|= n/4-1$, $\langle \bar{\rho}\,|\,\rho\in C\rangle\cong \Sym_{n/4}$, hence $\langle L\cup C\rangle\cong \Sym_{n/4}\times (C_2)^2$.

If $\gamma\in S$ is a central involution, particularly $\gamma\in C$, then $G\leq  \Sym_{n/8}\times (C_2)^3$ with $L'=\{\alpha,\beta, \gamma\}$ generating the block action. 
The set $C'$ of the elements commuting with the elements of $L'$ is equal to $C\setminus \{\gamma\}$. As before there is a natural embedding  of $C\setminus\{\gamma\}$ into $\Sym_{n/8}$.
Thus $|C\setminus\{\gamma\}|\leq n/8-1$, as wanted.

(2) Suppose there exist no pair $\{\alpha,\beta\}$ generating the block action with $(\alpha\beta)^2$ being trivial. Thus $L=\{\rho_i,\rho_{i+1}\}$ for some $i$. Then the elements of $C$ must fix the blocks otherwise we are in case (1).
Hence  $|C|\leq n/4-1$. 

Clearly if  $|C|= n/4-1$, then $\langle C\rangle=\Sym_{n/4}$. Finally if $\alpha$ centralizes $\langle C\rangle$, $|C|\leq n/8$, as wanted.
\end{proof}

A block system $\{B_1,\ldots, B_m\}$ is {\bf maximal block system} for $G$ if there is no other block system having a block $X$ ($|X|\neq n$) with $B_1\subseteq X$.

In what follows, we recall the results obtained in the proof of Proposition 2.1 of  \cite{2016CFLM}. 
We observe that in the proof of Proposition~2.1  of  \cite{2016CFLM}, the lower bound on $n$ given by the inequality $n/2 \geq 5$ was used to obtain a classification of the permutation representation graphs, rather than to bound the sets $L$, $R$, and $C$. As bounding these sets is the only focus of the following proposition, no condition on $n$ is required here.

\begin{prop}\label{LCR}\cite{2016CFLM}
Suppose that $L$, $C$ and $R$ are as defined above with respect to a maximal block system. Then these sets of generators satisfy the following properties:
\begin{enumerate}
\item $\langle L\rangle$ has a primitive action on the blocks;
\item $|L|\le m-1$;
\item If $|L|= m-1$ and $m\geq 5$ then the action of $\langle L\rangle$ on the blocks corresponds to the standard Coxeter generators of $\Sym_m$;
\item If the set of labels of $L$ is not an interval then $|L|\leq 2\log_2 m$ and $m\geq 60$, thus $|L|<m/4-3$;
\item If the set of labels of $L$ is not an interval  then $r<n/2$;
\item $|C|\leq k-1$;
\item If the set of labels of the elements of $L$ is an interval then $|R|\leq 2$; and
\item If $m\neq 2$ and the set of labels of the elements of $L$ is an interval, then $r\leq m+k-1$.\\
\end{enumerate}
\end{prop}

\begin{coro}~\label{k=2orm=2}
If $r\geq n/2\geq 7$ then one of the situations occurs: $k=2$, $m=2$. 
\end{coro}
\begin{proof}
Suppose that $k,\,m\neq 2$. 
By Proposition~\ref{LCR} (e) we may assume that the set of labels of the elements of $L$ is an interval. Then by Proposition~\ref{LCR} (h), $r\leq m+k-1$. Thus $\frac{mk}{2}\leq m+k-1$, which is equivalent to $(m-2)(k-2)\leq 2$.
The latter inequality holds for $n\geq 14$. 
\end{proof}

In most cases, we impose a lower bound \( n \geq 14 \) to ensure the analysis remains within the two main cases under consideration ( \( k = 2 \) and  \( m = 2 \)). However, this bound is not always required. Whenever it is necessary, we will state it explicitly.

\section{Case: Imprimitive groups with blocks of size $2$.}\label{k=2}

In this case we consider that $G$ has a maximal block system with blocks of size two, which implies that the action of $G$ on the blocks is primitive.
Consider the sets  $S$, $L$, $C$ and $R$ as in the previous section. Let us assume that $r=|S|\geq n/2$.
Then by Proposition~\ref{LCR} the set of labels of $L$ is an interval  and as $m\neq 2$, $|C|\leq 1$. Furthermore, if $C$ is nonempty then the element of $C$ is the permutation swapping the two points in every block.

In what follows let $\Psi$ be the sggi corresponding to the action of $L$ on the blocks.
As $\Psi$ is primitive, by  Lemma~\ref{maroti2}, one of the following situations must  occur:
$|L|\leq n/2-3$,  $\Psi$ is on the the groups of Table~\ref{primsggi} or  $\Psi\in\{\Sym_{n/2},\Alt_{n/2}\}$. In the following proposition we rule out the groups of Table~\ref{primsggi}.

\begin{prop}\label{phiprim}
If $r\geq n/2$ then $\Psi$ cannot be isomorphic to one of the groups $\Dy_{10}$, $\PGL_2(5)$ or $\PSL_2(5)$.
\end{prop}
\begin{proof}
Suppose that $\Psi$ is one of the groups  excluded in this proposition.
By Theorem~\ref{w1}  $|L|\leq n/2-2$. In these cases $|R\cup C|=3$. Up to duality we may assume that $C=\{\rho_0\}$ and $R=\{\rho_1,\rho_{r-1}\}$.  
Notice that  $G_{i,j}<G_i<G$, $G_{i,j}<G_j<G$ and $G_i\neq G_j$.  However $G_i$ and $G_j$ might be isomorphic, since  the corresponding permutation representations, can give distinct subgraphs of $\mathcal{G}$.

Let us consider each group, $\PGL_2(5)$, $\PSL_2(5)$ and $\Dy_{10}$, separately.

(1) $\Psi\cong \PGL_2(5)$:
The following diagram gives the subgroups of $\Cyc_2\wr \PGL_2(5)$ which have $\Psi$ as the block action.
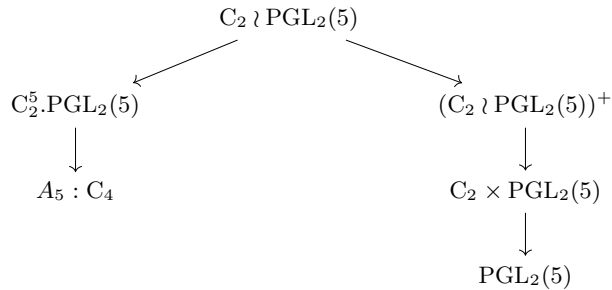
\begin{center}
\begin{small}
\begin{tikzcd}
    & \Cyc_2\wr \PGL_2(5) \arrow[ld] \arrow[rd] & \\
\Cyc_2^5.\PGL_2(5) \arrow[d] &   & (\Cyc_2\wr \PGL_2(5))^+ \arrow[d]    \\
\Alt_5:\Cyc_4  &   & \Cyc_2\times\PGL_2(5) \arrow[d] \\
    &      & \PGL_2(5)                      
\end{tikzcd}
\end{small}
\end{center}
Suppose first that $|R\cup C|=3$.
As $\PGL_2(5)$ is almost simple and  $G_{1,r-1} = \langle L\rangle \times\langle C\rangle$,  $G_{1,r-1} \ncong \PGL_2(5)$. 
Hence $G_{1,r-1}$ is either $\Alt_5:\Cyc_4$ or $\Cyc_2\times\PGL_2(5)$. 
Then $G_1\cong G_{r-1}\cong \Cyc_2^5.\PGL_2(5)$ or $G_1\cong G_{r-1}\cong (\Cyc_2\wr \PGL_2(5))^+$. But these groups have a unique transitive permutation representation on $12$ points that is represented as a subgraph of $\mathcal{G}$,  which gives $G_1=G_{r-1}$, a contradiction.
Thus $|R\cup C|=2$.

 Now let $R\cup C=\{\rho_i,\rho_j\}$. To avoid the previous contradiction we must have $G_{i,j}\cong \Psi\cong  \PGL_2(5)$ which might be transitive or intransitive (this is the only group of the diagram above that might be intransitive). 
But then either $G_i\cong G_j\cong  \Cyc_2\times\PGL_2(5)$ or $G_i\cong G_j\cong(\Cyc_2\wr \PGL_2(5))^+$. 
In the first case there is a central involution $\delta\in G_i\cap G_j=G_{i,j}\cong \PGL_2(5)$, a contradiction.
In the second case implies $G_i=G_j$, a contradiction.

(2) $\Psi\cong \PSL_2(5)$: Consider the case where $|R\cup C|=3$.
The transitive subgroups of $\Cyc_2\wr \PSL_2(5)$ having block action $\Psi$ are the following:
$$\PSL_2(5)< \Cyc_2\times \PSL_2(5)<(\Cyc_2\wr \PSL_2(5))^+<\Cyc_2\wr \PSL_2(5)$$
In this case we must have  $G_{1,r-1}\cong \Cyc_2\times \PSL_2(5)$ and $G_1\cong G_{r-1}$.
But this gives  $G_1=G_{r-1}$, a contradiction.
Then $|R\cup C|=2$. In this case, let $R = \{\rho_i\}$ and $C=\{\rho_j\}$. This implies that $G_{i}\cong \Cyc_2\times \PSL_2(5)$ and $G_j$ cannot be isomorphic to  $\Cyc_2\times \PSL_2(5)$. Hence either $G_j$ is a proper subgroup of $G_i$ or $G_i$ is a proper subgroup of $G_j$, a contradiction. 

 (3) $\Psi\cong \Dy_{10}$: In this case $|R\cup C|=3$.
The transitive subgroups of $\Cyc_2\wr \Dy_{10}$ having block action $\Psi$ are the following:
$$\Dy_{10};\, \Cyc_2\times \Dy_{10};\,(\Cyc_2\wr \Dy_{10})^+;\,\Cyc_2\wr \Dy_{10}$$
Then $G_1\cong \Cyc_2\times \Dy_{10}$. If $G_0\cong G_{r-1}\cong (\Cyc_2\wr \Dy_{10})^+$ then $G_0=G_{r-1}$, a contradiction.  
\end{proof}

By Proposition~\ref{LCR}, we have \( |R \cup C| \leq 3 \). We now show that if equality holds, then the group action on the blocks contains the alternating group.

\begin{prop}\label{RC3a}
Let $r\geq n/2$.
If $|R\cup C|=3$ then $\Psi$ is isomorphic to $\Sym_{n/2}$ or $\Alt_{n/2}$.
\end{prop}
\begin{proof}
Up to duality we may assume that $C=\{\rho_0\}$, $R=\{\rho_1,\rho_{r-1}\}$ and $L=\{\rho_2,\ldots,\rho_{r-2}\}$.
Let  $\alpha_i$ be the action of $\rho_i$ on the blocks. Then $\Psi=\langle \alpha_2,\ldots,\alpha_{r-2}\rangle$.
Here we consider the following notation  $\Psi_i:=\langle \alpha_j\,|\, j\neq i\rangle$.

Let us prove that $\Psi_i$ is intransitive for every $i\in\{2,\ldots,r-2\}$ when $n/2\geq 9$.

(1) $\Psi_2$ and $\Psi_{r-2}$ are intransitive:
Suppose that $\Psi_2$ is transitive.
If $\rho_1$ swaps a pair of points inside a block, then, as $\rho_1$ centralizes $G_{>2}$, $\rho_1=\rho_0$, a contradiction.
Thus  $\rho_1$ swaps a pair of blocks. Then, the transitivity of  $G_{>2}$, forces $\rho_1$ to swap all blocks pair wisely.
Moreover $\rho_0\rho_1$ is also a permutation swapping all blocks pair wisely and $\rho_0\rho_1\in \langle L\rangle$, a contradiction.
Therefore $\Psi_2$ is intransitive and by duality $\Psi_{r-2}$ is also intransitive.

(2) $\Psi_i$ is intransitive for $i\in\{3,\ldots,r-3\}$:
Suppose that $\Psi_i\,(=\Psi_{\{2,\ldots,i-1\}}\times \Psi_{\{i+1,\ldots,r-2\}})$ is transitive.
As $\Psi_{\{2,\ldots,i-1\}}\leq \Psi_{r-2}$ and $\Psi_{\{i+1,\ldots,r-2\}}\leq \Psi_2$, these groups  of the decomposition of $\Psi_i$ are intransitive, by (1). 
Hence $\Psi_i$ is imprimitive.
Then, by Lemma~\ref{X}, $r-4=|L|-1\leq (k'-1)+(m'-1)$. As  $|L|\geq n/2-3$, it follows that $n/2-4\leq k'+m'-2$, which is only possible if $(k'-1)(m'-1)\leq 3$, which is never the case as  $n/2=k'm'\geq 9$.
This proves that  for $i\in\{2,\dots, r-2\}$, $\Psi_i$ is intransitive.
Now if $n/2\geq 2. 3 +3\time 3=9$, by Proposition~\ref{sggiint}, $\Psi\cong \Sym_{n/2}$, as wanted.

Now suppose that $n/2\leq 8$. As $|L|\geq n/2-3$, by Lemma~\ref{maroti2} and Proposition~\ref{phiprim}  $\Psi$ is either isomorphic to  $\Sym_{n/2}$ or to $\Alt_{n/2}$.
\end{proof}

However, as we show in the following proposition, the case \( \Psi \in \{ \Sym_{n/2}, \Alt_{n/2} \} \) also leads to a contradiction. Consequently, we obtain an improved upper bound on the size of \( R \cup C \), and a corresponding lower bound on the size of \( L \).

\begin{prop}\label{RC<3}
If  $r\geq n/2$ then $|R\cup C|< 3$ and $|L|\geq n/2-2$.
\end{prop}
\begin{proof}
Suppose that $|R\cup C|=3$. 
Then, by Proposition~\ref{RC3a}, $\Psi$ is isomorphic to $\Sym_{n/2}$ or $\Alt_{n/2}$.
First consider the case $\Psi\cong \Sym_{n/2}$. 
Then $C=\{\rho_0\}$ and $R=\{\rho_1,\rho_{r-1}\}$. 
As $L$ is a minimal  set of generators generating the block action, and $G$ is transitive, $G_{0,1}=\langle L\cup\{\rho_{r-1}\}\rangle$ is transitive.
Hence the group $G_{0,1}$ is a transitive subgroup of $\Cyc_2\wr \Sym_{n/2}$.
Then, by Theorem~\ref{indexes}, $G_0$ and $G_1$ must be index $2$ subgroups of $\Cyc_2\wr \Sym_{n/2}$, while the order of $G_{0,1}$ is twice that of $\Sym_{n/2}$.
Moreover, $G_{0,1}$ contains $\rho_0$, the permutation swapping all points within the blocks, a contradiction.
 The same argument can be applied when $\Psi$ is isomorphic to $\Alt_{n/2}$.
 
Now as $|R\cup C|\leq 2$, it follows that $|L|= r-|R\cup C|\geq n/2-2$.
\end{proof}

In what follows, we consider the cases \( |R \cup C| = 2 \) and \( |R \cup C| = 1 \). Proposition~\ref{LC=2} and Corollary~\ref{k=2RC2} address the case \( |R \cup C| = 2 \).

\begin{prop}\label{LC=2}
 Let $r\geq n/2$. If $|R\cup C|=2$ then  $\langle L\rangle\cong \Sym_{n/2}$, $|R|=|C|=1$, $n/2$ is odd and $G\cong \Cyc_2\wr \Sym_{n/2}$.
\end{prop}
\begin{proof}
In this case $|L|\geq n/2-2$, hence $\Psi$ is a transitive  sggi of rank $n/2-2$ and degree $n/2$. Suppose that $\Psi$ is neither $\Sym_{n/2}$ nor  $\Alt_{n/2}$. Hence by Lemma~\ref{maroti2} $\Psi$ is isomorphic to 
one of the groups  $\Dy_{10}$, $\PSL_2(5)$ or  $\PGL_2(5)$. But Proposition~\ref{phiprim} excludes the possibility of $\Psi$ being isomorphic to these groups.
Hence $\langle L\rangle$ is a subgroup of $\Cyc_2\wr H$ with $H$ being $\Alt_{n/2}$ or $\Sym_{n/2}$. 
Let us now use Theorem~\ref{indexes} to conclude that $\langle L\rangle\cong H$.
As $\langle L\rangle=G_{i,j}$ for some $i$ and $j$, $\langle L\rangle$ cannot be an index $2$ subgroup of  $\Cyc_2\wr H$.
Suppose that $|\langle L\rangle|=2|H|$. Then for distinct $i$ and $j$ the subgroups $G_i$ and $G_j$ must be index $2$ subgroups of $\Cyc_2\wr H$.
Thus $G_i$ and $G_j$, and consequently $\langle L\rangle$, contain all even permutations fixing the blocks, a contradiction.
Hence $\langle L\rangle\cong H$, as wanted.

Now suppose that $H\cong \Alt_{n/2}$.
Then $\Psi\cong \Alt_{n/2}\cong\langle L\rangle$, thus $ \Psi$  is a string C-group. Hence  we can use Theorem~\ref{maxan} to conclude that $|L|\leq \frac{n/2+1}{2}$. This implies that $n/2\leq 5$. 
The only alternating group of degree at most $5$ that is a string C-group is  $\Alt_5$. This implies that $\langle L\rangle\cong \Psi\cong \Alt_5$.
But $\Alt_5$ does not have a transitive imprimitive permutation representation on $10$ points, a contradiction.

Suppose that $R\cup C=\{\rho_i,\rho_j\}$.
As $G_{i,j}$ does not contain a nontrivial permutation fixing all blocks, we may consider that, $G_{i,j}$ is an index $2$ subgroup of $G_j$ and $G_i$ is an index $2$ subgroup of $\Cyc_2\wr \Sym_{n/2}$.
Hence $\rho_i$ commutes with all the elements of $G_{i,j}$, thus $\rho_i\in C$. 
Moreover as $G_j$ is not a subgroup of $G_i$, $\rho_i$ must be an odd permutation, which is only possible if $n/2$ is odd.
We also conclude that  $G\cong \Cyc_2\wr \Sym_{n/2}$ 

The rest follows from Theorems~\ref{1} and ~\ref{2}.
\end{proof}

Now using Proposition~\ref{LC=2} it is possible to determine the possibilities for the permutation representation graph of $\Gamma$ when $|R\cup C|=2$.

\begin{coro}\label{k=2RC2}
  Let $r\geq n/2\geq 7$. If $|R\cup C|=2$ then $n/2$ is odd and $\mathcal{G}$ is, up to duality, one of the graphs of Table~\ref{RC2} (located at the end of the paper).
\end{coro}
\begin{proof}
  Using Proposition~\ref{LC=2}, it is possible to determine the possibilities for the permutation representation graph of $G$ when $|R\cup C| = 2$. 

  Up to duality we may assume that $C = \{\rho_0\}$ where $\rho_0$ is the permutation swapping all pairs of points within a block,
  say $\rho_0 =(1,2)\ldots(n-1,n)$. 
  Then either $R=\{\rho_1\}$ or $R=\{\rho_{r-1}\}$. 

  We also need to consider two possibilities, either $|L| = n/2-1$ or $|L| = n/2-2$.  When $|L| = n/2-1$, $\langle L\rangle$ is the automorphism group of the simplex, which is self dual; 
  when $|L| = n/2-2$, $\langle L\rangle$ is the automorphism group of one of the two abstract regular polytopes of rank $n/2-2$ for $\Sym_{n/2}$, having one of the Schl{\"a}fli symbol $\{3,\ldots,3,6,4\}$ or $\{4,6,3,\ldots,3\}$ (which are dual to each other).
Finally, the possibilities for the permutation representation graph of  these polytopes are determined by Propositions~\ref{highC} and ~\ref{2npointsrep}, depending on whether $\langle L\rangle$ is intransitive or transitive. If  $\langle L\rangle$  is intransitive, the  permutation graph of $\langle L\rangle$ is given by two copies of one of the graphs given in Proposition~\ref{highC}.
The possibilities for the element of $R$, which must be an even permutation, are determined by the commuting property. 
With this we obtain the  graphs listed in Table~\ref{RC2}.
\end{proof}

It remains to consider the case \( |R \cup C| = 1 \). In what follows, we determine the action on the blocks. To achieve this, a lower bound on \( n \) is required.

\begin{prop}\label{P}
  Let $r\geq n/2\geq 7$. If $|R\cup C|=1$ then $\Psi\cong \Sym_{n/2}$ and  the action of $G$ on the blocks is given by the following graph:
   $$\xymatrix@-1.7pc{*+[o][F]{}\ar@{-}[rr]^1&&*+[o][F]{}\ar@{.}[rr]&&*+[o][F]{}\ar@{-}[rr]^{r-1} &&*+[o][F]{} }$$
\end{prop}
\begin{proof}
  In this case $|L| \geq n/2-1$, hence $\Psi$ is a transitive sggi of rank $n/2 -1$ and degree $n/2$. Up to duality we may assume that $R\cup C=\{\rho_0\}$.  
  By Corollary~\ref{CCcoro},  $\Psi\cong \Sym_{n/2}$ and the block action graph is the one given in the statement of this proposition.

  \end{proof}
We now determine the action of \( \langle L \rangle \). Either \( \langle L \rangle \cong \Sym_{n/2} \) or \( \langle L \rangle \not\cong \Sym_{n/2} \), and the following proposition addresses the former case.

  \begin{prop}\label{S}
  Let $r\geq n/2\geq 7$. If $|R\cup C|=1$ and $\langle L \rangle \cong \Sym_{n/2}$ then $\mathcal{G}$ is one of the graphs of Table~\ref{L=S} or the graphs (1) and (2) of Table~\ref{m=2table}.
\end{prop}
\begin{proof}
  First, consider the case where $\langle L\rangle$ is intransitive.
  Then $\langle L\rangle$ is represented by two copies of the permutation graph of the symplex. 
  Suppose first that $|C|=1$. In this case $G\cong \Cyc_2\times \Sym_{n/2}$ and $G$ admits another block system with exactly two blocks. The permutation representation graph of $\Gamma$ is the graph (1) appearing on Table~\ref{m=2table}.
  If $|C|=0$ then we get, up to duality, the permutation representation graphs (13) and (14) of Table~\ref{L=S}.

  Consider now that $\langle L\rangle$ is transitive. In this case the permutation representation of $\langle L\rangle$ is given by Proposition~\ref{2npointsrep}.
  If $|C|=1$ then $G\cong \Cyc_2\times \Sym_{n/2}$ and $G$ admits another block system with exactly two blocks. The permutation representation graph of $\Gamma$ is the graph (2) of Table~\ref{m=2table}. 
  If $|C|=0$ then, we get the graphs (15) and (16) of Table~\ref{L=S}.
 \end{proof}

In what remains of this section assume the following: 
\begin{itemize}
\item $r\geq n/2\geq 7$;
\item  $R\cup C=\{\rho_0\}$;
\item $\langle L \rangle \not\cong \Sym_{n/2}$; and
\item $f : \langle L\rangle \rightarrow \Sym_{n/2}$, the homomorphism induced by the action of $\langle L\rangle$ on the $n/2$ blocks.
\end{itemize}
 By Proposition~\ref{P} the permutation representation graph of  $\Psi$ determines a natural ordering on the blocks: let $B_1$ be the first block (on the left) and $B_{n/2}$ be the last block (on the right).
 Considering $f: \langle L\rangle \rightarrow \Sym_{n/2}$ as above, by Theorem~\ref{C2Sn/2}, either ${\rm Ker}(f)\cong \Cyc_2 $, ${\rm Ker}(f)\cong (\Cyc_2)^{n/2-1}$ or ${\rm Ker}(f)\cong (\Cyc_2)^{n/2}$.
 However, it cannot be ${\rm Ker}(f) \cong (\Cyc_2)^{n/2}$, since this would imply that $\langle L \rangle$ is the full wreath product $\Cyc_2 \wr \Sym_{n/2}$, which is not possible because $R \cup C$ is nonempty.
 We will represent each element $\alpha\in {\rm Ker}(f)$ as a vector $\tilde{\alpha}\in \{0,1\}^{n/2}$. 
 Particularly, the central involution permuting all pairs of points within a block corresponds to the all 1's vector.
 In what follows  consider the following notation where $x^i$ represents a sequence of length $i$ of $x$'s, say  ($x,\,x,\ldots,\, x$), $x\in\{0,1\}$.

 \[
 \begin{array}{ll}
\mathcal{O}:=(0^r)&\mbox{all 0's vector}\\
\mathcal{U}:=(1^r)&\mbox{all 1's vector}\\
\mathcal{L}_i:=(1^i, 0^{r-i})&\mbox{a left-side 1's vector when }i\in\{1,\ldots,r-1\}; \, \mathcal{L}_0:= \mathcal{O}\mbox{ and }\mathcal{L}_r:= \mathcal{U}\\
\mathcal{R}_i:=(0^i, 1^{r-i})&\mbox{a right-side 1's vector when }i\in\{1,\ldots, r-1\}; \,  \mathcal{R}_{r}:= \mathcal{O}\mbox{ and }\mathcal{R}_0:= \mathcal{U}\\
\mathcal{V}_i:=(1^i,0,0,1^{r-(i+2)})&\mbox{a 2-central 0's vector when }i\in\{1, \ldots, r-3\};\,  \mathcal{V}_0:= \mathcal{R}_{2}; \, \mathcal{V}_{r-2}:= \mathcal{L}_{r-2}\\
\mathcal{T}_i:=(1^i,0,0,0,1^{r-(i+3)})&\mbox{a 3-central 0's vector when } i\in\{1,\ldots,r-4\}\}; \,  \mathcal{T}_0:= \mathcal{R}_{3}; \,\mathcal{T}_{r-3}:= \mathcal{L}_{r-3};\\
\end{array}
 \]
 
 For $i\in\{1,\ldots,r-1\}$  let  $\rho_i=\alpha_i\beta_i$ with $\alpha_i$ being a permutation fixing the blocks and $\beta_i$ being the permutation swapping $B_i$ and $B_{i+1}$. 
Then thanks  to the commuting property, $\tilde{\alpha}_i$ is  either  $\mathcal{O}$, $\mathcal{L}_{i-1}$, $\mathcal{R}_{i+1}$ or $\mathcal{V}_{i-1}$.

 Let $\delta_i:=(\rho_i\rho_{i+1})^3$ ($i>0$). In the following table we determine all the possibilities for $\tilde{\delta}_i$. As  $\delta_i\in {\rm Ker}(f)$ and  ${\rm Ker}(f)\cong \Cyc_2 $ or ${\rm Ker}(f)\cong (\Cyc_2)^{n/2-1}$,  $\delta_i=(\rho_i\rho_{i+1})^3$  is either the permutation $(1,2)\ldots (n-1,n)$ or an even permutation.
 In the following table we determine all the possibilities for  $\tilde{\delta}_i$ for all the possibilities for the pair $(\tilde{\alpha}_i,\tilde{\alpha}_{i+1})$. 
 In some cases the result is an odd permutation, thus these cases cannot happen (otherwise $\langle L\rangle$ is the full wreath product $\Cyc_2\wr \Sym_{n/2}$,  implying that $\rho_0\in \langle L\rangle$, a contradiction).   
  \begin{table}[htbp]
 \begin{tabular}{|c|cc|cc|}
 \hline
   $i\in\{2,\ldots, r-3\}$ &&  $i=1$ && $i=r-2$\\
   \hline
   &&&&\\
 \begin{tabular}{c|cccc}
$ \tilde{\alpha_i}	\backslash \tilde{\alpha}_{i+1}$ &$\mathcal{O}$&$\mathcal{L}_i$& $\mathcal{R}_{i+2}$&$\mathcal{V}_i$\\
 \hline
 $\mathcal{O}$&$\mathcal{O}$&$\mathcal{L}_{i+2}$& $\mathcal{R}_{i+2}$&$\mathcal{U}$\\
 $\mathcal{L}_{i-1}$&$\mathcal{L}_{i-1}$&odd&$\mathcal{T}_{i-1}$&$\mathcal{R}_{i-1}$\\
 $\mathcal{R}_{i+1}$ & $\mathcal{R}_{i-1}$ & $\mathcal{T}_{i-1}$ & odd & $\mathcal{L}_{i-1}$\\
 $\mathcal{V}_{i-1}$&$\mathcal{U}$&$\mathcal{R}_{i+2}$&$\mathcal{L}_{i+2}$&$\mathcal{O}$\\
 \end{tabular}
 & \qquad&
  \begin{tabular}{c|cccc}
 $ \tilde{\alpha_1}	\backslash \tilde{\alpha}_2$&$\mathcal{O}$&$\mathcal{L}_1$& $\mathcal{R}_3$&$\mathcal{V}_1$\\
  \hline 
  $\mathcal{O}$& $\mathcal{O}$&odd&$\mathcal{R}_3$&$\mathcal{U}$\\
  $\mathcal{R}_2$& $\mathcal{U}$&$\mathcal{R}_3$&odd&$\mathcal{O}$\\
  \end{tabular}
 &\qquad &
  \begin{tabular}{c|cc}
  $ \tilde{\alpha}_{r-2}	\backslash \tilde{\alpha}_{r-1}$ &$\mathcal{O}$& $\mathcal{L}_{r-2}$\\
  \hline 
  $\mathcal{O}$& $\mathcal{O}$&$\mathcal{U}$\\
  $\mathcal{L}_{r-3}$&$\mathcal{L}_{r-3}$&odd\\
 $\mathcal{R}_{r-1}$& odd & $\mathcal{L}_{r-3}$\\
 $\mathcal{V}_{r-3}$&$\mathcal{U}$&$\mathcal{O}$\\
  \end{tabular}\\
  &&&&\\
  \hline
 \end{tabular}
   \caption{ Possibilities for $\tilde{\delta}_i$.}
      \label{LRBO}
    \end{table}


\begin{prop}\label{K=C2} 
  Let $r\geq n/2\geq 7$, $|R\cup C|=1$ and $\langle L \rangle \ncong \Sym_{n/2}$.
  Let $f : \langle L\rangle \rightarrow \Sym_{n/2}$  be the homomorphism induced by the action of $\langle L\rangle$ on the $n/2$ blocks.
  If ${\rm Ker}(f) \cong \Cyc_2$ then $\mathcal{G}$ is, up to duality,  one of the graphs of Table~\ref{kc2}.
\end{prop}
\begin{proof}
 In this case  $\delta_i=(1,2)(3,4)\ldots(n-1,n)$ for some $i\in\{1,\ldots,r-2\}$ and, for $j\in\{1,\ldots, r-2\}\setminus \{i\}$, $\delta_j$ must be trivial.
In addition, up to duality, either  $\rho_0=(1,2)$ or  $\rho_0=(3,4)\ldots (n-1,n).$ 
Since $\tilde{\delta}_i=\mathcal{U}$, the possibilities for  $(\tilde{\alpha_i}, \tilde{\alpha}_{i+1})$ are determined in Table~\ref{LRBO}. 

Suppose that $(\tilde{\alpha_i}, \tilde{\alpha}_{i+1})=(\mathcal{V}_{i-1},\mathcal{O})$ and $i\neq 1$. 
Then, as $\tilde{\delta}_j=\mathcal{O}$  for $j\neq i$,  $\tilde{\alpha_j}=\mathcal{V}_{j-1}$ for each $j\in\{2,\ldots,i-1\}$ and $\tilde{\alpha_1}=\mathcal{R}_2$.
This gives the graphs (17) and (18) of Table~\ref{kc2}.
When $i=1$, $(\tilde{\alpha_i}, \tilde{\alpha}_{i+1})=(\mathcal{R}_{2},\mathcal{O})$, giving the graphs (21) and (22) of Table~\ref{kc2}.

Suppose now that $(\tilde{\alpha_i}, \tilde{\alpha}_{i+1})=(\mathcal{O}, \mathcal{V}_i)$ and $i\neq r-2$.
Using similar arguments as above, we get the graphs (19) and (20) of Table~\ref{kc2}.
When $i=r-2$,  $(\tilde{\alpha_i}, \tilde{\alpha}_{i+1})=(\mathcal{O}, \mathcal{L}_{r-2})=(\mathcal{O}, \mathcal{V}_{r-2})$, giving the graphs (23) and (24) of Table~\ref{kc2}.
\end{proof}

\begin{lemma}\label{r0}
 Let $r\geq n/2\geq 7$, $|R\cup C|=1$ and $\langle L \rangle \ncong \Sym_{n/2}$.  
  Let $i\in\{1,\ldots,r-2\}$. If $\delta_i$ is a non-trivial even permutation and $\delta_i\neq (1,2)(3,4)\ldots(n-1,n)$ then
$n/2$ is odd and $\rho_0=(1,2)(3,4)\ldots(n-1,n)$.
\end{lemma}
\begin{proof}
Suppose that neither $\delta_i$ nor $\rho_0$ is equal to  the permutation $(1,2)(3,4)\ldots(n-1,n)$.
Then, by the commuting property, $\rho_0$ is, up to duality, one of the permutations:
$(1,2)$ or  $(3,4)\ldots(n-1,n)$.

Let us first consider that $i\neq 1$. 
Note that  $(\rho_0\rho_1)^2=(1,2)(3,4)$ and  $[(\rho_0\rho_1)^2]^{\rho_1\rho_2\rho_1}=(3,4)(5,6)$.
If $\delta_i$ fixes $B_1$ pointwise then $G_{0,1}$ contains all even permutation fixing the blocks $B_2$, ..., $B_{n/2}$, particularly $(3,4)(5,6)\in G_{0,1}$. 
Notice that, as $(3,4)(5,6)\in G_{0,1}\cap\langle \rho_0,\rho_1,\rho_2\rangle=\langle \rho_2\rangle$, this leads to a contradiction.
If $\delta_i$  swaps the points of the block $B_1$, then $(1,2)(3,4)\in G_{0,1}$. Hence $(1,2)(3,4)\in G_{0,1}\cap\langle \rho_0,\rho_1\rangle$, a contradiction.
Thus if $\delta_i$ is not $(1,2)(3,4)\ldots(n-1,n)$ then  $\rho_0=(1,2)(3,4)\ldots(n-1,n)$.

Now suppose that $i=1$. Then $G_0$ contains all even permutations fixing the blocks $B_1,\ldots, B_{n/2}$. Particularly $(1,2)(3,4)\in G_0$.
If $\rho_0=(1,2)$ then $G_{<4}$ contains all permutations fixing the blocks $B_1$, $B_2$ and $B_3$ particularly $(1,2)(3,4)\in G_{<4}$. But $(1,2)(3,4)\notin \langle \rho_1,\rho_2\rangle$, contradicting the intersection property.
Now consider that  $\rho_0=(3,4)\ldots (n-1,n)$.
According to Table~\ref{LRBO}, $(\tilde{\alpha}_1, \tilde{\alpha}_2)\in \{(\mathcal{O},\mathcal{R}_3), (\mathcal{R}_2,\mathcal{L}_1)\}$, which gives $\tilde{\delta}_1=(0^3,1^{n/2-3})$. 
Thus  $n/2$ must be odd and therefore $\rho_0$ is even. As $G_0$ contains all even permutation fixing the blocks,  $\rho_0\in G_0$, a contradiction.
This shows that $\rho_0=(1,2)(3,4)\ldots(n-1,n)$.

Finally, if  $n/2$ is even,  as in both cases $G_0$ contains all even permutations fixing the blocks, $(1,2)(3,4)\ldots(n-1,n) \in G_0$. Hence $\rho_0\in G_0$, a contradiction. Hence $n/2$ is odd.
\end{proof}

 \begin{lemma}\label{Cons}
   Let $r\geq n/2\geq 7$, $|R\cup C|=1$ and $\langle L \rangle \ncong \Sym_{n/2}$.
  
 Let $i,j\in\{1,\ldots,r-2\}$ and $i<j$. If $\delta_i$ and $\delta_j$ are nontrivial even permutations different from $(1,2)(3,4)\ldots(n-1,n)$  then  either  $j = i+1$ or there exists  $k\in\{i+1,\ldots ,j-1\}$ such that  $\delta_k$ is nontrivial.
\end{lemma}
\begin{proof}
 Suppose for a contradiction that  $j\neq i+1$ and that $\delta_k$ is trivial for $k\in\{i+1,\ldots ,j-1\}$.
 Then $\langle\rho_{i+1},\ldots, \rho_j\rangle\cong \Sym_{j-i+1}$, particularly  $\langle\rho_{i+1},\ldots, \rho_j\rangle$ does not contain a nontrivial permutation fixing  $\{B_{i+1},\,\ldots, B_{j+1}\}$.
Moreover any permutation in $G_{0,i,j+1}$ that swaps a pair of points within one of the blocks $\{B_{i+1},\,\ldots, B_{j+1}\}$, must swap all pairs of points of these blocks.

As $\delta_j\neq (1,2)(3,4)\ldots(n-1,n)$, either $\delta_j\delta_j^{\rho_{j-1}}$   or 
$\delta_j\delta_j^{\rho_{j+2}}$ is a $2$-transposition fixing $\{B_{j+1},\,\ldots, B_{n/2}\}$.
Particularly, $G_{0,i}$ contains all the $2$-transpositions fixing $\{B_{i+1},\,\ldots, B_{j+1}\}$.
Using similar argument we also conclude that $G_{0,j+1}$ contains all the $2$-transpositions fixing $\{B_{i+1},\,\ldots, B_{j+1}\}$.
 But then $G_{0,j+1}\cap G_{0,i}> G_{0,i,j+1}$, a contradiction.
\end{proof}


\begin{prop}\label{More}
  Let $r\geq n/2\geq 7$, $|R\cup C|=1$ and $\langle L \rangle \ncong \Sym_{n/2}$.
  Let $f : \langle L\rangle \rightarrow \Sym_{n/2}$ be the homomorphism induced by the action of $\langle L\rangle$ on the $n/2$ blocks.
  If ${\rm Ker}(f) = \Cyc_2^{n/2-1}$ then the following holds:
\begin{enumerate}
\item $|\{j\in\{1,\ldots r-2\}\,|\, \delta_j\neq id\}|>1$ and $\{j\in\{1,\ldots r-2\}\,|\, \delta_j\neq id\}$ is an interval;  and
\item $n/2$ is odd and $\rho_0=(1,2)(3,4)\ldots(n-1,n)$.
\end{enumerate}
\end{prop}
\begin{proof}
As ${\rm Ker}(f) = \Cyc_2^{n/2-1}$ it follows that for some $i\geq1$,  $\delta_i$ is an even permutation different from $(1,2)\ldots (n-1,n)$. 
Suppose that $\delta_j$ is trivial for $j\in\{1,\ldots,  r-2\}\setminus \{i\}$. 

Consider first $i\neq 1,\,r-2$. 
According to Table~\ref{LRBO}
for $\tilde{\delta}_{i-1}=\mathcal{O}$, we have that $\tilde{\alpha}_{i}\in\{\mathcal{O},\mathcal{V}_{i-1}\}$.
Now given that \( \tilde{\alpha}_i \in \{ \mathcal{O}, \mathcal{V}_{i-1} \} \) and, in addition, \( \tilde{\delta}_i \notin \{ \mathcal{O}, \mathcal{U} \} \), it follows that \( \tilde{\alpha}_{i+1} \in \{ \mathcal{L}_i, \mathcal{R}_{i+2} \} \). However, to guarantee that $\tilde{\delta}_{i+1}=\mathcal{O}$, we must have $\tilde{\alpha}_{i+1}\in\{\mathcal{O},\mathcal{V}_i\}$, a contradiction.
For $i=1$ (and similarly when $i=r-2$), notice that for $\tilde{\delta}_{1}\notin\{\mathcal{O},\mathcal{U}\}$, we get that
$\tilde{\alpha}_{2} \in \{\mathcal{R}_3,\mathcal{L}_1\}$. However, \( \tilde{\delta}_{2} = \mathcal{O} \) implies that \( \tilde{\alpha}_{2} \in \{ \mathcal{O}, \mathcal{V}_1 \} \), which leads to a contradiction.

This proves that $|\{j\in\{1,\ldots r-2\}\,|\, \delta_j\neq id\}|>1$ and also shows that $\{j\,|\, \delta_j\neq id\}$ must be an interval.  
By Lemma~\ref{r0},  $n/2$ is odd and $\rho_0=(1,2)(3,4)\ldots(n-1,n)$.
\end{proof}

\begin{lemma}\label{Lnot3}
 Let $r\geq n/2\geq 7$, $|R\cup C|=1$ and $\langle L \rangle \ncong \Sym_{n/2}$.
  If $x=\mbox{min}\{j\in\{1,\ldots r-2\}\,|\, \delta_j\neq id\}$ and $h\geq x+3$ then $G_{<h}$ contains all even permutations fixing $B_1,\ldots, B_h$. 
The dual of this lemma also holds.
\end{lemma}
\begin{proof}
Suppose first that $x>1$. As in the previous proof we have,
$$\tilde{\delta}_{x-1}=\mathcal{O}\wedge \tilde{\delta}_x\neq \mathcal{O}  \Rightarrow  \tilde{\alpha}_x\in\{\mathcal{O},\mathcal{V}_{x-1}\} \wedge \tilde{\alpha}_{x+1}\notin\{\mathcal{O},\mathcal{V}_x\} \Rightarrow  \tilde{\delta}_x\in\{\mathcal{L}_{x+2}, \mathcal{R}_{x+2}\}$$
In any case $\delta_x\delta_x^{\rho_{x+2}}$ is the $2$-transposition swapping the points inside the blocks $B_{x+2}$ and $B_{x+3}$.
Hence  $G_{<x+3}$ contains all even permutations fixing $B_1,\ldots, B_{x+3}$. 

Consequently for  $h\geq x+3$,  $G_{<h}$ contains all even permutations fixing $B_1,\ldots, B_h$.  

Suppose that $x=1$. As, by Proposition~\ref{More} $\tilde{\rho}_0=\mathcal{U}$, we must have $\tilde{\delta}_1=\mathcal{R}_3$, hence $\delta_1\delta_1^{\rho_2}=(3,4)(5,6)$ hence we get the same result as for $x>1$.
\end{proof}

\begin{prop}\label{not3}
   Let $r\geq n/2\geq 7$, $|R\cup C|=1$ and $\langle L \rangle \ncong \Sym_{n/2}$.
  Then $\{j\in\{1,\ldots r-2\}\,|\, \delta_j\neq id\}=\{x,x+1\}$ for some $i\in\{1,\ldots,r-2\}$.
\end{prop}
\begin{proof}
Suppose that $x=\mbox{min}\{j\in\{1,\ldots r-2\}\,|\, \delta_j\neq id\}$ and $y=\mbox{max}\{j\in\{1,\ldots r-2\}\,|\, \delta_j\neq id\}$ and that $y>x+1$. Then, by Lemma~\ref{Lnot3} $G_{<x+3}$ contains all even permutations fixing $B_1,\ldots, B_{x+3}$. 
But also, as $x<y-1$, $G_{>x}$ contains all even permutations fixing $B_{x+1},\ldots, B_r$.  But then $G_{>x}\cap G_{<x+3}$ is not a dihedral group, contradicting the intersection property.
Hence $y=x+1$, as wanted.
\end{proof}

In what follows let $x$ be the index determined in the previous proposition, meaning that, $\delta_x$ and $\delta_{x+1}$ are the unique nontrivial $\delta$'s.


\begin{prop}\label{k=2notK=2}
 Let $r\geq n/2\geq 7$, $|R\cup C|=1$ and $\langle L \rangle \ncong \Sym_{n/2}$.
 Let $f : \langle L\rangle \rightarrow \Sym_{n/2}$ be the homomorphism induced by the action of $\langle L\rangle$ on the $n/2$ blocks.
  If ${\rm Ker}(f) = \Cyc_2^{n/2-1}$ then $\mathcal{G}$ is, up to duality,  one of the graphs of Table~\ref{KnotC2}.
\end{prop}
\begin{proof}
The following arguments make systematic use of Table~\ref{LRBO}.
Suppose first that $x\notin\{1,r-3\}$.

As $\tilde{\delta}_{x-1}=\tilde{\delta}_{x+2}=\mathcal{O}$, we must  have $\tilde{\alpha}_x\in \{\mathcal{O}, \mathcal{V}_{x-1}\}$ and $\tilde{\alpha}_{x+2}\in\{ \mathcal{O}, \mathcal{V}_{x+1}\}$. Moreover,
$$\tilde{\delta}_x\in\{ \mathcal{L}_{x+2}, \mathcal{R}_{x+2}\}\mbox{
and }
\tilde{\delta}_{x+1}\in\{\mathcal{L}_x, \mathcal{R}_x\}.$$

Let us consider the following cases separately: (A) \( x \) is even; and (B) \( x \) is odd.

Notice that $\mathcal{L}_i$ is even weight vector if and only if $i$ is even, while $\mathcal{R}_i$ is an even weight vector if and only if $i$ is odd.

(A) In this case  $\tilde{\delta}_x=\mathcal{L}_{x+2}$ and $\tilde{\delta}_{x+1}=\mathcal{L}_x$.

\[\begin{array}{cl}
\tilde{\delta}_x=\mathcal{L}_{x+2}&\Rightarrow (\tilde{\alpha}_x,\tilde{\alpha}_{x+1})\in \{( \mathcal{O},\mathcal{L}_x), (\mathcal{V}_{x-1}, \mathcal{R}_{x+2})\}\\
\end{array}\]

If $(\tilde{\alpha}_x,\tilde{\alpha}_{x+1})= (\mathcal{O},\mathcal{L}_x)$ we get the following.
\[\left\{\begin{array}{l}
\tilde{\alpha}_i=\mathcal{O}, \;i\neq x+1\\
\tilde{\alpha}_{x+1}=\mathcal{L}_x\\
\end{array}\right.\]

If $(\tilde{\alpha}_x,\tilde{\alpha}_{x+1})=(\mathcal{V}_{x-1}, \mathcal{R}_{x+2})$ we get the following.
\[\left\{\begin{array}{l}
\tilde{\alpha}_i=\mathcal{V}_{i-1}, \;i\neq x+1\\
\tilde{\alpha}_{x+1}=\mathcal{R}_{x+2}\\
\end{array}\right.\]

Then, when $x$ is even, we get two possibilities for  $\mathcal{G}$ corresponding to graphs (25) and (26) of  Table~\ref{KnotC2}.\\

(B) In this case  $\tilde{\delta}_x=\mathcal{R}_{x+2}$ and $\tilde{\delta}_{x+1}=\mathcal{R}_x$.
\[\begin{array}{cl}
\tilde{\delta}_x=\mathcal{R}_{x+2}&\Rightarrow (\tilde{\alpha}_x,\tilde{\alpha}_{x+1})\in \{( \mathcal{O},\mathcal{R}_{x+2}), (\mathcal{V}_{x-1}, \mathcal{L}_x)\}\\
\end{array}\]

If $(\tilde{\alpha}_x,\tilde{\alpha}_{x+1})= (\mathcal{O},\mathcal{R}_{x+2})$ we get the following.
\[\left\{\begin{array}{l}
\tilde{\alpha}_i=\mathcal{O}, \;i\neq x+1\\
\tilde{\alpha}_{x+1}=\mathcal{R}_{x+2}\\
\end{array}\right.\]

If $(\tilde{\alpha}_x,\tilde{\alpha}_{x+1})=(\mathcal{V}_{x-1}, \mathcal{L}_x)$ we get the following.
\[\left\{\begin{array}{l}
\tilde{\alpha}_i=\mathcal{V}_{i-1}, \;i\neq x+1\\
\tilde{\alpha}_{x+1}=\mathcal{L}_x\\
\end{array}\right.\]

Then, when $x$ is odd, we get two possibilities for  $\mathcal{G}$ corresponding to graphs (27) and (28) of  Table~\ref{KnotC2}.

Suppose that $x=1$. By Table~\ref{LRBO}, either  $\tilde{\delta}_1=\mathcal{R}_3$ and $\tilde{\delta}_2\in\{\mathcal{L}_1,\mathcal{R}_1,\mathcal{L}_4,\mathcal{T}_1, \mathcal{R}_4\}$. As $\mathcal{L}_1$ and $\mathcal{R}_4$ are odd permutations these can be excluded from the set of possibilities for $\tilde{\delta}_2$.
If $\tilde{\delta}_2\in\{\mathcal{T}_1,\mathcal{L}_4\}$ then $\tilde{\alpha}_3\notin\{\mathcal{O}, \mathcal{V}_2\}$, hence $\tilde{\delta}_3\neq \mathcal{O}$, a contradiction. This gives only one possibility  which is, $\tilde{\delta}_2=\mathcal{R}_1$.
Consequently $(\tilde{\alpha}_2,\tilde{\alpha}_3)\in \{(\mathcal{R}_3,\mathcal{O}),(\mathcal{L}_1,\mathcal{V}_2)\}$. If $(\tilde{\alpha}_2,\tilde{\alpha}_3)= (\mathcal{R}_3,\mathcal{O})$ then

\[\left\{\begin{array}{l}
\tilde{\alpha}_i=\mathcal{O}, \;i\neq 2\\
\tilde{\alpha}_2=\mathcal{R}_3\\
\end{array}\right.\]

 If $(\tilde{\alpha}_2,\tilde{\alpha}_3)= (\mathcal{L}_1,\mathcal{V}_2)$ then

\[\left\{\begin{array}{l}
\tilde{\alpha}_1=\mathcal{R}_2\\
\tilde{\alpha}_2=\mathcal{L}_1,\\
\tilde{\alpha}_i=\mathcal{V}_{i-1},\; i\geq 3\\
\end{array}\right.\]

Then, when $x=1$ (which is odd), we get two possibilities for  $\mathcal{G}$ corresponding to graphs (27) and (28) of  Table~\ref{KnotC2}.
For $x=r-3=n/2-3$ similar arguments give the possibilities (25) and (26) of  Table~\ref{KnotC2}.

\end{proof}


\section{Case: Imprimitive groups with two blocks.}\label{m=2}

In what follows let  $R$, $L$, $C$, $k$ and $m$ be as defined before and let $\{p_1,\dots, p_{r-1}\}$ be the Schl\"{a}fli symbol of $\Gamma$.
Now we deal with the case $m=2$  where $L$ is a singleton.
Let $\mathcal{B}=\{B_1,B_2\}$ denote the block system.
By Proposition~\ref{maximprim} we may assume that $r=n/2$.

\begin{prop}\label{R=0}
Let $r=n/2$. If  $|R|=0$, then $G\cong \Cyc_2\times \Sym_{n/2}$ and, up to duality, $p_1=2$ and $\Gamma_0$ is the automorphism group of a polytope of rank $(n/2-1)$ for $\Sym_{n/2}$.  If $n/2\geq 7$ then the  Schl\"{a}fli symbol of $\Gamma$ is $\{2,3,\ldots,3\}$ and $G$ admits one of the following two permutation representations:
$$
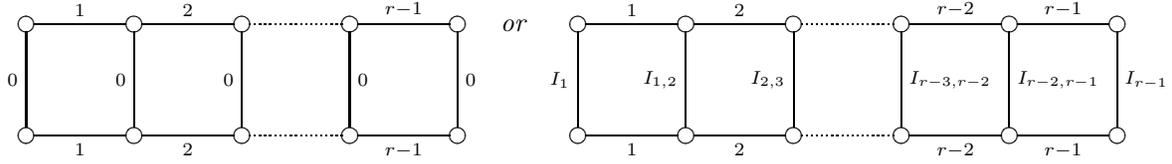
\xymatrix@-.8pc{
*+[o][F]{}\ar@{-}[dd]_0\ar@{-}[rr]^1&&*+[o][F]{}\ar@{-}[dd]_0\ar@{-}[rr]^2&&*+[o][F]{}\ar@{-}[dd]_0\ar@{.}[rr] &&*+[o][F]{} \ar@{-}[dd]^0\ar@{-}[rr]^{r-1}&& *+[o][F]{}   \ar@{-}[dd]^0 \\
&&&&&&&&\\
*+[o][F]{}\ar@{-}[rr]_1&&*+[o][F]{}\ar@{-}[rr]_2&&*+[o][F]{}\ar@{.}[rr]&&*+[o][F]{} \ar@{-}[rr]_{r-1}&& *+[o][F]{} }\;\mbox{ or }\;
\xymatrix@-.8pc{
*+[o][F]{}\ar@{-}[dd]_{I_1}\ar@{-}[rr]^1&&*+[o][F]{}\ar@{-}[dd]_{I_{1,2}}\ar@{-}[rr]^2&&*+[o][F]{}\ar@{-}[dd]_{I_{2,3}}\ar@{.}[rr]  &&*+[o][F]{} \ar@{-}[dd]^{I_{r-3,r-2}}\ar@{-}[rr]^{r-2}&&*+[o][F]{} \ar@{-}[dd]^{I_{r-2,r-1}}\ar@{-}[rr]^{r-1}&& *+[o][F]{}   \ar@{-}[dd]^{I_{r-1}} \\
&&&&&&&&\\
*+[o][F]{}\ar@{-}[rr]_1&&*+[o][F]{}\ar@{-}[rr]_2&&*+[o][F]{}\ar@{.}[rr]&&*+[o][F]{} \ar@{-}[rr]_{r-2}&&*+[o][F]{} \ar@{-}[rr]_{r-1}&& *+[o][F]{} }$$

 In the  first graph the blocks are the $G_0$-orbits while in the second graph each edge connects vertices in different blocks.
\end{prop}
\begin{proof}
In this case $\langle C\rangle$ acts faithfully on the pairs of points swapped by the element of $L$.
 Hence $|C|\leq n/2-1$. As $r=n/2$, $|C|=n/2-1$. 
 
Suppose first that $\langle C\rangle$ is intransitive. Then the permutation representation graph of $\langle C\rangle$ is given by two copies of one of the two first graphs given in Proposition~\ref{highC}. Hence, for $n/2\geq 7$, $\Gamma$  has the first permutation representation graph given in this proposition.

Let us now assume that $\langle C\rangle$ is transitive. Let $L=\{\rho_i\}$.
For $j\neq i$, let $\alpha_j=\rho_j\rho_i^{\tau}$ where $\tau=1$ if $B_1=B_2\rho_j$ and $\tau=0$ if $\rho_j$ fixes the blocks.
The set $\Lambda:=\{\alpha_j\,|\,i\in \{0,\ldots,r-1\}\}$ is independent, indeed if $\alpha_k\in\langle \alpha_j\,|\, j\neq k\rangle$, then $\rho_k\in G_k$, a contradiction.
Moreover $\langle \Lambda \rangle$ acts faithfully on  the $n/2$ pairs of points that are swapped by $\rho_i$, and $|\Lambda|=n/2-1$.
Hence by Corollary~\ref{CCcoro}, as $n/2\geq 7$, $\langle \Lambda\rangle$ is a string C-group having the first  permutation representation graph given in Proposition~\ref{highC}.
Particularly $\langle \Lambda\rangle\cong \Sym_{n/2}$, which implies that $i\in\{0,r-1\}$
and the order of the product of consecutive $\alpha_j$'s is $3$. 
Consider $i=0$ and let $j\in\{1,\ldots,r-1\}$ such that $\rho_j\in C$ swaps the blocks. 
Suppose that the consecutive generator $\rho_k\in C$, with $k\in\{j-1,j+1\}$ does not swap the blocks. As $(\alpha_j\alpha_k)^3 = id$, by the definition of $\alpha_j$,  $(\rho_j\rho_0\rho_k)^3 = id$. If $(\rho_j\rho_k)^3=id$, then we have $\rho_0 = id$, a contradiction. 
If $(\rho_j\rho_k)^3 \neq id$, then we have $\rho_0 = (\rho_j\rho_k)^3$, i.e. $\rho_0\in \langle C\rangle$, a contradiction. 
As $\langle C\rangle$ is transitive, hence the consecutive generator must also swap the blocks, implying that all generators of $C$ swap the blocks and the product of consecutive generators must also be $3$. 
The case when $i=r-1$ is equivalent. This gives, up to duality, the second possibility given in the statement of this proposition.
\end{proof}

\begin{prop}\label{m=2(1)}
  Let  $r=n/2\geq 7$. If  $\langle C\cup L\rangle$ is transitive  and $|R|\neq 0$ then $\langle C\rangle$ is transitive.
  \end{prop}
  \begin{proof}
  Suppose that $\langle C\cup L\rangle$ is transitive but $\langle C\rangle$ is intransitive.
  As $\langle C\rangle$ is a normal subgroup of $\langle C\cup L\rangle$, the $\langle C\rangle$-orbits are swapped by the element of $L$. 
  Particularly $\langle C\rangle$ must have exactly two orbits. 
  
  Notice that the $\langle C\rangle$-orbits do not need to be $B_1$ and $B_2$. Indeed the elements of $C$ do not need to fix the blocks of $\mathcal{B}$ and 
   the elements of $R$ do not need to preserve the $\langle C\rangle$-orbits.  
  Nevertheless if $L=\{\rho_i\}$ and $\delta\rho_i$ has a fixed point, then $B_1\delta=B_2$, which implies that  $\delta$ is fixed-point-free.
  
  The group generated by $C$ acts faithfully on the pairs of points swapped by the element of $L$.
  Thus $\langle C\rangle$ is embedded into $\Sym_{n/2}$. Let us deal separately with the cases $\langle C\rangle\cong \Sym_{n/2}$ and $\langle C\rangle\not\cong \Sym_{n/2}$  .

  Case 1:  $\langle C\rangle\cong \Sym_{n/2}$. 
  
  In this case an element of $R$ cannot fix the $\langle C\rangle$-orbits, if  $\delta\in R$ does fix the $\langle C\rangle$-orbits, then $\delta\delta^{\rho_i}\in \langle C\rangle$, a contradiction.

  Let us prove that in this case $|R|\neq 2$. Suppose the contrary, that $|R|=2$. 
  As the elements of  $C$ must be consecutive (because $\langle C\rangle\cong \Sym_{n/2}$),  up to duality, we may assume that $L=\{\rho_1\}$, $R=\{\rho_0,\rho_2\}$.
  In this case $\rho_0$ centralizes $\langle C\rangle$. 
  As $\rho_0\in R$,  as observed previously, $\rho_0$ cannot fix the $\langle C\rangle$-orbits. 
  Since $\rho_1$ also commutes with $\langle C\rangle$ and swaps the $\langle C\rangle$-orbits, there is only one possibility for $\rho_0$, which is $\rho_0=\rho_1$, a contradiction.
  Thus $|R|= 1$ and the element of $R$ cannot commute with all the elements of $C$. As $r= n/2$, we must have $|C|= n/2-2$.

  Up to duality, we may consider $L=\{\rho_0\}$, $R=\{\rho_1\}$ and $C=\{\rho_2,\ldots, \rho_{r-1}\}$.
  
  As $n/2\geq 7$, the permutation representation graph of $\langle C\cup L\rangle$ is, by Proposition~\ref{highC}, as follows.
  $$ 
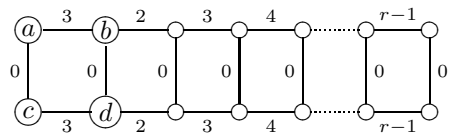
\xymatrix@-1.5pc{
  *+[o][F]{a}\ar@{-}[dd]_0\ar@{-}[rr]^3&&*+[o][F]{b}\ar@{-}[dd]_0\ar@{-}[rr]^2&&*+[o][F]{}\ar@{-}[dd]_0\ar@{-}[rr]^3&&*+[o][F]{}\ar@{-}[dd]_0\ar@{-}[rr]^4&&*+[o][F]{}\ar@{-}[dd]_0\ar@{.}[rr] &&*+[o][F]{} \ar@{-}[dd]^0\ar@{-}[rr]^{r-1}&& *+[o][F]{}   \ar@{-}[dd]^0 \\
  &&&&&&&&&&&&\\
  *+[o][F]{c}\ar@{-}[rr]_3&&*+[o][F]{d}\ar@{-}[rr]_2&&*+[o][F]{}\ar@{-}[rr]_3&&*+[o][F]{}\ar@{-}[rr]_4&&*+[o][F]{}\ar@{.}[rr]&&*+[o][F]{} \ar@{-}[rr]_{r-1}&& *+[o][F]{} }$$
  
  Let $\{a,b,c,d\}$ be as above.
  If $\rho_1$ acts nontrivially on $\{1,\ldots,n\}\setminus\{a,b,c,d\}$ then, $\mathcal{G}$ has $n/2-2$  $\{0,1\}$-edges, which implies that  $B_1\rho_1=B_2$.  Particularly $\rho_1$ is fixed-point-free. Then we get that $\rho_1$ commutes with $\rho_0$,  contradicting the definition of $C$. Thus  $\rho_1$ fixes $\{1,\ldots,n\}\setminus\{a,b,c,d\}$.
  Now to avoid a double $\{0,1\}$-edge and $\rho_1$ fixing the $\langle C\rangle$-orbits, let  $a\rho_1=d$, then, as $\rho_1$ and $\rho_3$ commute,  $b\rho_1=c$. But then $\rho_1$ and $\rho_0$ also commute, a contradiction.

  Case 2: $\langle C\rangle\not\cong\Sym_{n/2}$. 
  
  In this case, by Theorem~\ref{maxan} and Propositions~\ref{prim} and \ref{maximprim}, $|C|\leq (n/2)/2+1$.
  But as $r\leq |C|+3$ and $r\geq n/2$, we get that $n/2\leq 8$ (and $n/2\geq 7$). Hence $n\in\{7,8\}$ which implies that  $|C|\leq (n/2)/2$. Consequently, $n/2\leq  |C|+3\leq n/4+3$ gives $n\leq 12$, a contradiction.

  \end{proof}
  
  \begin{lemma}\label{CLprim}
  Let $r\geq n/2-2$, $n/2\geq 7$ and $n\neq 16$. Suppose that $\Phi=(H,\{\alpha_0, \ldots, \alpha_{r-1}\})$ is a string C-group satisfying the following:
  \begin{itemize}
  \item $H_0$ is transitive; 
  \item $(\alpha_0\alpha_1)^2=id$; and 
  \item  $H$ has a block system $\mathcal{B}=\{B_1,B_2\}$ with $B_1=B_2\alpha_0$.
  \end{itemize}
  Then $H$ has a primitive action on the $\langle \alpha_0\rangle$-orbits.   
  \end{lemma}
  \begin{proof}
  As  $(\alpha_0\alpha_1)^2=id$,  $\alpha_0$ is a central involution. 
  As in addition $H_0$ is transitive, the $\langle \alpha_0\rangle$-orbits  form a block system for $H$.
  Suppose that the action of $H$ on the $\langle \alpha_0\rangle$-orbits is imprimitive.
  Then  there exist a block system $\mathcal{V}$ with $m$ blocks of size $k$  such that $H\leq \Sym_k\wr \Sym_m$ with $n=km$ and such that $\alpha_0$ fixes the blocks. 
 It follows that $k$ is even and $k\geq 4$. 
  Let us also consider $\mathcal{V}$, with $k$ being maximal, that is, such that the action of $H$ on  $\mathcal{V}$ is primitive.
  
  Now let $L$ be a subset of $\{\alpha_0, \ldots, \alpha_{r-1}\}$ generating independently the action on the $m$ blocks. 
    As  $\langle L\rangle$ has a primitive action on the $m$ blocks, hence by Proposition~\ref{LCR}(d), as $r\geq n/2-2$, the elements of $L$ are consecutive.
  Let $C$ be the subset of $\{\alpha_0, \ldots, \alpha_{r-1}\}$ that commute with all the elements of $L$ and $R$ be the remaining generators of $H$. Notice that  $\alpha_0\in C$ and $|R|\leq 2$.
  
   Let $\bar{\alpha}_i$, for $i\neq 0$,  be the action of $\alpha_i$ on the $\langle \alpha_0\rangle$-orbits. 
   The set $\{\bar{\alpha}_i, i=1,\ldots,r-1\}$ is independent (similarly to the set $\Lambda$ that was considered in the proof of Proposition~\ref{R=0}).
   Let $\bar{L}=\{\bar{\alpha_i}\,|\, \alpha_i\in L\}$, $\bar{R}=\{\bar{\alpha_i}\,|\, \alpha_i\in R\}$ and $\bar{C}=\{\bar{\alpha_i}\,|\, \alpha_i\in C\setminus \{\alpha_0\}\}$.
   
   We claim that $|C|\leq k/2$. Indeed if $\langle C\rangle$ fixes the blocks  $|\bar{C}|\leq k/2-1$, which implies that  $|C|\leq k/2$.
   If an element of $C$ swaps the blocks, then, as  $\langle L\rangle$ is primitive,  $m=2$. In this case the elements of $C\setminus \{\alpha_0\}$ act independently on the $\langle L\cup\{\rho_0\}\rangle$-orbits, which have exactly four points.
   Hence $|C\setminus \{\alpha_0\}|\leq n/4-1=k/2-1$. Thus we also get what we want, that is $|C|\leq k/2$.
  
  Hence we have the following bound for $r$.
  $$r=|C\cup L\cup R|\leq k/2+m-1+2.$$
  Consequently, $ k/2+m+1\geq km/2-2$, which gives $(k-2)(m-1)\leq 8$.

  As $n/2\geq 7$, $n\neq 16$, $k$ is even and $k\geq 4$ we need only to consider the following possibilities:
 $(k,m)=(4,5)$;  $(k,m)=(6,3)$ or $(k,m)=(10,2)$. Let us analyse each of them separately.
  
   $\bullet\,(k,m)=(4,5)$: In this case $|C|=2$ and $\langle C\rangle$ fixes the blocks. Let $C=\{\alpha_0,\alpha_1\}$.  If $|L|=n/4-1=4$ then $\langle \bar{L}\cup\{\bar{\alpha}_1\}\rangle\cong \Cyc_2\times \Sym_5$ and $\langle \bar{L}\cup\{\bar{\alpha}_1\}\cup \bar{R}\rangle\leq \Cyc_2\wr \Sym_5$. Thus $|R|=|\bar{R}|\leq 1$ which gives $r=|L|+|C|+|R|\leq 4+2+1=7<n/2-2$, a contradiciton. 
  
  $\bullet\,(k,m)=(6,3)$ : In this case, as $r\geq 7$, we must have $|L|=2$, $|C|=3$ and $|R|=2$. Moreover we may assume that $C=\{\alpha_0,\alpha_1,\alpha_2\}$, $R=\{\alpha_3,\alpha_6\}$ and $L=\{\alpha_4,\alpha_5\}$. Let $\mathcal{V}=\{\mathcal{V}_1,\mathcal{V}_2,\mathcal{V}_3\}$ be the block system where $\mathcal{V}_1\alpha_4 = \mathcal{V}_2$ and $\mathcal{V}_2\alpha_5 = \mathcal{V}_3$. Then $\langle C\rangle$ must  fix the blocks of $\mathcal{V}$, moreover the $\langle C\rangle$-orbits are precisely the  blocks of $\mathcal{V}$.
  Then $\langle C\rangle \cong \Cyc_2\times \Sym_3$, and it has two possible permutation representations determined by the transitive and the intransitive representation of  $\Sym_3$ on $6$ points.   
  Now $G_{3,5,6}$ has exactly two orbits, $\mathcal{V}_1\cup \mathcal{V}_2$ and $\mathcal{V}_3$,  one of size $12$ and the other one of size $6$. 
  Thus $\alpha_6$ must fix these two sets. Moreover, if $\mathcal{V}_1\alpha_6 = \mathcal{V}_2$, then we can redefine $L$, say  $L = \{\alpha_5,\alpha_6\}$, and then $|R|<2$, giving $r<n/2-2$, a contradiction.
Hence, $\alpha_6$ must fix each block of $\mathcal{V}$.  Then $\alpha_6$ must swap the blocks of $\mathcal{B}$ which forces the equality $\alpha_6=\alpha_0$, a contradiction.
  
  $\bullet\,(k,m)=(10,2)$: In this case we have $|L|=1$ and, as $r\geq n/2-2=8$, we must  have $|C|=n/4=5$  and $|R|=2$. 
Let $\mathcal{V}=\{\mathcal{V}_1,\mathcal{V}_2\}$ be the block system where the blocks are swapped by the element in $L$. 
 In this case, $\langle C\rangle\cong \Cyc_2\times \Sym_5$ and  $\langle C \cup L\rangle \cong \Cyc_2^2\times \Sym_5$. 
  Then we may assume that all elements of $C$ are consecutive and that  the last generator of $G$ belongs to $R$, that is $\alpha_7\in R$. Furthermore $\alpha_7$ commutes with all the elements of $C$. Particularly $\langle C \cup \{\alpha_{r-1}\}\rangle \cong C_2^2\times \Sym_5$.
 If it swaps the blocks of $\mathcal{V}$ then we can make a different choice for the element of $L$ giving  $|R|<2$ and $r<8$, a contradiction. Thus  $\alpha_7$ fixes the blocks.
 
  First note that  $\langle C \rangle$ and $\rho_7$ cannot  both fix the blocks of $\mathcal{V}$, as in that case we would get an intransitive permutation representation of  $\Cyc_2^2\times \Sym_5$ with two orbits of size $n/2=10$, which is impossible. Indeed it can be checked computationally that the minimal transitive degree of   $\Cyc_2^2\times \Sym_5$  is greater than $10$. Thus $\langle C \rangle$ must be transitive. 
  Hence, one generator of $C\backslash\{\alpha_0\}$ swaps the blocks of $\mathcal{V}$. Then we can consider other elements for $L$ giving  $|C|<5$, and therefore $r<8$, a contradiction.
    
  \end{proof}
  
  \begin{lemma}\label{coroprimitive}
   Let $r\geq n/2\geq 7$.  If  $\langle L \cup C\rangle$ is transitive and $|R|\neq 0$, then $\langle L \cup C\rangle$ has a primitive action on the $\langle \rho_i\rangle$-orbits, where $L=\{\rho_i\}$.
  \end{lemma}
  \begin{proof}
   As $\langle L \cup C\rangle$ is transitive and $|R|\neq 0$, by Proposition~\ref{m=2(1)},  $\langle C\rangle$ is transitive. 
   Suppose first that  $n\neq 16$. Then  by Lemma~\ref{CLprim}, $\langle L \cup C\rangle$ has a primitive action on the $\langle \rho_i\rangle$-orbits, where $L=\{\rho_i\}$.
   Now let $n=16$ and suppose that $\langle L \cup C\rangle$ acts imprimitively on the $\langle \rho_i\rangle$-orbits. The action of  $\langle C\rangle$ on the $\langle \rho_i\rangle$-orbits is faithfull, hence $\langle C\rangle$ is a string C-group representation of a transitive group of degree $8$. Thus by Proposition~\ref{maximprim} $|C|\leq 8/2+1=5$, moreover as $r\geq 16/2=8$ we must have $|C|=5$ and $|R|=2$. Hence  $\langle C\rangle$ is the automorphism group of a polytope with Schl\"{a}fli  symbol $\{3,4,4,3\}$ given in Table~\ref{imprimPolys}. Let $C=\{\rho_0,\ldots,\rho_4\}$, $R=\{\rho_5,\rho_7\}$ and $L=\{\rho_6\}$.  The permutation representation of $G_{5,7}$, can be determined computationally and is as follows:
    \begin{small}
 \[ \begin{array}{ll}
       \rho_0 = ( 1,10)( 2, 9)( 3,12)( 4,11)( 5,16)( 6,15)( 7,14)( 8,13) & \rho_3= ( 1, 9)( 2,10)( 3,13)( 4,14)( 5,15)( 6,16)( 7,11)( 8,12)\\
       \rho_1 = ( 1,10)( 2, 9)( 3,14)( 4,13)( 5,12)( 6,11)( 7,16)( 8,15) & \rho_4 = ( 1, 9)( 2,10)( 3,11)( 4,12)( 5,13)( 6,14)( 7,15)( 8,16)\\
      \rho_2 = ( 1, 3)( 2, 4)( 5, 7)( 6, 8)( 9,11)(10,12)(13,15)(14,16) &  \rho_6= ( 1, 2)( 3, 4)( 5, 6)( 7, 8)( 9,10)(11,12)(13,14)(15,16)
    \end{array}\]
\end{small}
Now $\rho_7$ is an involution commuting with all the elements of $C$.
This gives only one possibility for $\rho_7$ which is $\rho_7=\rho_6$, a contradiction.
   Therefore, $\langle L \cup C\rangle$ has a primitive action on the $\langle \rho_i\rangle$-orbits.
  \end{proof}
  

  \begin{prop}
   Let  $r=n/2\geq 7$.  Then $\langle C\cup L\rangle$ is transitive if and only if $|R|=0$.
   \end{prop}
  
  \begin{proof} 
  Suppose that $\langle C\cup L\rangle$ is transitive and $|R|\neq 0$.
  By Proposition~\ref{m=2(1)} $\langle C\rangle$ is transitive.
  Let $C=\{\rho_j\,|\, j\in I\}$ and $L=\{\rho_i\}$.
  By Lemma~\ref{coroprimitive}, $\langle L\cup C\rangle$  has a primitive action on the $\langle\rho_i\rangle$-orbits. 
  In addition, notice that the elements of $C$ generate independently the action on the  $\rho_i$-orbits. 
  Hence if two consecutive elements of $C$ commute then $\langle C\rangle$ can be factorized has a direct product of two groups that act regularly on the $n/2$ blocks. 
  Hence $|\langle C\rangle|=(n/2)^2$ and, as  $\langle C\rangle$ is primitive, it follows that $n/2\geq 60$. 
  Consequently $|C|\leq 2\log_2(n/2)\leq n/2-3$, giving a contradiction.
  Thus $I$ is an interval.
  Now let us deal separately with the cases $|R|=2$ and $|R|=1$. 
  
  $\bullet\,|R|=2$: We may assume that $R=\{\rho_0,\rho_2\}$, $L=\{\rho_1\}$ and $C=\{\rho_3,\ldots,\rho_{r-1}\}$.
  Then both $\rho_0$  and $\rho_1$ centralize $\langle C\rangle$. 
  But  $\langle C\rangle$ has a primitive action on the $\langle \rho_1\rangle$-orbits, hence $\rho_0=\rho_1$, a contradiction.
  
  $\bullet\,|R|=1$: To avoid the previous contradiction the element of $R$ cannot centralize $\langle C\rangle$. Thus let  $L=\{\rho_0\}$, $R=\{\rho_1\}$ and $C=\{\rho_2,\ldots,\rho_{r-1}\}$. 
  Let us suppose that $G_{1,i}$ is transitive for $i\in\{2,\ldots,r-1\}$. If $n\neq 16$, then  $G_{1,i}$ satisfies the conditions of Lemma~\ref{CLprim}, hence $G_{1,i}$ has a primitive action on the $\langle\rho_0\rangle$-orbits.
  Moreover $C\setminus \{\rho_i\}$ generate independently the action of $G_{1,i}$  on the $\langle\rho_0\rangle$-orbits.
  As $|C\setminus \{\rho_i\}|=n/2-3$ then, by Proposition~\ref{maroti2}, the action on the blocks is either $\Sym_{n/2}$ or $\Alt_{n/2}$.
  Also the action of $\langle C\rangle$ on the $\langle\rho_0\rangle$-orbits must be one of these groups, giving a contradiction.
   Hence $n=16$. As $G_1 =\langle L\cup C\rangle$ is transitive, by Lemma~\ref{coroprimitive},  $G_1$ has a primitive action on the $\langle\rho_0\rangle$-orbits.
  In particular, the group action of $G_1$ on the $\langle\rho_0\rangle$-orbits is a transitive group of degree $8$ of rank $|C|=16/2 - 2 = 6$.
  As the action of $\langle C\rangle$ on the $\langle \rho_0\rangle$-orbits is faithful (the identity is the unique element of  $\langle C\rangle$ fixing the $\langle \rho_0\rangle$-orbits), it follows that $\langle C\rangle$ is a rank $6$ string C-group for a primitive group of degree $8$.
  By Proposition~\ref{prim} and Theorem~\ref{maxan}, we conclude that the group action of $\langle C\rangle$ on the $\langle \rho_0\rangle$-orbits must be $\Sym_{8}$.
 Moreover, its permutation representation graph is one of those given in Proposition~\ref{highC}.
 This implies also that $\langle C\rangle \cong \Sym_{8}$ (as the action of $\langle C\rangle$ on the $\langle \rho_0\rangle$-orbits is faithful). Hence, since $\langle C\rangle$ is transitive on $16$ points, $\langle C\rangle$ has the permutation representation graph given in Proposition~\ref{2npointsrep}, which is as follows:
  $$\xymatrix@-1pc{
          *+[o][F]{}\ar@{-}[rr]^3 \ar@{-}[dd]^{I_{1,3}} && *+[o][F]{}\ar@{-}[rr]^2 \ar@{-}[dd]^{I_{1,2,3}}&&*+[o][F]{}\ar@{-}[dd]^{I_{1,2,3}}\ar@{-}[rr]^3 &&*+[o][F]{}\ar@{-}[dd]^{I_{1,3,4}}\ar@{-}[rr]^4 &&*+[o][F]{}\ar@{-}[dd]^{I_{1,3,4,5}}\ar@{.}[rr] &&*+[o][F]{} \ar@{-}[dd]^{I_{1,3,r-2,r-1}}\ar@{-}[rrr]^{r-1}&&& *+[o][F]{}  \ar@{-}[dd]^{I_{1,3,r-1}} \\
          &&&&&&&&\\
          *+[o][F]{}\ar@{-}[rr]_3&&*+[o][F]{}\ar@{-}[rr]_2&& *+[o][F]{}\ar@{-}[rr]_3 && *+[o][F]{}\ar@{-}[rr]_4 &&*+[o][F]{}\ar@{.}[rr]&&*+[o][F]{} \ar@{-}[rrr]_{r-1}&&& *+[o][F]{} }
   $$      
  Computationally one can check that  if there is an involution $\rho_1$ in $\Sym_8\wr \Cyc_2$ commuting with all the elements of  $C\setminus\{\rho_2\}$, then it must belong to $\langle L \cup C\rangle$, a contradiction.
  Therefore, the case $n=16$ cannot happen.
  
  Hence $G_{1,i}$ is intransitive for every $i\in\{2,\ldots,r-1\}$.
  Let $\Phi=\{\delta_2,\ldots,\delta_{r-1}\}$ where $\delta_i$ is the action of $\rho_i$ on the $\rho_0$-orbits.
  As $\Phi$ is a sggi of rank $n/2-2$ and $\Phi_i$ is intransitive for every $i\in\{2,\ldots,r-1\}$, by Proposition~\ref{sggiint}, up to duality, $\Phi$ is a string C-group having the permutation representation given at the end of Proposition~\ref{highC}.
  Now if $\rho_1$ fixes the blocks $B_1$ and $B_2$ then $\rho_1\rho_1^{\rho_0}\in \langle C\cup L\rangle$, a contradiction.
  If $B_1=B_2\rho_1$ then as $\rho_1$ commutes with $\rho_i$ $i>2$, but, as in Proposition~\ref{m=2(1)}, this forces $\rho_1$ to commute also with $\rho_0$, giving a contradiction with the definition of $C$.
  
  Now let $|R|=0$. As $G$ is a transitive group and $G = \langle C\cup L\rangle$, then $\langle C\cup L\rangle$ is transitive.
  \end{proof}
  
  \begin{prop}\label{aallint}
    Let $r\geq n/2\geq 7$.
    If $\langle C\cup L\rangle$ is intransitive  and $L=\{\rho_i\}$ then $G_j$ is intransitive for $j\notin\{0,i,r-1\}$.
    \end{prop}
    \begin{proof}
    Suppose that $G_j$ is transitive.
    In this case $G_j=G_{<j}\times G_{>j}$.
    Suppose that $\{ \rho_0,\ldots,\rho_{j-1}\}\not\subseteq C$ and $\{ \rho_{j+1},\ldots,\rho_{r-1}\}\not\subseteq C$. 
    Then $i-1\leq j-1$ and $i+1\geq j+1$, which give $j=i$, a contradiction.
    Hence, either $\{ \rho_0,\ldots,\rho_{j-1}\}\subseteq C$ or  $\{ \rho_{j+1},\ldots,\rho_{r-1}\}\subseteq C$.
    Suppose, without loss of generality  that $\{ \rho_0,\ldots,\rho_{j-1}\}\subseteq C$.
    As $\langle C\cup L\rangle$ is intransitive,  $G_{<j}$ is also intransitive. Therefore the $G_{<j}$-orbits determine a block system for $G_j$.
    Now either $\rho_i$ fix all $G_{<j}$-orbits or swaps all of them pair wise. 
    When $\rho_i$ swaps two $G_{<j}$-orbits, say $O_1$ and $O_2$, then $O_1\cup O_2$ is a block of another block system whose blocks are twice bigger.
    Notice that, as $\{ \rho_0,\ldots, \rho_{j-1}, \rho_i\}\subseteq L\cup C$, $O_1\cup O_2 \neq \{1,\ldots,n\}$. 
    Consider  a maximal block system such that $G_{<j}$ fixes the blocks. Then $\rho_i$ fixes the blocks.
    Particularly the maximality of the blocks implies the action on the blocks is primitive.
    Let $k'$ and $m'$ be the size of a block and the number of blocks, respectively.
    
    Consider first the case $m'>2$.
      As the permutations $\rho_0,\ldots, \rho_{j-1}$ and $\rho_i$ fix the blocks, $k'$ is even and $k'\geq 4$.
    Consider the sets $L'$,  generating independently block action; $C'$, the set of generators of $G_j$ that commute with all the elements of $L'$; and $R'$, the set of the remaining generators of $G_j$.
      It follows that $|L'|\leq m'-1$.  As $\langle L'\rangle$ is primitive we may assume that the elements of  $L'$ are consecutive, hence $|R'|\leq 2$. 
    
    In this case $\langle C'\rangle$ is  an imprimitive group, with two blocks, embedded into $\Sym_{k'}$.  Hence, by Proposition~\ref{maximprim},  $|C'|\leq k'/2$. 
    Consequently, when $m'\neq 2$, $m'k'/2-1=r-1\leq (m'-1)+k'/2+2$, or equivalently, $(m'-1)(k'-2)\leq 6$, as $n>12$, this is only possible if $(m',k')=(4,4)$.
    Now it remains to consider the cases $(m',k')=(4,4)$ and $m'=2$. Let us deal with them separately.
    
    $(m',k')=(4,4)$: In this case $|L'|= 3$, $|C'|=2$ and $|R|=2$.
    Suppose first that $\rho_i\in C'$. In this case $j=1$ and $C'=\{\rho_0,\rho_i\}$. Thus the elements of $L'$ commute with $\rho_i$, which mean that $L'\subseteq C$.
    As $\langle L'\cup C'\rangle$ is transitive, we get that $\langle L\cup C\rangle$ is transitive, a contradiction.
    Thus $\rho_i\notin C'$. Hence $\rho_i\in R'$.  Let $C'=\{\rho_0,\rho_x\}$. As $\rho_i$ fixes the blocks (of size $4$) 
    and $B_1\rho_i=B_2$,  $\langle \rho_0,\rho_x,\rho_i\rangle\cong (\Cyc_2)^3$ (and $\rho_0$ also commutes with the elements of $L'$).
    Thus $\rho_0$ is a central involution in $G_j$.
    For each $l\neq j$, let $\bar{\rho_l}$ denote the action of $\rho_l$ on the $\langle \rho_0\rangle$-orbits.
    The set $\{\bar{\rho_x},\, \bar{\rho_i}\}\cup \{\bar{\rho}_y\,|\, \rho_y\in L'\}$ is independent and has size $5$. 
    Moreover $H:=\langle\bar{\rho_l}\,|\, l\neq 0,\,j\rangle$ is a transitive subgroup of $\Sym_2\wr \Sym_4$ whose block action is $\Sym_4$, while $\bar{\rho_x}$ and  $\bar{\rho_i}$ fix the blocks ($\bar{\rho_x}$ is a central involution).
    Thus $H_x$ and $H_i$ are transitive subgroups of $\Sym_2\wr \Sym_4$ whose block action is $\Sym_4$.
    As the number of blocks is even ($4$ blocks), we necessarily have $\bar{\rho_x}\in H_x$, a contradiction.
    
    $m'=2$: 
    Let $L'=\{\rho_x\}$ and $B'_1$ and $B'_2$ be the blocks swapped by $\rho_x$.
    As $\rho_i$ fixes $B'_1$ and $B'_2$, $G_j\leq \Sym_{n/4}\wr \Sym_4$, to be precise the block action is $\Cyc_2\times \Cyc_2$. The set  $\{\rho_i,\rho_x\}$ generate the block action, which can be described by an alternating $\{i,x\}$-square, meaning that $(\rho_i\rho_x)^2$ acts as the identity on the blocks.
     For this embedding the block action is generated by $\rho_i$ and $\rho_x$. Let $\mathcal{V}=\{\mathcal{V}_1,\mathcal{V}_2,\mathcal{V}_3,\mathcal{V}_4\}$ denote the  block system mentioned above with  $\mathcal{V}_1\rho_i=\mathcal{V}_2$, $\mathcal{V}_3\rho_i=\mathcal{V}_4$, $\mathcal{V}_1\rho_x=\mathcal{V}_3$ and $\mathcal{V}_2\rho_x=\mathcal{V}_4$. 
     Let $M=\{\rho_i, \rho_x\}$, $D$ be the set of generators of $G_j$ that commute with $\rho_i$ and $\rho_x$ and $N$ be the generators of $G_j$ that are neither in $D$ nor in $M$.   
     It follows that $|D|\leq n/4-1$ by Lemma~\ref{C4}.  In this case there are at most four generators that do not commute with both $\rho_i$ and $\rho_x$, hence $|N|\leq 4$.
    Thus $r-1=|D|+|M|+|N|\leq (n/4-1)+2+4$, giving a contradiction for $n> 24$. 
    We need to consider $n\in\{16,20,24\}$.
    
    Let $p$ be order of $\rho_i\rho_{i+1}$, which is even.
    In this case  $(\rho_i\rho_{i+1})^{p/2}$ is a nontrivial permutation fixing the four blocks.
    Moreover this permutation commutes with both $\rho_{i+1}$ and $\rho_i$. As  $(\rho_i\rho_{i+1})^{p/2}\notin \langle D\rangle$,
    we  conclude that $|D|\leq n/4-2$.  Consequently  $r-1=|D|+|M|+|N|\leq (n/4-2)+2+2$, which gives $n\leq 12$, a contradiction.
    We get the same conclusion if  $\{\rho_i,\rho_{i-1}\}$ generate the action on the blocks $\{\mathcal{V}_1,\mathcal{V}_2,\mathcal{V}_3,\mathcal{V}_4\}$ and $(\rho_i\rho_{i-1})^2$ is nontrivial.
    
    Now consider that $M=\{\rho_i,\rho_x\}$ and $\rho_x\in C$. 
    Let $q$ be the size of a $G_{<j}$-orbit.
    Having in mind  that $q$ divides $n/2$,  let us consider all the possibilities for $q$.
    
    $\bullet\,q=n/2$:
    If $q=n/2$ then $O_1=\mathcal{V}_1\cup \mathcal{V}_2$ and $O_2=\mathcal{V}_3\cup \mathcal{V}_4$ would be the $G_{<j}$-orbits, but then $\langle \rho_0,\ldots,\rho_{j-1},\rho_x\rangle$ would be transitive, a contradiction.
    
    $\bullet\,q=n/4$: 
    In this case the $G_{<j}$-orbits cannot be $\mathcal{V}_1,\, \mathcal{V}_2,\,\mathcal{V}_3$ and $\mathcal{V}_4$, otherwise $\langle\rho_0,\ldots, \rho_{j-1},\rho_x,\rho_i\rangle$ would be transitive, a contradiction.
    Thus, in this case,  $\rho_i$ must fix the $G_{<j}$-orbits. Particularly $n/4$ is even (either $n=16$ or $n=24$). 
    Suppose that $q=6$. In this case either $G_{<j}=\langle D\rangle$  or $G_{<j}$ is  in the center $\langle D\rangle$. In the first case $|D|\leq 3$ and in the second case, by Lemma~\ref{C4}, $|D|\leq n/4-2=4$, hence $r-1\leq  4+2+4$, giving $r<n/2$, a contradiction.
    Thus $q=4$.  As  $\rho_i$ fix the $G_{<j}$-orbits, $j=2$. Moreover $\langle \rho_0,\rho_1\rangle=\Cyc_2\times\Cyc_2$, thus by Lemma~\ref{C4},  $r-3\leq 16/4-1$, which gives a bound below $n/2$.
    
    $\bullet\,q=n/6$ or $q=3$: In this case $n=24$. Consider first the case $G_{<j}$ fixing $\mathcal{V}_1$, $\mathcal{V}_2$, $\mathcal{V}_3$ and $\mathcal{V}_4$.  Recall that  $G_{<j}=\langle D\rangle$  or $G_{<j}$ is  in the center $\langle D\rangle$.  
    In the first $\langle D\rangle$ is intransitive.  In the second case  $\langle D\rangle$ is not isomorphic to $\Sym_{n/4}$.  In both cases we get that $|D|\leq n/4-2$, and as $n=24$, $r<n/2=12$. 
    Thus, $\rho_i$ fixes the $G_{<j}$-orbits, which implies that $q=4$. Then, as before $j=2$ and, $G_{<2}=\Cyc_2\times\Cyc_2$. Thus by Lemma~\ref{C4},  $r-3\leq 24/4-1=7<24/2$.
    
     $\bullet\, q\in \{n/8, n/10,n/12\}$ and $q\neq 3$:  In this case $q=2$ and $j=1$. Thus $\rho_0$ is a central involution of $G_1$. If $\mathcal{V}_1\rho_0=\mathcal{V}_2$ (hence  $\mathcal{V}_3\rho_0=\mathcal{V}_4$) then if we consider  $M'=\{\rho_0,\rho_x\}$ for the generators on $\mathcal{V}$, $D'$ the set of generators commuting with $M'$ and $N'$ the generators of $G_j$ not in $M'\cup D'$,  we get $|N'|\leq 2$. In addition $\rho_0\rho_i$ fixes the blocks and commutes with all the elements of $M'$. Thus $|D'|<n/4-1$, and $r-1=|D'|+|M'|+|N'|\leq (n/4-2)+2+2$, which gives $n\leq 12$, a contradiction.
    
    Thus $\rho_0$ is a central involution fixing the blocks of $\mathcal{V}$. Hence $G_j\leq \Sym_{n/8}\wr (\Cyc_2)^3$ and Lemma~\ref{C4} implies that  $r-4\leq n/8-1$, giving $r<n/2$.
    \end{proof}
  
 
For the remaining of this section we assume the following:
\begin{itemize}
\item $n/2\geq 7$;
\item  $\langle L \cup C\rangle$ is intransitive;
\item $L=\{\rho_i\}$;
\item $|R|>0$; and 
\item $G_j$ is intransitive for $j\notin\{0,i,r-1\}$.
\end{itemize}
To determine  the remaining possibilities  for $\mathcal{G}$, consider the graphs $\mathcal{I}$ and   $\mathcal{F}$ given by the following construction.
  \begin{construction}~\label{bob}
  Consider a graph $\mathcal{I}$ whose vertices are the $\langle \rho_i\rangle$-orbits denoted by $O_1,\,\ldots O_{n/2}$, and with a $j$-edge ($j\neq i$) when $a\rho_j=b$ for $a$ and $b$ in different $\langle \rho_i\rangle$-orbits. 
  The graph $\mathcal{I}$ has the following properties: 
  \begin{itemize}
  \item[P1:] Adjacent edges, or parallel edges, of $\mathcal{I}$ either correspond to adjacent edges, or parallel edges, of $\mathcal{G}$ or to a pair of edges that are adjacent to a common $i$-edge;
  \item[P2:] If a $j$-edge and a $l$-edge are adjacent in $\mathcal{I}$ but not in $\mathcal{G}$ then $\{j,l\}\subseteq \{i-1,i+1\}$; and
  \item[P3:] If two $j$-edges of  $\mathcal{I}$ are adjacent  or parallel, then $j=i\pm1$. Particularly, $\mathcal{I}$ has a cycle with edge having the same label, if and only if,  $\mathcal{G}$ has a $\{i,j+1\}$-cycle.
  \end{itemize}
  Now consider a  generalization of the concept of a fracture graph\footnote{This is defined as a spanning subgraph of the permutation graph that contains exactly one \( i \)-edge \( \{a, b\} \) for each \( i \in \{0, \ldots, r-1\} \), with the property that \( a \) and \( b \) lie in different \( G_i \)-orbits.}
 introduced in \cite{extension}. 
  Let $\mathcal{F}$ be a spanning subgraph of $\mathcal{I}$ with exactly one $j$-edge $\{O_s,O_t\}$ for each label $j$ chosen among the $j$-edges connecting $\langle \rho_i\rangle$-orbits that belong to different $G_j$-orbits.
  The number of edges of  $\mathcal{F}$ is the cardinality of the set $\{j\,|\, G_j \mbox{ is intransitive}\ \wedge\ j\neq i\}$. 
  \end{construction} 
 
 Some of the key properties of fracture graphs, namely Lemmas 3.2, 3.3, 3.5 and 3.6 of \cite{extension}, also hold for a graph  $\mathcal{F}$ obtained by the construction above. 
  
  \begin{lemma}\label{FG}
  \begin{enumerate}
  \item  Let $f=|\{j\,|\, G_j \mbox{ is intransitive} \ \wedge\ j\neq i\}|$. The graph $\mathcal{F}$ is a forest with $n/2-f$ connected components.
  \item If there exist two edges $\{ O_s, O_t \} $ with distinct labels $j$ and $l$ in $\mathcal{I}$, then $O_s$ and $O_t$ are in distinct connected components of $\mathcal{F}$.
  \item If there exist two $j$-edges $\{ O_s, O_t \} $ and $\{ O_u, O_v \} $  in $\mathcal{I}$, then not all vertices $\{ O_s, O_t, O_u, O_v \} $ are in a same connected component of $\mathcal{F}$.
  \item If a cycle $\mathcal{C}$ of $\mathcal{I}$ contains the $j$-edge of  $\mathcal{F}$, then $\mathcal{C}$ contains another $j$-edge.
  \item If $O_s$ and $O_t$ are vertices in the same connected component of $\mathcal{F}$, and  $e=\{ O_s, O_t \} $ is an $j$-edge in $\mathcal{I}$, then $e$ is in $\mathcal{F}$.
  \item Let $O_v,\, O_w,\,O_s,\,O_t$ be vertices of an alternating square of $\mathcal{I}$ as in the following figure. 
  $$ 
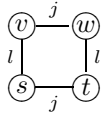
\xymatrix@-1pc{   *+[o][F]{v} \ar@{-}[d]_{l} \ar@{-}[r]^j  & *+[o][F]{w}\ar@{-}[d]^{l} \\
  *+[o][F]{s} \ar@{-}[r]_j& *+[o][F]{t}}$$
  If  $\{O_v,\, O_w\}$ and $\{O_v,\,O_s\}$ are edges of $\mathcal{F}$ then $O_t$ is in different connected components of $\mathcal{F}$.
  
  \end{enumerate}
  \end{lemma}
  \begin{proof}
 Let us use the same arguments used in Lemma 3.2 (1)  of \cite{extension} to prove (a).  If $j$ is the label of an edge $\{O_s,O_t\}$ of $\mathcal{F}$ belonging to a cycle then $j$ is also an edge of $\mathcal{G}$ that belongs to a cycle  (either with the same number of edges or with some extra $i$-edges) that does not contain other $j$-edges.
  Therefore $O_s$ and $O_t$ are in the same $G_j$-orbit, a contradiction. Thus $\mathcal{F}$ is a forest. As $\mathcal{F}$ has $n/2$ vertices and $\epsilon$ edges, the number of connected components is given by $n/2-\epsilon$.
  
  To prove (b), (c), (d), (e) and (f) we just need to adapt the proofs of Lemmas 3.2(3), 3.2(4), 3.3, 3.5 and 3.6 of \cite{extension}, respectively.   
  \end{proof}

  \begin{prop}\label{ai}
 If $G_0$ and $G_{r-1}$ are intransitive, then  $\mathcal{G}$ is, up to duality,  the graph (3) or  (4) of Table~\ref{m=2table}.
  \end{prop}
  \begin{proof}
  In this case the number of edges of   $\mathcal{F}$ is precisely $n/2-1$.
  Hence by Lemma~\ref{FG} (a) $\mathcal{F}$ is a tree. Moreover, by  Lemma~\ref{FG} (b) and (c)  $\mathcal{I}=\mathcal{F}$.  
  
  Suppose first that any pair of adjacent edges of $\mathcal{F}$ have consecutive labels.  
  Then, $i\in\{0,r-1\}$. Up to duality we may assume that $i=0$ and therefore $\mathcal{F}$ is as follows.
  $$ \xymatrix@-1pc{*+[o][F]{}\ar@{-}[rr]^1&&*+[o][F]{}\ar@{-}[rr]^2&&*+[o][F]{}\ar@{.}[rr] &&*+[o][F]{} \ar@{-}[rr]^{r-1}&& *+[o][F]{}  }$$
  As $|R|\geq 1$, we must have $R=\{\rho_1\}$. 
  Let $O_1$ and $O_2$ be the  $\langle \rho_0\rangle$-orbits swapped by $\rho_1$.
  Suppose there is a $x\neq 0$ such that $B_1\rho_x=B_2$. First $x\neq 1$, otherwise $\rho_1$ and $\rho_0$ commute, a contradiction.
  In order to avoid the same contradiction, the unique permutation $\rho_j$ (with $j\notin\{ 0,1\}$) that may act nontrivially  on $O_1$ and $O_2$, is $\rho_2$.
  This gives a unique possibility for $x$, which is $x=2$.
   Let us then assume that $B_1\rho_2=B_2$, Then $\mathcal{G}$ is as follows.
   
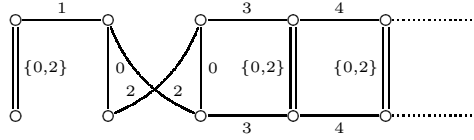
\begin{small}
$$ \xymatrix@-1pc{
  *+[o][F]{}\ar@{-}[rr]^1\ar@{=}[dd]^{\{0,2\}}&&*+[o][F]{}\ar@/_.5pc/@{-}[rrdd]_2\ar@{-}[dd]^0&&*+[o][F]{}\ar@{-}[dd]^0\ar@{-}[rr]^3&&*+[o][F]{}\ar@{=}[dd]_{\{0,2\}} \ar@{-}[rr]^4&&*+[o][F]{}\ar@{=}[dd]_{\{0,2\}}  \ar@{.}[rr]&&\\
  && && && &&&&\\
  *+[o][F]{}&&*+[o][F]{}\ar@/_.5pc/@{-}[rruu]_2&&*+[o][F]{}\ar@{-}[rr]_3&&*+[o][F]{} \ar@{-}[rr]_4&& *+[o][F]{}\ar@{.}[rr] &&}$$
\end{small} 
  Then $(\rho_2\rho_3)^3=\rho_0$, a contradiction.
  Then $\rho_0$ is the only generator swapping the blocks.
  Hence $\mathcal{G}$ is the graph (3) of Table~\ref{m=2table}.
  
  Now  suppose that  $l$ and $j$ are nonconsecutive  labels of adjacent edges of $\mathcal{F}$. 
  If $\{l,j\}\neq \{i-1,i+1\}$ then
  $\mathcal{G}$ has an alternating $\{l,j\}$-square, and consequently $\mathcal{I}$ has an alternating $\{l,j\}$-square. 
  Then, by Lemma~\ref{FG} (f), $\mathcal{F}$ cannot be a tree, a contradiction.
   Thus $\{l,j\}= \{i-1,i+1\}$.
   This gives following possibility for $\mathcal{F}$.
    $$ \xymatrix@-1pc{*+[o][F]{}\ar@{-}[rr]^1&&*+[o][F]{}\ar@{.}[rr] &&*+[o][F]{} \ar@{-}[rr]^{i-1}&& *+[o][F]{}  \ar@{-}[rr]^{i+1}&&*+[o][F]{}\ar@{.}[rr] &&*+[o][F]{} \ar@{-}[rr]^{r-1}&& *+[o][F]{}  }$$
    If $\rho_i$ is the unique permutation permuting $B_1$ and $B_2$ then $\mathcal{G}$ is the graph (4) of Table~\ref{m=2table}.
  Let us prove this is the only possibility.
  Suppose on the contrary that there exits $x\neq i$ such that $B_1\rho_x=B_2$. Let $O_s$ and $O_t$ be the  $\langle \rho_i\rangle$-obits that are merged by $\rho_x$. 
  Then $\rho_x$ is a fixed point free permutation fixing all $\langle \rho_i\rangle$-orbits except  $O_s$ and $O_t$. If $x\neq i\pm 1$, as $\rho_x$ commutes either with $\rho_{i-1}$ or with $\rho_{i+1}$, then $\mathcal{G}$ has an alternating $\{i-1,i+1\}$-square.
  Consequently $\mathcal{I}$ has an alternating $\{i-1,i+1\}$-square. Hence, by Lemma~\ref{FG}(f) $\mathcal{F}$ has two connected components, a contradiction. If $x=i\pm 1$ then  $\mathcal{G}$ also has a pair of adjacent edges with labels $i-1$ and $i+1$, giving rise to the same contradiction as before.
 
  \end{proof}

  \begin{prop}
   $G_0$ is transitive if and only if $G_{r-1}$ is transitive.
  \end{prop}
  \begin{proof}
 Suppose that $G_0$ is transitive and that $G_{r-1}$ is intransitive.
  If $i=0$, by Proposition~\ref{aallint}, $G_j$ is intransitive for every $j\neq 0$. Then by Proposition~\ref{ai}, $\mathcal{G}$ is the graph (3) or (4) of Table~\ref{m=2table}, but in both cases $G_0$ is intransitive, a contradiction.
  Thus $i\neq 0$.
  In this case $\mathcal{F}$  is a forest with exactly two connected components.
  Moreover $\mathcal{I}$ has a $0$-edge $\{O_s,O_t\}$ connecting the two connected components of $\mathcal{F}$. 
  Additionally, by Lemma~\ref{FG} (c), all the edges of $\mathcal{I}$ that do not belong to $\mathcal{F}$ must be edges incident to $O_s$ or to $O_t$.
  
  Suppose first that any pair of adjacent edges of $\mathcal{F}$ are consecutive.  
   Then, up to duality,  $\mathcal{F}$ is as one of the following graphs.
   \begin{small}
  $$ \xymatrix@-1.5pc{*+[o][F]{}&&*+[o][F]{}\ar@{-}[rr]^1&&*+[o][F]{}\ar@{-}[rr]^2&&*+[o][F]{}\ar@{.}[rr] &&*+[o][F]{} \ar@{-}[rr]^{r-2}&& *+[o][F]{}  }\qquad\qquad
  \xymatrix@-1.5pc{*+[o][F]{}&&*+[o][F]{}\ar@{-}[rr]^2&&*+[o][F]{}\ar@{-}[rr]^3&&*+[o][F]{}\ar@{.}[rr] &&*+[o][F]{} \ar@{-}[rr]^{r-1}&& *+[o][F]{}  }$$
  \end{small}
  In the case on the left $i=r-1$ and, as $|R|>1$, $\rho_{r-2}$ must be a transposition. But then $G_0\cong \Sym_{n/2}\wr \Cyc_2$, hence $\rho_0\in G_0$, a contradiction.
  Thus $\mathcal{F}$ must be the graph on the right, particularly $i=1$.
   Now  $\mathcal{I}$ is one of the following graphs where $l\in\{2,3\}$.
  \begin{small}
    \[\begin{array}{cc}
 
   \mbox{(a)}\;  \xymatrix@-1.5pc{&*+[o][F]{}\ar@{=}[rr]_0^l&&*+[o][F]{}\ar@{-}[rr]^2&&*+[o][F]{}\ar@{-}[rr]^3&&*+[o][F]{}\ar@{.}[rr] &&*+[o][F]{} \ar@{-}[rr]^{r-1}&& *+[o][F]{}  }&\qquad  \mbox{(c)}\; \xymatrix@-1.5pc{&*+[o][F]{}\ar@{-}[rrd]^2\ar@{=}[dd]_0^l\\
  &&& *+[o][F]{}\ar@{-}[rr]^3&&*+[o][F]{}\ar@{-}[rr]^4&&*+[o][F]{}\ar@{.}[rr] &&*+[o][F]{} \ar@{-}[rr]^{r-1}&& *+[o][F]{}\\
  &*+[o][F]{}\ar@{-}[rru]_2}\\
   \mbox{(b)}\; \xymatrix@-1.5pc{&*+[o][F]{}\ar@{-}[rrd]^2\ar@{-}[dd]_0\\
  &&& *+[o][F]{}\ar@{-}[rr]^3&&*+[o][F]{}\ar@{-}[rr]^4&&*+[o][F]{}\ar@{.}[rr] &&*+[o][F]{} \ar@{-}[rr]^{r-1}&& *+[o][F]{}\\
  &*+[o][F]{}\ar@{-}[rru]_2}&\\
  \end{array}\]
\end{small}
  Consider first the graph (a).
  Suppose that, for $x\neq 1$, $B_1\rho_x=B_2$. In any case this forces the existence of a $\{0,2\}$-square, giving a contradiction.
  Hence $\rho_1$ is the unique permutation swapping the blocks $B_1$ and $B_2$.
  This gives the following three possible graphs.
 \[\begin{array}{cc}
 (X) \; \xymatrix@-1.5pc{
  *+[o][F]{}\ar@{-}[dd]_1\ar@{=}[rr]^0_2&&*+[o][F]{}\ar@{-}[dd]_1&&*+[o][F]{}\ar@{-}[dd]_1\ar@{-}[rr]^3&&*+[o][F]{}\ar@{-}[dd]_1\ar@{.}[rr] &&*+[o][F]{} \ar@{-}[dd]^1\ar@{-}[rr]^{r-1}&& *+[o][F]{}   \ar@{-}[dd]^1 \\
  &&&&&&&&&&\\
  *+[o][F]{}&&*+[o][F]{}\ar@{-}[rr]_2&&*+[o][F]{}\ar@{-}[rr]_3&&*+[o][F]{}\ar@{.}[rr]&&*+[o][F]{} \ar@{-}[rr]_{r-1}&& *+[o][F]{} }
 &(Y)\;  \xymatrix@-1.5pc{
  *+[o][F]{}\ar@{-}[dd]_1\ar@{=}[rr]^0_3&&*+[o][F]{}\ar@{-}[dd]_1&&*+[o][F]{}\ar@{-}[dd]_1\ar@{-}[rr]^3&&*+[o][F]{}\ar@{-}[dd]_1\ar@{.}[rr] &&*+[o][F]{} \ar@{-}[dd]^1\ar@{-}[rr]^{r-1}&& *+[o][F]{}   \ar@{-}[dd]^1 \\
  &&&&&&&&&&\\
  *+[o][F]{}\ar@{-}[rr]_3&&*+[o][F]{}\ar@{-}[rr]_2&&*+[o][F]{}\ar@{-}[rr]_3&&*+[o][F]{}\ar@{.}[rr]&&*+[o][F]{} \ar@{-}[rr]_{r-1}&& *+[o][F]{} }
  \\
 \; (Z)  \xymatrix@-1.5pc{
  *+[o][F]{}\ar@{-}[dd]_1\ar@{-}[rr]\ar@{=}[rr]^{\{0,2,3\}}&&*+[o][F]{}\ar@{-}[dd]_1&&*+[o][F]{}\ar@{-}[dd]_1\ar@{-}[rr]^3&&*+[o][F]{}\ar@{-}[dd]_1\ar@{.}[rr] &&*+[o][F]{} \ar@{-}[dd]^1\ar@{-}[rr]^{r-1}&& *+[o][F]{}   \ar@{-}[dd]^1 \\
  &&&&&&&&&&\\
  *+[o][F]{}\ar@{-}[rr]_3&&*+[o][F]{}\ar@{-}[rr]_2&&*+[o][F]{}\ar@{-}[rr]_3&&*+[o][F]{}\ar@{.}[rr]&&*+[o][F]{} \ar@{-}[rr]_{r-1}&& *+[o][F]{} }
  &
 \end{array}\]
Consider first graph $(X)$. We can easily check computationally that $G_{0,4,\ldots,r-1} \cap G_{3,4,\ldots,r-1} \cong \Sym_3\times \Sym_3$ while $G_{0,3,4,\ldots,r-1}\cong D_{12}$, contradicting the intersection property. Thus $(X)$ is not a permutation representation graph of  a string C-group.
The graph $(Y)$ gives the relation $\rho_2^{\rho_3\rho_4\rho_3\rho_4\rho_3\rho_2\rho_1} = \rho_0$, thus the corresponding set of generators is not independent.
Lastly, in graph $(Z)$, let $\rho_1=(1,2)(3,4)\ldots(n-1,n)$, $\rho_0=(1,3)$ and $\rho_2=(1,3)(4,6)$. Both $(\rho_0\rho_1)^2 \in G_{3,4,\ldots,r-1}$ and $\rho_2^{\rho_3\rho_4\rho_3\rho_4\rho_3\rho_2}\in G_{0,5,\ldots,r-1}$ are equal to $(1,3)(2,4)$, however $(1,3)(2,4)\notin G_{0,3,4,\ldots,r-1}$, thus the intersection property fails.

 Now consider the graphs (b) and (c). In these cases $\mathcal{G}$ contains the following graph.
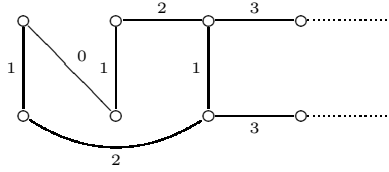
\begin{small}
$$ \xymatrix@-1pc{
  *+[o][F]{}\ar@{-}[dd]_1\ar@{-}[rrdd]^0&&*+[o][F]{}\ar@{-}[rr]^2\ar@{-}[dd]_1&&*+[o][F]{}\ar@{-}[dd]_1\ar@{-}[rr]^3&&*+[o][F]{}\ar@{.}[rr]&& \\
  && && && &&\\
  *+[o][F]{}\ar@/_1pc/@{-}[rrrr]_2&&*+[o][F]{}&&*+[o][F]{} \ar@{-}[rr]_3&& *+[o][F]{} *+[o][F]{}\ar@{.}[rr]&&}$$
\end{small} 
 But then $B_1\rho_0=B_2$, and consequently $\rho_0\in C$ and $\mathcal{G}$ contains the following graph.
 
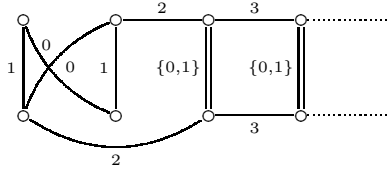
\begin{small}
$$ \xymatrix@-1pc{
  *+[o][F]{}\ar@{-}[dd]_1\ar@/_.5pc/@{-}[rrdd]^0&&*+[o][F]{}\ar@/_.5pc/@{-}[lldd]_0\ar@{-}[rr]^2\ar@{-}[dd]_1&&*+[o][F]{}\ar@{=}[dd]_{\{0,1\}}\ar@{-}[rr]^3&&*+[o][F]{}\ar@{=}[dd]_{\{0,1\}}\ar@{.}[rr]&& \\
  && && && &&\\
  *+[o][F]{}\ar@/_1pc/@{-}[rrrr]_2&&*+[o][F]{}&&*+[o][F]{} \ar@{-}[rr]_3&& *+[o][F]{} *+[o][F]{}\ar@{.}[rr]&&}$$
\end{small} 

 This implies that $(\rho_2\rho_3)^3=\rho_0$, a contradiction.
   
  Now suppose that there is a pair of incident edges of  $\mathcal{F}$ which have nonconsecutive labels.
  Then their labels must be $i-1$ and $i+1$. Moreover $R=\{\rho_{i-1},\rho_{i+1}\}$. In this case the elements of $R$ are transpositions, hence $G_0\cong \Sym_{n/2}\wr \Cyc_2$, a contradiction.
  
 We have shown that if \( G_0 \) is transitive, then \( G_{r-1} \) must be intransitive. By duality, the converse implication also holds.
  \end{proof}

  \begin{prop}
   $G_0$ and $G_{r-1}$ cannot be both transitive.
  \end{prop}
  \begin{proof}
  Suppose first that $\rho_0$ and $\rho_{r-1}$ are the only permutations swapping $B_1$ and $B_2$ and let $L=\{\rho_0\}$ (meaning that $i=0$). Then $\mathcal{F}$ has $n/2-2$ edges and two connected components, which are joined in $\mathcal{I}$ by a double $\{r-1,l\}$-edge for some $l\neq r-1$. But then, this double edge must belong to a square whose vertices must belong to at least three different connected components of $\mathcal{F}$, a contradiction. 
   Hence, there exists $j\notin\{0,r-1\}$ such that $B_1\rho_j=B_2$ thus we may assume that the element of $L$ is neither $\rho_0$ nor $\rho_{r-1}$.
   In this case $\mathcal{F}$ has exactly three components.
  
  Suppose that $\mathcal{I}$ has a double  $\{0,r-1\}$-edge. Then this double edge must belong to a square having two vertices in the same connected component of $\mathcal{F}$.
  Hence one edge of this square belongs to  $\mathcal{F}$, by Lemma~\ref{FG} (e). 
  As $n>8$ and $r>6$, there is another square that is adjacent to the previous one and these two adjacent squares form a graph with $6$ vertices and at most two edges in $\mathcal{F}$.
  Thus  $\mathcal{F}$ has at most four components, a contradiction.

 Now suppose that $\mathcal{I}$ has a $0$-edge that is adjacent to a $(r-1)$-edge. Then  $\mathcal{I}$ has an alternating square whose vertices belong to different  components of  $\mathcal{F}$, a contradiction.
  Thus $\mathcal{I}$ has a $0$-edge and a  $(r-1)$-edge that are not adjacent.
  This determine the three components of $\mathcal{F}$.
  An edge of $\mathcal{F}$ which is adjacent to the $0$-edge of $\mathcal{I}$ must have label $1$ and an edge of $\mathcal{F}$ which is adjacent to the $(r-1)$-edge of $\mathcal{I}$ must have label $r-2$.
  Hence two components of $\mathcal{F}$ are isolated vertices.
  Suppose first that adjacent edges of $\mathcal{F}$ have consecutive labels. Then either $i=1$ or $i=r-2$. Up to duality we may assume that $i=1$, then $\mathcal{F}$ is as follows.
    \begin{small}
   $$\xymatrix@-1.7pc{*+[o][F]{}&&*+[o][F]{}\ar@{-}[rr]^2&&*+[o][F]{} \ar@{.}[rr]&&*+[o][F]{}\ar@{-}[rr]^{r-2} &&*+[o][F]{}&& *+[o][F]{} }$$
   \end{small}
   Now any edge in $\mathcal{I}$ that is not in $\mathcal{F}$ must be incident to one of the isolated vertices of $\mathcal{F}$.
   Hence $\rho_0$ is a transposition, thus $G_{r-1}\cong \Sym_{n/2}\wr \Cyc_2$, a contradiction.

 Now consider the case where $\mathcal{F}$ has adjacent edges with nonconsecutive labels. This is only possible when a $(i-1)$-edge is adjacent to a $(i+1)$-edge, and then $R=\{\rho_{i-1},\rho_{i+1}\}$. But in this case the elements of $R$ are transpositions, hence $G_0\cong \Sym_{n/2}\wr\Cyc_2$, a contradiction.
 \end{proof}

   
  \section{Proof of Theorem~\ref{main} and list of all possibilities for $\mathcal{G}$} \label{tables+proof}
  
Let $G$ be the automorphism group of an abstract regular polytope $r\geq n/2$ whose automorphism  group has degree $n \geq 14$. 
By Proposition~\ref{prim} $G$ must be imprimitive and therefore embedded into $\Sym_k\wr \Sym_m$ with $n=mk$. Moreover Corollary~\ref{k=2orm=2} shows that either $k=2$ or $m=2$.
In Section~\ref{k=2} all the possibilities for $\mathcal{G}$ when $G$ admits a block system with $n/2$ blocks of size $2$ were determined, while Section~\ref{m=2} covers all the possibilities when $G$ has two blocks of size $n/2$. 
This leads to the possibilities given in Tables~(\ref{RC2})--(\ref{m=2table}).
All the graphs presented in these tables correspond to permutation representations of SGGIs. 
Based on computational results, we conjecture that these sets of generators satisfy the intersection property. 
In fact, some cases can be proven by induction. 
However, we have chosen to omit this analysis from the current paper and leave it as future work. 
One reason for this decision is our belief that a more elegant and general method may exist—one that avoids an exhaustive case-by-case analysis.
Having established this result, it is natural to ask how many polytopes, up to isomorphism and duality, exist in this classification, and to describe their Schläfli types.

\begin{tiny}
  \begin{table}[htbp!]
  \[\begin{array}{|cc|}
  \hline
&\mbox{ Case: }  |R\cup C|=2;\;  \langle L\rangle \cong \Sym_{n/2}; \; n/2\mbox{ odd };\; \langle L\rangle \mbox{ intransitive }.\\
      \hline
   (1) &\xymatrix@-1pc{
          *+[o][F]{}\ar@{-}[rr]^2 \ar@{-}[dd]_0&&*+[o][F]{}\ar@{=}[dd]_0^1\ar@{-}[rr]^3&&*+[o][F]{}\ar@{=}[dd]_0^1\ar@{.}[rr] &&*+[o][F]{} \ar@{=}[dd]_0^1\ar@{-}[rr]^{r-1}&& *+[o][F]{}   \ar@{=}[dd]_0^1 \\
          &&&&&&&&\\
          *+[o][F]{}\ar@{-}[rr]_2&&*+[o][F]{}\ar@{-}[rr]_3&&*+[o][F]{}\ar@{.}[rr]&&*+[o][F]{} \ar@{-}[rr]_{r-1}&& *+[o][F]{} }
     \quad(2)\quad \xymatrix@-1pc{
          *+[o][F]{}\ar@{-}[rr]^1 \ar@{=}[dd]_0^{r-1}&&*+[o][F]{}\ar@{=}[dd]_0^{r-1}\ar@{-}[rr]^2&&*+[o][F]{}\ar@{=}[dd]_0^{r-1}\ar@{.}[rr] &&*+[o][F]{} \ar@{=}[dd]_0^{r-1}\ar@{-}[rr]^{r-2}&& *+[o][F]{}   \ar@{-}[dd]_0 \\
          &&&&&&&&\\
          *+[o][F]{}\ar@{-}[rr]_1&&*+[o][F]{}\ar@{-}[rr]_2&&*+[o][F]{}\ar@{.}[rr]&&*+[o][F]{} \ar@{-}[rr]_{r-2}&& *+[o][F]{} }
          \\
      (3) & \xymatrix@-1pc{
        *+[o][F]{}\ar@{-}[rr]^3 \ar@{=}[dd]_0^1 && *+[o][F]{}\ar@{-}[rr]^2 \ar@{=}[dd]_0^1&&*+[o][F]{}\ar@{-}[dd]_0\ar@{-}[rr]^3 &&*+[o][F]{}\ar@{-}[dd]_0\ar@{-}[rr]^4 &&*+[o][F]{}\ar@{-}[dd]_0\ar@{.}[rr] &&*+[o][F]{} \ar@{-}[dd]_0\ar@{-}[rr]^{r-1}&& *+[o][F]{}   \ar@{-}[dd]_0 \\
        &&&&&&&&\\
        *+[o][F]{}\ar@{-}[rr]_3&&*+[o][F]{}\ar@{-}[rr]_2&& *+[o][F]{}\ar@{-}[rr]_3 && *+[o][F]{}\ar@{-}[rr]_4 &&*+[o][F]{}\ar@{.}[rr]&&*+[o][F]{} \ar@{-}[rr]_{r-1}&& *+[o][F]{} }
        \\
      (4) & \xymatrix@-1pc{
          *+[o][F]{}\ar@{-}[rr]^2 \ar@{=}[dd]_0^{r-1} && *+[o][F]{}\ar@{-}[rr]^1 \ar@{=}[dd]_0^{r-1} &&*+[o][F]{}\ar@{=}[dd]_0^{r-1} \ar@{-}[rr]^2 &&*+[o][F]{}\ar@{=}[dd]_0^{r-1}\ar@{-}[rr]^3 &&*+[o][F]{}\ar@{=}[dd]_0^{r-1}\ar@{.}[rr] &&*+[o][F]{} \ar@{=}[dd]_0^{r-1}\ar@{-}[rr]^{r-3} &&*+[o][F]{} \ar@{=}[dd]_0^{r-1} \ar@{-}[rr]^{r-2} && *+[o][F]{}   \ar@{-}[dd]_0 \\
          &&&&&&&&\\
          *+[o][F]{}\ar@{-}[rr]_2&&*+[o][F]{}\ar@{-}[rr]_1&& *+[o][F]{}\ar@{-}[rr]_2 && *+[o][F]{}\ar@{-}[rr]_3 &&*+[o][F]{}\ar@{.}[rr]&&*+[o][F]{} \ar@{-}[rr]_{r-3} &&*+[o][F]{} \ar@{-}[rr]_{r-2} && *+[o][F]{} }
          \\
      (5) & \xymatrix@-1pc{
        *+[o][F]{}\ar@{-}[rr]^2 \ar@{-}[dd]_0&&*+[o][F]{}\ar@{=}[dd]_0^{1}\ar@{-}[rr]^3&&*+[o][F]{}\ar@{=}[dd]_0^{1}\ar@{.}[rr] &&*+[o][F]{} \ar@{=}[dd]_0^{1}\ar@{-}[rr]^{r-3} &&*+[o][F]{} \ar@{=}[dd]_0^{1}\ar@{-}[rr]^{r-2} &&*+[o][F]{} \ar@{=}[dd]_0^{1}\ar@{-}[rr]^{r-1} &&*+[o][F]{} \ar@{=}[dd]_0^{1}\ar@{-}[rr]^{r-2} && *+[o][F]{}   \ar@{=}[dd]_0^{1} \\
        &&&&&&&&\\
        *+[o][F]{}\ar@{-}[rr]_2&&*+[o][F]{}\ar@{-}[rr]_3&&*+[o][F]{}\ar@{.}[rr]&&*+[o][F]{} \ar@{-}[rr]_{r-3} &&*+[o][F]{} \ar@{-}[rr]_{r-2} &&*+[o][F]{} \ar@{-}[rr]_{r-1} &&*+[o][F]{} \ar@{-}[rr]_{r-2} && *+[o][F]{} }
        \\
      (6) & \xymatrix@-1pc{
        *+[o][F]{}\ar@{-}[rr]^1 \ar@{-}[dd]_0 &&*+[o][F]{}\ar@{-}[dd]_0 \ar@{-}[rr]^2&&*+[o][F]{}\ar@{-}[dd]_0 \ar@{.}[rr] &&*+[o][F]{} \ar@{-}[dd]_0 \ar@{-}[rr]^{r-4} &&*+[o][F]{} \ar@{-}[dd]_0 \ar@{-}[rr]^{r-3} &&*+[o][F]{} \ar@{-}[dd]_0\ar@{-}[rr]^{r-2} &&*+[o][F]{} \ar@{=}[dd]_0^{r-1}\ar@{-}[rr]^{r-3} && *+[o][F]{}   \ar@{=}[dd]_0^{r-1} \\
        &&&&&&&&\\
        *+[o][F]{}\ar@{-}[rr]_1 &&*+[o][F]{}\ar@{-}[rr]_2&&*+[o][F]{}\ar@{.}[rr]&&*+[o][F]{} \ar@{-}[rr]_{r-4} &&*+[o][F]{} \ar@{-}[rr]_{r-3} &&*+[o][F]{} \ar@{-}[rr]_{r-2} &&*+[o][F]{} \ar@{-}[rr]_{r-3} && *+[o][F]{} }
        \\
      \hline
      &\mbox{ Case: }  |R\cup C|=2;\;  \langle L\rangle \cong \Sym_{n/2}; \; n/2\mbox{ odd };\;   \langle L\rangle \mbox{ transitive }.\\
      \hline 
         (7) & \xymatrix@-1pc{
            *+[o][F]{}\ar@{-}[rr]^2 \ar@{-}[dd]_{I_{1,2}}&&*+[o][F]{}\ar@{-}[dd]^{I_{2,3}}\ar@{-}[rr]^3&&*+[o][F]{}\ar@{-}[dd]^{I_{3,4}}\ar@{.}[rr] &&*+[o][F]{} \ar@{-}[dd]^{I_{r-2,r-1}}\ar@{-}[rr]^{r-1}&& *+[o][F]{}   \ar@{-}[dd]^{I_{r-1}} \\
            &&&&&&&&\\
            *+[o][F]{}\ar@{-}[rr]_2&&*+[o][F]{}\ar@{-}[rr]_3&&*+[o][F]{}\ar@{.}[rr]&&*+[o][F]{} \ar@{-}[rr]_{r-1}&& *+[o][F]{} }
           
        \qquad (8)\qquad \xymatrix@-1pc{
            *+[o][F]{}\ar@{-}[rr]^1 \ar@{-}[dd]^{I_{1}} &&*+[o][F]{}\ar@{-}[dd]^{I_{1,2}}\ar@{-}[rr]^2&&*+[o][F]{}\ar@{-}[dd]^{I_{2,3}}\ar@{.}[rr] &&*+[o][F]{} \ar@{-}[dd]^{I_{r-3,r-2}} \ar@{-}[rr]^{r-2}&& *+[o][F]{}  \ar@{-}[dd]^{I_{r-2,r-1}} \\
            &&&&&&&&\\
            *+[o][F]{}\ar@{-}[rr]_1&&*+[o][F]{}\ar@{-}[rr]_2&&*+[o][F]{}\ar@{.}[rr]&&*+[o][F]{} \ar@{-}[rr]_{r-2}&& *+[o][F]{} }
            \\
        (9) & \xymatrix@-1pc{
          *+[o][F]{}\ar@{-}[rr]^3 \ar@{-}[dd]^{I_{3}} && *+[o][F]{}\ar@{-}[rr]^2 \ar@{-}[dd]^{I_{2,3}}&&*+[o][F]{}\ar@{-}[dd]^{I_{1,2,3}}\ar@{-}[rr]^3 &&*+[o][F]{}\ar@{-}[dd]^{I_{1,3,4}}\ar@{-}[rr]^4 &&*+[o][F]{}\ar@{-}[dd]^{I_{1,3,4,5}}\ar@{.}[rr] &&*+[o][F]{} \ar@{-}[dd]^{I_{1,3,r-2,r-1}}\ar@{-}[rrr]^{r-1}&&& *+[o][F]{}  \ar@{-}[dd]^{I_{1,3,r-1}} \\
          &&&&&&&&\\
          *+[o][F]{}\ar@{-}[rr]_3&&*+[o][F]{}\ar@{-}[rr]_2&& *+[o][F]{}\ar@{-}[rr]_3 && *+[o][F]{}\ar@{-}[rr]_4 &&*+[o][F]{}\ar@{.}[rr]&&*+[o][F]{} \ar@{-}[rrr]_{r-1}&&& *+[o][F]{} }
          \\
        (10) & \xymatrix@-1pc{
            *+[o][F]{}\ar@{-}[rr]^2 \ar@{-}[dd]^{I_{2}} && *+[o][F]{}\ar@{-}[rr]^1 \ar@{-}[dd]^{I_{1,2}} &&*+[o][F]{}\ar@{-}[dd]^{I_{1,2}}\ar@{-}[rr]^2 &&*+[o][F]{}\ar@{-}[dd]^{I_{2,3}}\ar@{-}[rr]^3 &&*+[o][F]{}\ar@{-}[dd]^{I_{2,3,4}}\ar@{.}[rr] &&*+[o][F]{} \ar@{-}[dd]^{I_{2,r-4,r-3}}\ar@{-}[rrr]^{r-3} &&&*+[o][F]{} \ar@{-}[dd]^{I_{2,r-3,r-2}} \ar@{-}[rrr]^{r-2} &&& *+[o][F]{}   \ar@{-}[dd]^{I_{2,r-2,r-1}} \\
            &&&&&&&&\\
            *+[o][F]{}\ar@{-}[rr]_2&&*+[o][F]{}\ar@{-}[rr]_1&& *+[o][F]{}\ar@{-}[rr]_2 && *+[o][F]{}\ar@{-}[rr]_3 &&*+[o][F]{}\ar@{.}[rr]&&*+[o][F]{} \ar@{-}[rrr]_{r-3} &&&*+[o][F]{} \ar@{-}[rrr]_{r-2} &&& *+[o][F]{} }
            \\
        (11) & \xymatrix@-1pc{
          *+[o][F]{}\ar@{-}[rr]^2 \ar@{-}[dd]^{I_{1,2,r-2}} &&*+[o][F]{}\ar@{-}[dd]^{I_{2,3,r-2}}\ar@{-}[rr]^3&&*+[o][F]{}\ar@{-}[dd]^{I_{3,4,r-2}}\ar@{.}[rr] &&*+[o][F]{} \ar@{-}[dd]^{I_{r-4,r-3,r-2}}\ar@{-}[rrr]^{r-3} &&&*+[o][F]{} \ar@{-}[dd]^{I_{r-3,r-2}}\ar@{-}[rr]^{r-2} &&*+[o][F]{} \ar@{-}[dd]^{I_{r-2,r-1}}\ar@{-}[rr]^{r-1} &&*+[o][F]{} \ar@{-}[dd]^{I_{r-2,r-1}}\ar@{-}[rr]^{r-2} && *+[o][F]{}   \ar@{-}[dd]^{I_{r-2}} \\
          &&&&&&&&\\
          *+[o][F]{}\ar@{-}[rr]_2&&*+[o][F]{}\ar@{-}[rr]_3&&*+[o][F]{}\ar@{.}[rr]&&*+[o][F]{} \ar@{-}[rrr]_{r-3} &&&*+[o][F]{} \ar@{-}[rr]_{r-2} &&*+[o][F]{} \ar@{-}[rr]_{r-1} &&*+[o][F]{} \ar@{-}[rr]_{r-2} && *+[o][F]{} }
          \\
        (12) & \xymatrix@-1pc{
          *+[o][F]{}\ar@{-}[rrr]^1 \ar@{-}[dd]^{I_{1,r-3,r-1}} &&&*+[o][F]{}\ar@{-}[dd]^{I_{1,2,r-3,r-1}} \ar@{-}[rrr]^2&&&*+[o][F]{}\ar@{-}[dd]^{I_{2,3,r-3,r-1}} \ar@{.}[rrr] &&&*+[o][F]{} \ar@{-}[dd]^{I_{r-5,r-4,r-3,r-1}} \ar@{-}[rrrr]^{r-4} &&&&*+[o][F]{} \ar@{-}[dd]^{I_{r-4,r-3,r-1}} \ar@{-}[rrr]^{r-3} &&&*+[o][F]{} \ar@{-}[dd]^{I_{r-3,r-2,r-1}}\ar@{-}[rrr]^{r-2} &&&*+[o][F]{} \ar@{-}[dd]^{I_{r-3,r-2}} \ar@{-}[rr]^{r-3} && *+[o][F]{}   \ar@{-}[dd]^{I_{r-3}} \\
          &&&&&&&&\\
          *+[o][F]{}\ar@{-}[rrr]_1 &&&*+[o][F]{}\ar@{-}[rrr]_2&&&*+[o][F]{}\ar@{.}[rrr]&&&*+[o][F]{} \ar@{-}[rrrr]_{r-4} &&&&*+[o][F]{} \ar@{-}[rrr]_{r-3} &&&*+[o][F]{} \ar@{-}[rrr]_{r-2} &&&*+[o][F]{} \ar@{-}[rr]_{r-3} && *+[o][F]{} }
          \\
        \hline
    \end{array}\] 
    \caption{ $\mathbf{k = 2}$; Corollary~\ref{k=2RC2}.}
      \label{RC2}
    \end{table}
    \end{tiny}

  \begin{tiny}
    \begin{table}[htbp!]
      \[\begin{array}{|cc|}
      \hline
&\mbox{ Case: } |R\cup C|=1;\;  \langle L\rangle \cong \Sym_{n/2}; \;  \langle L\rangle \mbox{ intransitive }.\\
      \hline

      \hline
      (13) & $\xymatrix@-1pc{
        *+[o][F]{}\ar@{-}[dd]^0\ar@{-}[rr]^1&&*+[o][F]{}\ar@{-}[rr]^2&&*+[o][F]{}\ar@{.}[rr] &&*+[o][F]{} \ar@{-}[rr]^{r-1}&& *+[o][F]{}  \\
        &&&&&&&&\\
        *+[o][F]{}\ar@{-}[rr]_1&&*+[o][F]{}\ar@{-}[rr]_2&&*+[o][F]{}\ar@{.}[rr]&&*+[o][F]{} \ar@{-}[rr]_{r-1}&& *+[o][F]{} }$
      \qquad(14) \qquad 
      $\xymatrix@-1pc{
          *+[o][F]{}\ar@{-}[rr]^1&&*+[o][F]{}\ar@{-}[dd]_0\ar@{-}[rr]^2&&*+[o][F]{}\ar@{-}[dd]_0\ar@{.}[rr] &&*+[o][F]{} \ar@{-}[dd]^0\ar@{-}[rr]^{r-1}&& *+[o][F]{}   \ar@{-}[dd]^0 \\
          &&&&&&&&\\
          *+[o][F]{}\ar@{-}[rr]_1&&*+[o][F]{}\ar@{-}[rr]_2&&*+[o][F]{}\ar@{.}[rr]&&*+[o][F]{} \ar@{-}[rr]_{r-1}&& *+[o][F]{} }$
          \\
             \hline
&\mbox{ Case: } |R\cup C|=1;\;  \langle L\rangle \cong \Sym_{n/2}; \;  \langle L\rangle \mbox{ transitive }.\\
      \hline
      (15) & $\xymatrix@-1pc{*+[o][F]{}\ar@{-}[dd]_{I_1}\ar@{-}[rr]^1&&*+[o][F]{}\ar@{-}[dd]_{I_{0,1,2}}\ar@{-}[rr]^2&&*+[o][F]{}\ar@{-}[dd]_{I_{0,2,3}}\ar@{.}[rr]  &&*+[o][F]{} \ar@{-}[dd]^{I_{0,r-3,r-2}}\ar@{-}[rrr]^{r-2}&&&*+[o][F]{} \ar@{-}[dd]^{I_{0,r-2,r-1}}\ar@{-}[rrr]^{r-1}&&& *+[o][F]{}   \ar@{-}[dd]^{I_{0,r-1}} \\
      &&&&&&&&\\
      *+[o][F]{}\ar@{-}[rr]_1&&*+[o][F]{}\ar@{-}[rr]_2&&*+[o][F]{}\ar@{.}[rr]&&*+[o][F]{} \ar@{-}[rrr]_{r-2}&&&*+[o][F]{} \ar@{-}[rrr]_{r-1}&&& *+[o][F]{} }$
      \\
    (16)& $\xymatrix@-1pc{*+[o][F]{}\ar@{-}[dd]_{I_{0,1}}\ar@{-}[rr]^1&&*+[o][F]{}\ar@{-}[dd]_{I_{1,2}}\ar@{-}[rr]^2&&*+[o][F]{}\ar@{-}[dd]_{I_{2,3}}\ar@{.}[rr]  &&*+[o][F]{} \ar@{-}[dd]^{I_{r-3,r-2}}\ar@{-}[rr]^{r-2}&&*+[o][F]{} \ar@{-}[dd]^{I_{r-2,r-1}}\ar@{-}[rr]^{r-1}&& *+[o][F]{}   \ar@{-}[dd]^{I_{r-1}} \\
      &&&&&&&&\\
      *+[o][F]{}\ar@{-}[rr]_1&&*+[o][F]{}\ar@{-}[rr]_2&&*+[o][F]{}\ar@{.}[rr]&&*+[o][F]{} \ar@{-}[rr]_{r-2}&&*+[o][F]{} \ar@{-}[rr]_{r-1}&& *+[o][F]{} }$
      \\
      \hline
        \end{array}\] 
    \caption{ $\mathbf{k = 2}$; Proposition~\ref{S}.}
      \label{L=S}
    \end{table}
    \end{tiny}

  \begin{tiny}
    \begin{table}[htbp!]
      \[\begin{array}{|cc|}
        \hline
&\mbox{ Case: }  |R\cup C|=1\;  \langle L\rangle \not\cong \Sym_{n/2}; \;{\rm Ker}(f) \cong \Cyc_2.\\
      \hline
      & 2\leq i \leq r-2\\
      \hline 
      (17) & \xymatrix@-1pc{
      *+[o][F]{}\ar@{-}[rr]^{1} \ar@{-}[dd]_{I^{\leq i}_{0,1}} && *+[o][F]{}\ar@{-}[rr]^{2} \ar@{-}[dd]_{I^{\leq i}_{1,2}} && *+[o][F]{} \ar@{.}[rr] \ar@{-}[dd]_{I^{\leq i}_{2,3}} && *+[o][F]{} \ar@{-}[rr]^{i-1} \ar@{-}[dd]_{I^{\leq i}_{i-2,i-1}} && *+[o][F]{} \ar@{-}[rr]^{i}\ar@{-}[dd]_{I^{\leq i}_{i-1,i}} && *+[o][F]{} \ar@{-}[rr]^{i+1}\ar@{-}[dd]_{I^{\leq i}_{i}}&& *+[o][F]{} \ar@{.}[rr] \ar@{-}[dd]_{I^{\leq i}}  && *+[o][F]{} \ar@{-}[rr]^{r-1} \ar@{-}[dd]_{I^{\leq i}} && *+[o][F]{} \ar@{-}[dd]_{I^{\leq i}}\\
          \\
      *+[o][F]{}\ar@{-}[rr]_{1}                 && *+[o][F]{}\ar@{-}[rr]_{2} && *+[o][F]{} \ar@{.}[rr] &&  *+[o][F]{} \ar@{-}[rr]_{i-1} && *+[o][F]{} \ar@{-}[rr]_{i} && *+[o][F]{} \ar@{-}[rr]_{i+1} && *+[o][F]{} \ar@{.}[rr]   && *+[o][F]{} \ar@{-}[rr]_{r-1} && *+[o][F]{}
      }  \\

      (18) & \xymatrix@-1pc{
      *+[o][F]{}\ar@{-}[rr]^{1} \ar@{-}[dd]_{I^{\leq i}_{1}} && *+[o][F]{}\ar@{-}[rr]^{2} \ar@{-}[dd]_{I^{\leq i}_{0,1,2}} && *+[o][F]{} \ar@{.}[rrr] \ar@{-}[dd]_{I^{\leq i}_{0,2,3}} &&& *+[o][F]{} \ar@{-}[rr]^{i-1} \ar@{-}[dd]_{I^{\leq i}_{0,i-2,i-1}} && *+[o][F]{} \ar@{-}[rr]^{i}\ar@{-}[dd]_{I^{\leq i}_{0,i-1,i}} && *+[o][F]{} \ar@{-}[rr]^{i+1}\ar@{-}[dd]_{I^{\leq i}_{0,i}} && *+[o][F]{} \ar@{.}[rr] \ar@{-}[dd]_{I^{\leq i}_0}  && *+[o][F]{} \ar@{-}[rr]^{r-1} \ar@{-}[dd]_{I^{\leq i}_0} && *+[o][F]{} \ar@{-}[dd]_{I^{\leq i}_0}\\
          \\
      *+[o][F]{}\ar@{-}[rr]_{1}                 && *+[o][F]{}\ar@{-}[rr]_{2} && *+[o][F]{} \ar@{.}[rrr]   &&& *+[o][F]{} \ar@{-}[rr]_{i-1} && *+[o][F]{} \ar@{-}[rr]_{i} && *+[o][F]{} \ar@{-}[rr]_{i+1} && *+[o][F]{} \ar@{.}[rr]  && *+[o][F]{} \ar@{-}[rr]_{r-1} && *+[o][F]{}
      }  \\
      \hline
      & 1\leq i \leq r-3\\
      \hline 
      (19) &  \xymatrix@-1pc{
      *+[o][F]{}\ar@{-}[rr]^{1} \ar@{-}[dd]_{I^{\geq i+1}} && *+[o][F]{} \ar@{.}[rr] \ar@{=}[dd]_{I^{\geq i+1}}^{0}  && *+[o][F]{} \ar@{-}[rr]^{i}\ar@{=}[dd]_{I^{\geq i+1}}^{0} && *+[o][F]{} \ar@{-}[rrr]^{i+1}\ar@{=}[dd]_{I^{> i+1}}^0&&& *+[o][F]{} \ar@{.}[rrr] \ar@{=}[dd]_{I^{> i+1}_{i+2}}^0 &&& *+[o][F]{} \ar@{-}[rrr]^{r-2}\ar@{=}[dd]_{I^{\geq i+1}_{r-3,r-2}}^0 &&& *+[o][F]{} \ar@{-}[rr]^{r-1} \ar@{=}[dd]_{I^{\geq i+1}_{r-2,r-1}}^0 && *+[o][F]{} \ar@{=}[dd]_{I^{\geq i+1}_{r-1}}^0\\
          \\
      *+[o][F]{}\ar@{-}[rr]_{1}                && *+[o][F]{} \ar@{.}[rr] && *+[o][F]{} \ar@{-}[rr]_{i} && *+[o][F]{} \ar@{-}[rrr]_{i+1} &&& *+[o][F]{} \ar@{.}[rrr] &&& *+[o][F]{} \ar@{-}[rrr]_{r-2} &&& *+[o][F]{} \ar@{-}[rr]_{r-1} && *+[o][F]{}
      } \\

      (20) & \xymatrix@-1pc{
      *+[o][F]{}\ar@{-}[rr]^{1} \ar@{=}[dd]_{I^{\geq i+1}}^0 && *+[o][F]{}\ar@{-}[rr]^{2} \ar@{-}[dd]_{I^{\geq i+1}} && *+[o][F]{} \ar@{.}[rr] \ar@{-}[dd]_{I^{\geq i+1}} && *+[o][F]{} \ar@{-}[rr]^{i-1} \ar@{-}[dd]_{I^{\geq i+1}} && *+[o][F]{} \ar@{-}[rr]^{i}\ar@{-}[dd]_{I^{\geq i+1}} && *+[o][F]{} \ar@{-}[rr]^{i+1}\ar@{-}[dd]_{I^{> i+1}}&& *+[o][F]{} \ar@{.}[rr] \ar@{-}[dd]_{I^{> i+1}_{i+2}} && *+[o][F]{} \ar@{-}[rr]^{r-2}\ar@{-}[dd]_{I^{\geq i+1}_{r-3,r-2}} && *+[o][F]{} \ar@{-}[rr]^{r-1} \ar@{-}[dd]_{I^{\geq i+1}_{r-2,r-1}} && *+[o][F]{} \ar@{-}[dd]_{I^{\geq i+1}_{r-1}}\\
          \\
      *+[o][F]{}\ar@{-}[rr]_{1}                 && *+[o][F]{}\ar@{-}[rr]_{2} && *+[o][F]{} \ar@{.}[rr] &&  *+[o][F]{} \ar@{-}[rr]_{i-1} && *+[o][F]{} \ar@{-}[rr]_{i} && *+[o][F]{} \ar@{-}[rr]_{i+1} && *+[o][F]{} \ar@{.}[rr] && *+[o][F]{} \ar@{-}[rr]_{r-2} && *+[o][F]{} \ar@{-}[rr]_{r-1} && *+[o][F]{}
      } \\
      \hline
      & \mbox{Remaining cases}\\
      \hline 
      (21) &\xymatrix@-1pc{
        *+[o][F]{}\ar@{-}[rr]^1\ar@{-}[dd]^{0}&&*+[o][F]{}\ar@{-}[rr]^2&&*+[o][F]{}\ar@{-}[rr]^3\ar@{-}[dd]^{1}&&*+[o][F]{}\ar@{.}[rr]\ar@{-}[dd]^{1}&&*+[o][F]{} \ar@{-}[rr]^{r-1}\ar@{-}[dd]^{1}&& *+[o][F]{}\ar@{-}[dd]^{1}  \\
        &&&&&&&&\\
        *+[o][F]{}\ar@{-}[rr]_1&&*+[o][F]{}\ar@{-}[rr]_2&&*+[o][F]{}\ar@{-}[rr]_3&&*+[o][F]{}\ar@{.}[rr]&&*+[o][F]{} \ar@{-}[rr]_{r-1}&& *+[o][F]{} } \\
      (22)& \xymatrix@-1pc{
        *+[o][F]{}\ar@{-}[rr]^1&&*+[o][F]{}\ar@{-}[rr]^2\ar@{-}[dd]^0&&*+[o][F]{}\ar@{-}[rr]^3\ar@{=}[dd]^{0,1}&&*+[o][F]{}\ar@{.}[rr]\ar@{=}[dd]^{0,1}&&*+[o][F]{} \ar@{-}[rr]^{r-1}\ar@{=}[dd]^{0,1}&& *+[o][F]{}\ar@{=}[dd]^{0,1}  \\
        &&&&&&&&\\
        *+[o][F]{}\ar@{-}[rr]_1&&*+[o][F]{}\ar@{-}[rr]_2&&*+[o][F]{}\ar@{-}[rr]_3&&*+[o][F]{}\ar@{.}[rr]&&*+[o][F]{} \ar@{-}[rr]_{r-1}&& *+[o][F]{} } 
        \\
    (23) & \xymatrix@-1pc{
        *+[o][F]{}\ar@{-}[rr]^1\ar@{=}[dd]_{\{0,r-1\}}&&*+[o][F]{}\ar@{-}[dd]_{r-1}\ar@{-}[rr]^2&&*+[o][F]{}\ar@{-}[dd]_{r-1}\ar@{.}[rr] &&*+[o][F]{} \ar@{-}[dd]_{r-1}\ar@{-}[rr]^{r-2}  
        &&*+[o][F]{} \ar@{-}[rr]^{r-1}&& *+[o][F]{}   \\
        &&&&&&&&\\
        *+[o][F]{}\ar@{-}[rr]_1&&*+[o][F]{}\ar@{-}[rr]_2&&*+[o][F]{}\ar@{.}[rr]&&*+[o][F]{} \ar@{-}[rr]_{r-2}&& *+[o][F]{} \ar@{-}[rr]_{r-1}&& *+[o][F]{} } \\
(24) & \xymatrix@-1pc{
        *+[o][F]{}\ar@{-}[rr]^1\ar@{-}[dd]_{r-1}&&*+[o][F]{}\ar@{=}[dd]_{\{0,r-1\}}\ar@{-}[rr]^2&&*+[o][F]{}\ar@{=}[dd]_{\{0,r-1\}}\ar@{.}[rr] &&*+[o][F]{} \ar@{=}[dd]_{\{0,r-1\}}\ar@{-}[rr]^{r-2}  
        &&*+[o][F]{} \ar@{-}[dd]_0\ar@{-}[rr]^{r-1}&& *+[o][F]{}   \ar@{-}[dd]_0  \\
        &&&&&&&&\\
        *+[o][F]{}\ar@{-}[rr]_1&&*+[o][F]{}\ar@{-}[rr]_2&&*+[o][F]{}\ar@{.}[rr]&&*+[o][F]{} \ar@{-}[rr]_{r-2}&& *+[o][F]{} \ar@{-}[rr]_{r-1}&& *+[o][F]{} } 
        \\
        \hline
          \end{array}\] 
    \caption{  $\mathbf{k = 2}$; Proposition~\ref{K=C2}.}
      \label{kc2}
    \end{table}
    \end{tiny}

          \begin{tiny}
    \begin{table}[htbp!]
    \[\begin{array}{|cc|}
     \hline
&\mbox{ Case: }  |R\cup C|=1;\;  \langle L\rangle \not\cong \Sym_{n/2}; \;{\rm Ker}(f) \not\cong \Cyc_2.\\
      \hline
      
   & \textnormal{$x$ even and $n/2$ odd}\\ 
  \hline
(25) &$\xymatrix@-1pc{
      *+[o][F]{}\ar@{-}[rr]^1\ar@{=}[dd]_0^{x+1}&&*+[o][F]{}\ar@{=}[dd]_0^{x+1}\ar@{-}[rr]^2&&*+[o][F]{}\ar@{=}[dd]_0^{x+1}\ar@{.}[rr] &&*+[o][F]{}\ar@{-}[rr]^{x}\ar@{=}[dd]_0^{x+1} &&*+[o][F]{}\ar@{-}[rr]^{x+1}\ar@{-}[dd]_0 &&*+[o][F]{}\ar@{-}[dd]_0\ar@{-}[rr]^{x+2}&&*+[o][F]{}\ar@{-}[dd]_0 \ar@{.}[rr] && *+[o][F]{} \ar@{-}[dd]_0\ar@{-}[rr]^{r-2}&& *+[o][F]{}   \ar@{-}[dd]_0 \ar@{-}[rr]^{r-1} && *+[o][F]{}   \ar@{-}[dd]_0 \\
      &&&&&&&&\\
      *+[o][F]{}\ar@{-}[rr]_1&&*+[o][F]{}\ar@{-}[rr]_2&&*+[o][F]{}\ar@{.}[rr]&&*+[o][F]{}\ar@{-}[rr]_{x}&&*+[o][F]{}\ar@{-}[rr]_{x+1}&&*+[o][F]{}\ar@{-}[rr]_{x+2} &&*+[o][F]{} \ar@{.}[rr] &&*+[o][F]{}\ar@{-}[rr]_{r-2}&& *+[o][F]{} \ar@{-}[rr]_{r-1} && *+[o][F]{} }$
      \\
           (26)  & $\xymatrix@-1pc{
      *+[o][F]{}\ar@{-}[rr]^1\ar@{-}[dd]_{I_{1,x+1}}&&*+[o][F]{}\ar@{-}[dd]_{I_{1,2,x+1}}\ar@{-}[rr]^2&&*+[o][F]{}\ar@{-}[dd]_{I_{2,3,x+1}}\ar@{.}[rrr] &&&*+[o][F]{}\ar@{-}[rr]^{x}\ar@{-}[dd]_{I_{x-1,x,x+1}}&&*+[o][F]{}\ar@{-}[dd]_{I_{x,x+1}}\ar@{-}[rr]^{x+1}&&*+[o][F]{}\ar@{-}[dd]_{I_{x+1,x+2}}\ar@{-}[rr]^{x+2}&&*+[o][F]{}\ar@{-}[dd]_{I_{x+2,x+3}} \ar@{.}[rr] && *+[o][F]{} \ar@{-}[dd]_{I_{r-3,r-2}}\ar@{-}[rr]^{r-2}&& *+[o][F]{}   \ar@{-}[dd]_{I_{r-2,r-1}} \ar@{-}[rr]^{r-1} && *+[o][F]{}   \ar@{-}[dd]_{I_{r-1}} \\
      &&&&&&&&\\
      *+[o][F]{}\ar@{-}[rr]_1&&*+[o][F]{}\ar@{-}[rr]_2&&*+[o][F]{}\ar@{.}[rrr]&&&*+[o][F]{}\ar@{-}[rr]_{x}&&*+[o][F]{}\ar@{-}[rr]_{x+1}&&*+[o][F]{}\ar@{-}[rr]_{x+2} &&*+[o][F]{} \ar@{.}[rr] &&*+[o][F]{}\ar@{-}[rr]_{r-2}&& *+[o][F]{} \ar@{-}[rr]_{r-1} && *+[o][F]{} }$
      \\
      \hline
          & \textnormal{$x$ odd and $n/2$ odd}\\ 
       \hline    
           (27)   &  $\xymatrix@-1pc{
      *+[o][F]{}\ar@{-}[rr]^1\ar@{-}[dd]_0&&*+[o][F]{}\ar@{-}[dd]_0\ar@{-}[rr]^2&&*+[o][F]{}\ar@{-}[dd]_0\ar@{.}[rr] &&*+[o][F]{}\ar@{-}[rr]^{x}\ar@{-}[dd]_0 &&*+[o][F]{}\ar@{-}[rr]^{x+1}\ar@{-}[dd]_0 &&*+[o][F]{}\ar@{-}[dd]_0\ar@{-}[rr]^{x+2}&&*+[o][F]{}\ar@{=}[dd]_0^{x+1} \ar@{.}[rr] && *+[o][F]{} \ar@{=}[dd]^{x+1}_0\ar@{-}[rr]^{r-2}&& *+[o][F]{}   \ar@{=}[dd]_0^{x+1} \ar@{-}[rr]^{r-1} && *+[o][F]{}   \ar@{=}[dd]_0^{x+1} \\
      &&&&&&&&\\
      *+[o][F]{}\ar@{-}[rr]_1&&*+[o][F]{}\ar@{-}[rr]_2&&*+[o][F]{}\ar@{.}[rr]&&*+[o][F]{}\ar@{-}[rr]_{x}&&*+[o][F]{}\ar@{-}[rr]_{x+1}&&*+[o][F]{}\ar@{-}[rr]_{x+2} &&*+[o][F]{} \ar@{.}[rr] &&*+[o][F]{}\ar@{-}[rr]_{r-2}&& *+[o][F]{} \ar@{-}[rr]_{r-1} && *+[o][F]{} }$
      \\

       (28)    &  $\xymatrix@-1pc{
        *+[o][F]{}\ar@{-}[rr]^1\ar@{-}[dd]_{I_{1}}&&*+[o][F]{}\ar@{-}[dd]_{I_{1,2}}\ar@{-}[rr]^2&&*+[o][F]{}\ar@{-}[dd]_{I_{2,3}}\ar@{.}[rr] &&*+[o][F]{}\ar@{-}[rr]^{x}\ar@{-}[dd]_{I_{x-1,x}}&&*+[o][F]{}\ar@{-}[dd]_{I_{x,x+1}}\ar@{-}[rr]^{x+1}&&*+[o][F]{}\ar@{-}[dd]_{I_{x+1,x+2}}\ar@{-}[rrr]^{x+2}&&&*+[o][F]{}\ar@{-}[dd]_{I_{x+1,x+2,x+3}} \ar@{.}[rrr] &&& *+[o][F]{} \ar@{-}[dd]_{I_{x+1,r-3,r-2}}\ar@{-}[rrr]^{r-2}&&& *+[o][F]{}   \ar@{-}[dd]_{I_{x+1,r-2,r-1}} \ar@{-}[rr]^{r-1} && *+[o][F]{}   \ar@{-}[dd]_{I_{x+1,r-1}} \\
        &&&&&&&&\\
        *+[o][F]{}\ar@{-}[rr]_1&&*+[o][F]{}\ar@{-}[rr]_2&&*+[o][F]{}\ar@{.}[rr]&&*+[o][F]{}\ar@{-}[rr]_{x}&&*+[o][F]{}\ar@{-}[rr]_{x+1}&&*+[o][F]{}\ar@{-}[rrr]_{x+2} &&& *+[o][F]{} \ar@{.}[rrr] &&&*+[o][F]{}\ar@{-}[rrr]_{r-2}&&& *+[o][F]{} \ar@{-}[rr]_{r-1} && *+[o][F]{} }$
      \\
               \hline
          \end{array}\] 
  \caption{$\mathbf{k = 2}$;  Proposition~\ref{k=2notK=2}.}
    \label{KnotC2}
  \end{table}
\end{tiny}

\begin{tiny}
  \begin{table}[htbp!]
  \[\begin{array}{|cc|}
  \hline
  (1)\; \xymatrix@-1.5pc{
*+[o][F]{}\ar@{-}[dd]_0\ar@{-}[rr]^1&&*+[o][F]{}\ar@{-}[dd]_0\ar@{-}[rr]^2&&*+[o][F]{}\ar@{-}[dd]_0\ar@{.}[rr] &&*+[o][F]{} \ar@{-}[dd]^0\ar@{-}[rr]^{r-1}&& *+[o][F]{}   \ar@{-}[dd]^0 \\
&&&&&&&&\\
*+[o][F]{}\ar@{-}[rr]_1&&*+[o][F]{}\ar@{-}[rr]_2&&*+[o][F]{}\ar@{.}[rr]&&*+[o][F]{} \ar@{-}[rr]_{r-1}&& *+[o][F]{} }
&(2)\;\xymatrix@-1.5pc{
*+[o][F]{}\ar@{-}[dd]_{I_1}\ar@{-}[rr]^1&&*+[o][F]{}\ar@{-}[dd]_{I_{1,2}}\ar@{-}[rr]^2&&*+[o][F]{}\ar@{-}[dd]_{I_{2,3}}\ar@{.}[rr]  &&*+[o][F]{} \ar@{-}[dd]^{I_{r-3,r-2}}\ar@{-}[rrr]^{r-2}&&&*+[o][F]{} \ar@{-}[dd]^{I_{r-2,r-1}}\ar@{-}[rrr]^{r-1}&&& *+[o][F]{}   \ar@{-}[dd]^{I_{r-1}} \\
&&&&&&&&\\
*+[o][F]{}\ar@{-}[rr]_1&&*+[o][F]{}\ar@{-}[rr]_2&&*+[o][F]{}\ar@{.}[rr]&&*+[o][F]{} \ar@{-}[rrr]_{r-2}&&&*+[o][F]{} \ar@{-}[rrr]_{r-1}&&& *+[o][F]{} }\\
  (3)\; \xymatrix@-1.5pc{
  *+[o][F]{}\ar@{-}[dd]_0&&*+[o][F]{}\ar@{-}[dd]_0\ar@{-}[rr]^2&&*+[o][F]{}\ar@{-}[dd]_0\ar@{.}[rr] &&*+[o][F]{} \ar@{-}[dd]^0\ar@{-}[rr]^{r-1}&& *+[o][F]{}   \ar@{-}[dd]^0 \\
  &&&&&&&&\\
  *+[o][F]{}\ar@{-}[rr]_1&&*+[o][F]{}\ar@{-}[rr]_2&&*+[o][F]{}\ar@{.}[rr]&&*+[o][F]{} \ar@{-}[rr]_{r-1}&& *+[o][F]{} }
  &(4)\;  \xymatrix@-1.5pc{
  *+[o][F]{}\ar@{-}[dd]_i\ar@{-}[rr]^0&&*+[o][F]{}\ar@{-}[dd]_i\ar@{.}[rr]\ar@{-}[dd]_i&&*+[o][F]{}\ar@{-}[rr]^{i-2} \ar@{-}[dd]_i &&*+[o][F]{}\ar@{-}[dd]_i\ar@{-}[rr]^{i-1}&&*+[o][F]{}\ar@{-}[dd]_i&&*+[o][F]{}\ar@{-}[dd]_i\ar@{-}[rr]^{i+2}&&*+[o][F]{}\ar@{-}[dd]_i\ar@{.}[rr] &&*+[o][F]{} \ar@{-}[dd]^i\ar@{-}[rr]^{r-1}&& *+[o][F]{}   \ar@{-}[dd]^i\\
  &&&&&&&&&&&&&&&&\\
  *+[o][F]{}\ar@{-}[rr]_0&&*+[o][F]{}\ar@{.}[rr]&&*+[o][F]{}\ar@{-}[rr]_{i-2} &&*+[o][F]{}&&*+[o][F]{}\ar@{-}[rr]_{i+1} &&*+[o][F]{}\ar@{-}[rr]_{i+2}&&*+[o][F]{}\ar@{.}[rr]&&*+[o][F]{} \ar@{-}[rr]_{r-1}&& *+[o][F]{} }
  \\
  \hline
  \end{array}\] 
  \caption{$\mathbf{m = 2}$.}
  \label{m=2table}
  \end{table}
  \end{tiny}

\section{Acknowledgements}

The author Maria Elisa Fernandes was supported by  CIDMA under the
Portuguese Foundation for Science and Technology 
(FCT, https://ror.org/00snfqn58)  
Multi-Annual Financing Program for R$\&$D Units,
grants UID/4106/2025 and UID/PRR/4106/2025.
The author Claudio Alexandre Piedade was partially supported by CMUP, member of LASI, which is financed by national funds through FCT – Fundação para a Ciência e a Tecnologia, I.P., under the projects with reference UIDB/00144/2025 and UIDP/00144/2025.


\bibliographystyle{amsplain}

\end{document}